%% file: aschobers.tex
\documentclass[10pt, reqno]{article}
\input{preamble.tex}

\usetikzlibrary{patterns,snakes}

\usepackage{subfiles}

\usepackage{etoolbox}
\tracingpatches
\makeatletter
\patchcmd{\@setref}{\bfseries ??}{\bfseries\color{red} OWE A COFFEE/BEER}{}{}
\makeatother

\newcommand{\ba}{\textbf{a}}
\newcommand{\bone}{\textbf{1}}
\newcommand{\ula}{\underline{\lambda}}

\title{Perverse schobers of Coxeter type $\AA$}

\author{Tobias Dyckerhoff\footnote{Universität Hamburg, Fachbereich Mathematik, Bundesstraße 55, 20146 Hamburg, Germany
	email: {\tt tobias.dyckerhoff@uni-hamburg.de}}, Paul Wedrich\footnote{Universität Hamburg, Fachbereich Mathematik, Bundesstraße 55, 20146 Hamburg, Germany
	email: {\tt paul.wedrich@uni-hamburg.de}} }

\begin{document}

\maketitle

\begin{abstract} 
    We define the concept of an $\AA_n$-schober as a categorification of
    classification data for perverse sheaves on $\Sym^{n+1}(\CC)$ due to
    Kapranov-Schechtman. We show that any $\AA_n$-schober gives rise to a
    categorical action of the Artin braid group $\Br_{n+1}$ and demonstrate how
    this recovers familiar examples of such actions arising from  Seidel-Thomas
    $\AA_n$-configurations of spherical objects in categorical Picard-Lefschetz
    theory and Rickard complexes in link homology theory. As a key example, we
    use singular Soergel bimodules to construct a factorizing family of
    $\AA_n$-schobers which we refer to as Soergel schobers. We expect such
    families to give rise to a categorical analog of a graded bialgebra valued
    in a suitably defined freely generated braided monoidal
    $(\infty,2)$-category. 



\end{abstract}

\tableofcontents

\section{Introduction}

Spherical objects have been introduced by Seidel--Thomas \cite{ST01} where, in
the context of a categorical approach to Picard-Lefschetz theory, they were
used to construct celebrated braid group actions. As a natural generalization,
spherical functors, introduced by Anno-Logvinenko \cite{AL17}, seem to play an
increasingly important role in various topics in the vicinity of homological
mirror symmetry. The striking proposal of Kapranov-Schechtman \cite{1411.2772}
to interpret spherical functors as categorical analogs of perverse sheaves,
termed perverse schobers, raises expectations of a rich theory, built in
analogy to the successful classical theory of perverse sheaves, with potential
applications to higher categorical braid group actions. Various glimpses of
this hypothetical theory have emerged over the time, but it seems, we are still
rather far away from a systematic approach. 

In a parallel development, link homology theories originating in Khovanov's
categorification of the Jones polynomial~\cite{Kho} and Rouquier's categorical
braid group action \cite{0409593} have revolutionized low-dimensional topology
and motivated much of modern representation theory, yet remain poorly understood on a deep
conceptual level.

We take inspiration from link homology theory and the higher
representation theory of categorified quantum groups to develop a new series of
examples of perverse schobers, associated to the family $\Sym^{n+1}(\CC)$, $n \ge
0$, of symmetric products of the complex plane, stratified by the discriminant. 

\begin{definitionsintro} In this work, we give rigorous definitions for
\begin{itemize}[itemsep=-2pt]
    \item the concept of a perverse schober on $\Sym^{n+1}(\CC)$, called {\em $\AA_n$-schober}, in Definition~\ref{defi:anschober}
    \item the property of an $\AA_n$-schober to be \emph{framed}, in Definition~\ref{defi:highercotwist},
    \item the concept of a \emph{factorizing family} of $\AA_n$-schobers, $n\in \N$, in Definition~\ref{defi:multschober}.
\end{itemize}
\end{definitionsintro} 
As a first sanity check for our proposed definitions, we show that, from any
$\AA_n$-schober, we can construct a canonical categorical action of the braid
group $\Br_{n+1}$ on $n+1$ strands. This result generalizes (and conceptualizes)
the braid group actions known since the very origins of the theory: any
$\AA_n$-collection of spherical objects defines an $\AA_n$-schober and the
resulting braid group action recovers the one discovered by Seidel and Thomas.
However, from the schober perspective, $\AA_n$--collections of spherical objects
are a rather degenerate example of a perverse schober which, in terms of
vanishing cycle data, is zero on the singular locus of the discriminant. 

In \S
\ref{sec:soergelschober}, we introduce our main example of an $\AA_n$-schober
which exhausts the full capacity of the concept, in the sense that it has
nontrivial support on all strata of the discriminant stratification. %

\begin{theoremintro} Type $\AA$ singular Soergel bimodules define a factorizing family of framed $\AA_n$-schobers.
\end{theoremintro}

We refer to these $\AA_n$-schobers as {\em Soergel schobers}. They are central
for this work as, on the one hand, the main definitions mentioned above are
modelled on these prototypical examples and, on the other hand, they shed light
on the conceptual foundations of the Rouquier braiding.

\subsection{Approach via classifying data for perverse sheaves}


Our modelling of the concept of $\AA_n$-schobers is based on the insights
offered by a series of generalizations of the results of \cite{gm:rebroussement}
due to Kapranov-Schechtman (cf. \cite{KS:hyperplane,MR4367797}), most
specifically, on the perspective provided in \cite{2102.13321}. In this latter
work, Kapranov-Schechtman provide a remarkable interpretation of the
classification data for perverse sheaves on $\Sym^{n+1}(\CC)$ in terms of graded
bialgebra relations in a suitably defined freely generated braided monoidal
category. We now recall this interpretation, as it is central to our
categorification approach.

\smallskip

A bialgebra in a braided monoidal category $\VS$ is an object $X$ together with
compatible (co)unital and (co)associative multiplication and comultiplication morphisms
$m\colon X\otimes X \to X$ and $\Delta\colon X \to X \otimes X$. The
compatibility condition requires $\Delta$ to intertwine the algebra structures on $X$ and $X\otimes X$, the latter of which
uses the braiding on $\VS$; in terms of string diagrams (read right to left):
\[
	\begin{tikzpicture}[scale=.4,smallnodes,anchorbase,rotate=270]
		\draw[very thick] (1,-1) to [out=150,in=330] (0,1) to (0,2) ; 
		\draw[line width=5pt,color=white] (0,-2) to (0,-1) to [out=30,in=210] (1,1);
		\draw[very thick] (0,-2)  to (0,-1) to [out=30,in=210] (1,1);
		\draw[very thick] (1,1) to (1,2) ;
		\draw[very thick] (1,-2) to (1,-1); 
		\draw[very thick] (1,-1) to [out=30,in=330]  (1,1); 
		\draw[very thick] (0,-1) to [out=150,in=210] (0,1); 
		\end{tikzpicture}
		=
\begin{tikzpicture}[scale=.4,smallnodes,anchorbase,rotate=270]
	\draw[very thick] (.5,.5) to [out=150,in=270] (0,1) to (0,2); 
	\draw[very thick] (0,-2)  to (0,-1) to [out=90,in=210] (.5,-.5) to (.5,.5);
	\draw[very thick] (.5,.5) to [out=30,in=270]  (1,1) to (1,2) ;
	\draw[very thick] (1,-2) to (1,-1) to [out=90,in=330] (.5,-.5); 
	\end{tikzpicture}
\]
For an $\N$-graded bialgebra in an additive braided category $\VS$,
	we additionally assume that there is a decomposition $X=\bigoplus_{n\geq 0}
	X_n$ with $X_0=\bone$ the monoidal unit, that the multiplication and
	comultiplication are homogeneous, and that the components involving $X_0$
	are identities. The compatibility of $m$ and $\Delta$ can then be phrased
	entirely in terms of homogeneous components

	\begin{equation}
        \label{eq:bialgebra}
    \sum_s
	\begin{tikzpicture}[scale=.4,smallnodes,anchorbase,rotate=270]
	\draw[very thick] (1,-1) to [out=150,in=330] (0,1) to (0,2) node[right=-2pt]{$b$}; 
	\draw[line width=5pt,color=white] (0,-2) to (0,-1) to [out=30,in=210] (1,1);
	\draw[very thick] (0,-2) node[left=-2pt]{$d$} to (0,-1) to [out=30,in=210] (1,1);
	\draw[very thick] (1,1) to (1,2) node[right=-2pt]{$a$};
	\draw[very thick] (1,-2) node[left=-2pt]{$c$} to (1,-1); 
	\draw[very thick] (1,-1) to [out=30,in=330]node[below=1pt]{$\phantom{t}$}  (1,1); 
	\draw[very thick] (0,-1) to [out=150,in=210]node[above=-1pt]{$s$} (0,1); 
	\end{tikzpicture}
	=
	\begin{tikzpicture}[scale=.4,smallnodes,anchorbase,rotate=270]
		\draw[very thick] (.5,.5) to [out=150,in=270] (0,1) to (0,2) node[right=-2pt]{$b$}; 
		\draw[very thick] (0,-2) node[left=-2pt]{$d$} to (0,-1) to [out=90,in=210] (.5,-.5) to (.5,.5);
		\draw[very thick] (.5,.5) to [out=30,in=270]  (1,1) to (1,2) node[right=-2pt]{$a$};
		\draw[very thick] (1,-2) node[left=-2pt]{$c$} to (1,-1) to [out=90,in=330] (.5,-.5); 
		\end{tikzpicture}
	\end{equation}
	where we must have $a+b=c+d$. By considering homogeneous components, graded
	bialgebras also make sense as families of objects $(X_n)_{n\geq 0}$ in a
	braided category $\VS$ that is merely required to be semiadditive, so that
	$\bigoplus_{n\geq 0} X_n$ need not exist as an object of $\VS$.
	Specifically, graded bialgebras in such $\VS$ are modelled as braided
	monoidal functors $\BS \to \VS$, with $\BS$ denoting the braided monoidal
	(semiadditive) category freely generated by the \emph{universal graded bialgebra} $(\ba_n)_{n\geq 0}$ with $\ba_0=\bone$. 

    Explicitly, $\BS$ has objects
    $\ba_{\ula}:=\ba_{\lambda_1}\otimes \cdots \otimes
    \ba_{\lambda_r}$ where $r\geq 0$ and
    $\ula=(\lambda_1,\dots,\lambda_r)\in \N^r$. The morphisms are
    generated by the components of the bialgebra (co)multiplication
    \[
    \mu^{p,q}_{p+q}\colon \ba_p\otimes \ba_q \to \ba_{p+q},\quad 	
    \Delta^{p+q}_{p,q}\colon \ba_{p+q} \to \ba_p\otimes \ba_q, \quad \mu^{p,0}_p=\mu^{0,p}_p=\Delta^p_{p,0}=\Delta^p_{0,p}=\id_{\ba_p}.
    \] 
    as well as the braid generators $\sigma_{p,q}\colon \ba_p\otimes \ba_q\to
    \ba_q\otimes \ba_p$, modulo the relations expressing the axioms of a braided
    (strict) monoidal category and a graded bialgebra.

    For $n\geq 0$ denote by $\BS_n$ the full subcategory of $\BS$ on objects
    $\ba_{\ula}$ of degree $n$, i.e. indexed by those
    $\ula=(\lambda_1,\dots,\lambda_r)$ for which $\sum_{i=1}^r
    \lambda_i=n$. 
    
    \begin{theoremintro}[{\cite[Theorem 1.3]{2102.13321}}]\label{thm:KS} For any abelian category $\AS$ (not necessarily monoidal) there is an equivalence of categories
        \[\mathrm{Perv}(\mathrm{Sym}^{n+1}(\CC),\AS)\simeq \mathrm{Fun(\BS_n,\AS)}\] where
        $\mathrm{Perv}(\mathrm{Sym}^{n+1}(\CC),\AS)$ denotes the category of $\AS$-valued
        perverse sheaves on $\mathrm{Sym}^{n+1}(\CC)$ which are constructible with respect to
    the discriminant stratification.
    \end{theoremintro}

    \begin{remarkintro}
Important for our approach to $\AA_n$-schobers as categorification of perverse
sheaves on $\mathrm{Sym}^{n+1}(\CC)$ are two suggestions from Theorem~\ref{thm:KS}.
First, it may be easier to define a consistent family perverse schobers on all
symmetric products $\Sym^{n+1}(\CC)$ for $n\in \N$ than any individual one. Second,
such a family should resemble a categorified graded bialgebra in a braided
monoidal $2$-category.
    \end{remarkintro}

\subsection{A categorified graded bialgebra from singular Soergel bimodules}

The celebrated, motivating example of a perverse schober is captured by the data
of a spherical functor and type $\AA$-configurations of such functors give rise
to categorical braid group actions as mentioned above. An important example in
link homology theory is given by the \emph{Rouquier complexes} of \emph{Soergel
bimodules} for $S_n$ \cite{0409593}, which yield categorical actions of the
$n$-strand Artin braid group $\Br_n$. It is an intriguing question whether
Rouquier complexes are part of the structure of a perverse schober on
$\mathrm{Sym}^n(\CC)$. Indeed, \cite[Theorem 1.3]{2102.13321} as well as the
discussion in \cite{MR4367797} suggest that Rouquier complexes of Soergel
bimodules may not be enough to directly express the structure of a perverse schober, and
that parabolic induction and restriction should play an important role. This is
where \emph{singular} Soergel bimodules and \emph{Rickard complexes} come into play. In the following we recall these concepts.

Soergel bimodules \cite{Soergel} of type $S_n$ categorify the Hecke algebra of $S_n$, a
quotient of the group algebra of the braid group $\Br_n$.
\[
H_n:=\Z[q,q^{-1}][\Br_n]/\langle \sigma_i-\sigma_i^{-1}=(q-q^{-1})e	\rangle
\]
For every braid, there exists a corresponding Rouquier complex \cite{0409593} of
Soergel bimodules, whose class in K-theory is the Hecke algebra element
represented by the braid. Rouquier complexes respect the braid relations up to
canonical homotopy equivalence and are compatible with parabolic induction. 

The main result of \cite{LMGRSW} shows that Rouquier complexes of Soergel
bimodules for all $S_n$, $n\geq 0$ give rise to a braided monoidal
$(\infty,2)$-category. The braiding in this higher category is the source of
most representation-theoretic link homology theories, including Khovanov
homology \cite{Kho}, the Khovanov--Rozansky link homologies for $\glN$ \cite{KR}, and triply-graded
link homology \cite{MR2421131, MR2339573}, see \cite{arxiv.2207.05139} for a survey. After truncation to
homotopy categories, the braiding furthermore serves as input datum for a family
of topological field theories \cite{2019arXiv190712194M}, which are sensitive to
smooth structure in dimension four \cite{ren2024khovanov}. 

Singular Soergel bimodules \cite{Williamson-thesis} categorify partial
idempotent completions of Hecke algebras known as Schur algebroids. The
idempotents adjoined correspond to (anti-)symmetrizations with respect to
parabolic subgroups and are realized by (rescalings of) Kazdhan--Lusztig basis
elements corresponding to their longest elements. Rickard complexes
\cite{MR2373155} are the generalization of Rouquier complexes to the setting of
singular Soergel bimodules. In the context of link homology theories, they are
used to construct categorifications of \emph{colored} link polynomials
\cite{MR2746676, MR3687104, Wed3} i.e. with higher exterior/symmetric powers of
natural representations labelling the link components. 

We expect that Rickard complexes of singular Soergel bimodules will give rise to
a braided monoidal $(\infty,2)$-category generalizing that of \cite{LMGRSW}. As
key step towards the construction of the Soergel schobers, we exhibit a
categorified graded bialgebra structure with respect to the Rickard braiding.

\begin{theoremintro}[Categorified bialgebra structure, Theorem~\ref{thm:main}]
	Let $a,b,c,d$ be non-negative integers, $a+b=c+d$. Then
	there exists a homotopy equivalence of twisted complexes of singular Soergel
	bimodules
	\begin{equation*}
	\tw_{D}\left(\bigoplus_{s}
	\qdeg^{-s(s+a-d)}
	\left\llbracket
	\begin{tikzpicture}[scale=.4,smallnodes,anchorbase,rotate=270]
	\draw[very thick] (1,-1) to [out=150,in=330] (0,1) to (0,2) node[right=-2pt]{$b$}; 
	\draw[line width=5pt,color=white] (0,-2) to (0,-1) to [out=30,in=210] (1,1);
	\draw[very thick] (0,-2) node[left=-2pt]{$d$} to (0,-1) to [out=30,in=210] (1,1);
	\draw[very thick] (1,1) to (1,2) node[right=-2pt]{$a$};
	\draw[very thick] (1,-2) node[left=-2pt]{$c$} to (1,-1); 
	\draw[very thick] (1,-1) to [out=30,in=330]  (1,1); 
	\draw[very thick] (0,-1) to [out=150,in=210]node[above=-1pt]{$s$} (0,1); 
	\end{tikzpicture}
	\right\rrbracket
	\right) 
	\simeq 
	\left\llbracket
	\begin{tikzpicture}[scale=.4,smallnodes,anchorbase,rotate=270]
		\draw[very thick] (.5,.5) to [out=150,in=270] (0,1) to (0,2) node[right=-2pt]{$b$}; 
		\draw[very thick] (0,-2) node[left=-2pt]{$d$} to (0,-1) to [out=90,in=210] (.5,-.5) to (.5,.5);
		\draw[very thick] (.5,.5) to [out=30,in=270]  (1,1) to (1,2) node[right=-2pt]{$a$};
		\draw[very thick] (1,-2) node[left=-2pt]{$c$} to (1,-1) to [out=90,in=330] (.5,-.5); 
		\end{tikzpicture}
		\right\rrbracket
	\end{equation*}
	where the twist $D$ is strictly decreasing in $s$. 
\end{theoremintro}
\noindent For definitions and conventions for singular Soergel bimodules see
	Section~\ref{sec:soergelschober}. There we also give a direct proof of the
	categorified bialgebra structure using technology developed in
	\cite{HRW1,HRW2} to describe generalized skein relations for Rickard
	complexes of singular Soergel bimodules. 

    \begin{remarkintro}
By decategorification, we obtain as a corollary that the type $\AA$ Schur
        algebroids constitute a braided monoidal category generated by a graded
        bialgebra object, and thus recover the $q$-Schur category of \cite[Theorem
        1]{brundan2024qschurcategorypolynomialtilting}. Further quotients
        describe categories of polynomial tilting modules for type $\AA$ quantum
        groups \cite{CKM,TVW,latifi2021minimalpresentationsglnwebcategories} and, via \cite[Theorem 1.3]{2102.13321}, determine
        representation category-value families of perverse sheaves on
        $\mathrm{Sym}^n(\CC)$ for all $n\geq 0$. 
    \end{remarkintro}

\subsection{The higher algebra of parabolic induction}

The essential common feature of the graded bialgebras of \cite[Theorem
1.3]{2102.13321} and singular Soergel bimodules are \emph{cubes of
compositions}. 

A \emph{composition} (or \emph{unordered partition}) of a natural number $n\geq
0$ is an ordered tuple\footnote{We allow the empty tuple as unique composition of $0$.} of positive natural numbers $(n_1, n_2, \ldots, n_k)$
summing up to $n$. The set $\Comp(n)$ of all compositions of $n$ is partially
ordered by refinement $(n)<\cdots< (1,\dots, 1)$ and the operation of
concatenating compositions endows the disjoint union $\Comp := \amalg_{n \in
\NN} \Comp(n)$ of all composition posets with a strict monoidal structure. As a
monoidal category $\Comp$ is generated by elementary refinements $(a+b)<(a,b)$,
while the opposite poset $\Comp^\op$ is generated by the dual coarsenings:
\[\begin{tikzpicture}[scale =.75, smallnodes,rotate=270,anchorbase]
    \draw[very thick] (0.25,0) node[left,xshift=2pt]{$b$} to [out=90,in=210] (.5,.75);
    \draw[very thick] (.75,0) node[left,xshift=2pt]{$a$} to [out=90,in=330] (.5,.75);
    \draw[very thick] (.5,.75) to (.5,1.5) node[right,xshift=-2pt]{$a{+}b$};
\end{tikzpicture}
\qquad, \qquad 
\begin{tikzpicture}[scale =.75, smallnodes, rotate=90,anchorbase]
    \draw[very thick] (0.25,0) node[right,xshift=-2pt]{$a$} to [out=90,in=210] (.5,.75);
    \draw[very thick] (.75,0) node[right,xshift=-2pt]{$b$} to [out=90,in=330] (.5,.75);
    \draw[very thick] (.5,.75) to (.5,1.5) node[left,xshift=2pt]{$a{+}b$};
\end{tikzpicture}
\] In the context of graded bialgebras, these elementary refinements and coarsenings
correspond to components $\Delta^{a+b}_{ab}$ and $\mu^{ab}_{a+b}$ of the
(co)multiplication, while for singular Soergel bimodules, they represent
induction and restriction bimodules between partially symmetric polynomial rings
of geometric origin, namely equivariant cohomology rings of partial flag
varieties. 

On a combinatorial level, the composition poset $\Comp(n)$ can be identified
with the hypercube $[0,1]^{n-1}$, in which vertices correspond to compositions
and edges to elementary refinements. A key observation is that for Soergel
bimodules the restriction bimodules associated to elementary coarsenings are
automatically incorporated, namely as right adjoints of the induction bimodules
for elementary refinements. 

One is, thus, led to a typical problem of higher representation theory, the
study of the (monoidal) 2-categorical structure generated by the units and
counits for the adjoint pairs of induction and restriction bimodules associated
to the edges of the composition cube. In our prototypical case, this is, by
definition, the (monoidal) 2-category of singular Bott--Samelson bimodules
\cite{Williamson-thesis,ESW,stroppel2024braidingtypesoergelbimodules}, which can
be described using the diagrammatic language of foams \cite{HRW1} as we recall
in Remark~\ref{rem:graphicallanguage}, or via Schur quotients of type A
categorified quantum groups \cite{KL3, KLMS} by means of categorified skew Howe
duality, as explained in \cite{QR}.

\subsection{On the schober definition}

Our notion of $\AA_n$-schober is based on an abstraction from the
prototypical example of singular Soergel bimodules. The key insight is that the
graded bialgebra relations \eqref{eq:bialgebra} admit a very natural stable
categorification under fairly general assumptions. Underlying an $\AA_n$-schober
we, thus, only need the following.
\begin{itemize}[leftmargin=8mm]
\item[\textbf{Data:}] A coherent diagram $\X: \Comp(n+1) \to \Stab$ of exact
functors of stable $\infty$-categories.
\end{itemize}
These functors are required to satisfy five types of properties listed in
Definition~\ref{defi:anschober}: 
\begin{itemize}[leftmargin=18mm]
    \item[\textbf{Properties:}] Adjunctability, Recursiveness, Far-commutativity, Twist invertibility, Defect vanishing.
\end{itemize}
Adjunctability requries that all functors in $\X$ admit right adjoints. The
remaining properties impose certain higher categorical constraints among the
various unit and counit transformations witnessing the adjunctions. These are
satisfied in the case of Soergel bimodules as well as in various ad-hoc
definitions of perverse schobers for small Coxeter types, that we discuss in
\S\ref{sec:smallrank}. Imposing an additional requirement on cotwists leads to the concept that we term \emph{framed} schobers, see Definition~\ref{defi:highercotwist}:

\begin{itemize}[leftmargin=72mm]
    \item[\textbf{Additional property of a framed schober:}] Cotwist
    invertibility.
\end{itemize}

In particular, we discuss in detail the following cases in \S\ref{sec:smallrank}:
\begin{itemize}[itemsep=-2pt]
\item An $\AA_1$-schober is simply an exact functor $F: \A \to \B$ of stable
$\infty$-categories, admitting a right adjoint $F^*$, such that the \emph{twist}, the
fiber of the counit $F F^* \to \id_{\B}$, defines an autoequivalence of $\B$.
\item For a framed $\AA_1$-schober, the additional requirement is that \emph{cotwist}, the
cofiber of the counit $\id_{\B} \to F^* F$, defines an autoequivalence of
$\A$. This is nothing but a spherical functor \cite{AL17}.
\item The concept of an $\AA_2$-schober was the starting point of our approach,
which we reached by an ad-hoc categorification of the classical diagrammatic
classification result by Granger-Maisonobe \cite{gm:rebroussement} for perverse
sheaves associated to cuspidal cubics. 
\item The concept of a framed $\AA_2$-schober
has been discovered independently in \cite{AL:skein}, where they are termed
skein--triangulated categorical representations of generalized braid groups (in
the context of ``enhanced triangulated categories''). 
\end{itemize}

At first sight, the data and relations of $\AA_n$-schobers already look
challenging to keep track of in small ranks, both from a technical and
conceptual perspective. However, as it turns out, all needed constraints can be
organized and formulated in terms of so--called higher--dimensional
Beck--Chevalley cubes, which we introduce in \S \ref{subsec:higherbc} and heavily use for the definitions in \S\ref{sec:anschobers}.
\smallskip

The remainder of \S\ref{sec:anschobers} and \S\ref{sec:monoidal} is devoted to a
categorification of two important aspects of \cite[Theorem 1.3]{2102.13321} that
we have previously highlighted. 

On the one hand, it is natural to consider \emph{factorizing families} of
$\AA_n$-schobers for all $n\geq 0$, modelled as certain monoidal functors in
Definition~\ref{defi:multschober}. Monoidality incorporates the Recursiveness
and Far-commutativity properties for all $n$, leading to a simplified
definition.

On the other hand, in Theorem~\ref{thm:schoberbialgebra} we show that any
$\AA_n$-schober furnishes a categorification of the graded bialgebra relations
\eqref{eq:bialgebra} homogeneous of degree $n+1$. This yields a concrete
recursive description of the higher twists in terms of lower order twists, which
we use in \S\ref{sec:Soergelschoberbialgebra} to identify the higher twists for
singular Soergel bimodules with Rickard complexes. In particular, this
illustrates in the example of singular Soergel bimodules, how \emph{the
categorified graded bialgebra generates the braiding}, and sheds light on the
higher-categorical foundations of Rickard complexes. 

As a sanity check in the general case, we show in Corollary~\ref{cor:anbraid}
    that any $\AA_n$-schober furnishes a canonical $\Br_{n+1}$-action up to
    homotopy on the underlying stable $\infty$-category associated to the open
    stratum. A lift to a coherent action is left as
    Conjecture~\ref{conj:anbraidcoherent}.
\smallskip

We speculate that factorizing families of framed $\AA_n$-schobers are the
appropriate categorified analog of graded bialgebras in a stable
categorification of \cite[Theorem 1.3]{2102.13321}. There should be a locally
stable $\EE_2$-monoidal $(\infty,2)$-category that is freely generated by a
factorizing family of framed $\AA_n$-schobers (a ``colored 2PROB''), which then serves as
corepresenting object of framed $\AA_n$-schobers valued in any locally stable
$\EE_2$-monoidal $(\infty,2)$-category. 
\smallskip

Finally we remark that, although our work here only concerns the Coxeter types
$\AA$, many essential aspects of our approach generalize to other finite Coxeter
types, where compositions cubes are replaced by cubes of parabolic subgroup
inclusions, with monoidal structure implemented by parabolic induction.

\subsection*{Acknowledgements}
We would like to thank Matthew Hogancamp, Mikhail Kapranov, Timothy Logvinenko, David Reutter, David Rose, Raphael
Rouquier and Catharina Stroppel for helpful discussions.

\subsection*{Funding}
The authors acknowledge support from the Deutsche
Forschungsgemeinschaft (DFG, German Research Foundation) under Germany's
Excellence Strategy - EXC 2121 ``Quantum Universe'' - 390833306 and the
Collaborative Research Center - SFB 1624 ``Higher structures, moduli spaces and
integrability'' - 506632645.

\section{Schobers of type \texorpdfstring{$\AA_1$}{A1}, \texorpdfstring{$\AA_1 \times \AA_1$}{A1 x A1} and \texorpdfstring{$\AA_2$}{A2}}
\label{sec:smallrank}

We discuss some perverse schober conditions in low dimensions, explaining how
they decategorify to known classification data for perverse sheaves. We will
freely use the language of $\infty$-categories, in particular, stable
$\infty$-categories using \cite{lurie:htt} and \cite{lurie:ha} as standard
references. 

\subsection{\texorpdfstring{$\AA_1$}{A1}-schobers}

\begin{defi}
    \label{defi:a1schober}
    An exact functor
    \[
        F: \A \to \B
    \]
    of stable $\infty$-categories is called an {\em $\AA_1$-schober} if 
    \begin{enumerate}[label=(A1.\arabic *)]
        \item $F$ admits a right adjoint, denoted by $F^*$, 
        \item the fiber $T$ of the counit 
            \[
                F F^* \to \id_{\B}
            \]
            defines an autoequivalence of $\B$. 
    \end{enumerate}
    We refer to the autoequivalence $T$ as the {\em twist functor} associated to $F$.
\end{defi}

Given an $\AA_1$-schober, we thus obtain an exact sequence
\[
\begin{tikzcd}
    T \ar{r}\ar{d} & F F^* \ar{d}\\
    0 \ar{r} & \id
\end{tikzcd}
\]
in the stable $\infty$-category $\Fun(\B,\B)$ exhibiting $FF^*$ as an extension
of $\id$ and $T$. This extension splits upon passage to Grothendieck groups, so
that we obtain abelian groups $A = K_0(\A)$ and $B = K_0(\B)$ along with
maps $f,f^*,t$ satisfying $f f^* = \id + t$. Such a datum is classically known as
to classify a perverse sheaf on $\CC$ with respect to the stratification $\CC =
\{0\} \cup \CC^*$. 

\begin{rem}
    \label{rem:spherical}
    Note that our notion of an $\AA_1$-schober is weaker than the concept of a
    spherical functor, since we do not require the fiber of the {\em unit} of
    the adjunction $F \dashv F^*$, the so-called {\em cotwist}, to be an
    equivalence. On the one hand, even this weaker structure along with the
    generalizations introduced in this work, will yield perverse sheaves when
    passing to Grothendieck groups. On the other hand, as we will see later, it
    also makes sense from a structural perspective to interpret the collection
    of invertible higher {\em cotwists} as extra data that may or may not be
    available (cf. \S \ref{subsec:framed}). 
\end{rem}

\subsection{\texorpdfstring{$\AA_1 \times \AA_1$}{A1 x A1}-schobers}

Let 
\begin{equation}
    \label{eq:commdiag}
    \begin{tikzcd}
        \A \ar{r}{I}\ar[swap]{d}{H} & \B \ar{d}{F}\\
        \C \ar{r}{G} & \D
    \end{tikzcd}
\end{equation}
be a coherent square of exact functors of stable
$\infty$-categories such that all functors admit right adjoints, denoted $F^*,
G^*, H^*, I^*$, respectively. We define the {\em Beck--Chevalley map}
\[
    H I^* \to G^* F
\]
to be the natural transformation given as the composite of 
\[
    H I^* \to H I^* F^* F \simeq H  (FI)^* F \simeq H  (GH)^* F  \simeq H H^* G^* F \to G^* F
\]
where the first map is the whiskering of the unit of $F \dashv F^*$
with $H I^*$, the last map is the whiskering of the counit of $H
\dashv H^*$ with $G^* F$, and the intermediate maps are rather
apparent canonical identifications. 

\begin{defi}
    \label{defi:a1a1schober}
    A coherent square as in \eqref{eq:commdiag} is called an {\em
    $\AA_1 \times \AA_1$-schober} if the following conditions hold:
    \begin{enumerate}[label = (A1xA1.\arabic *)]
        \item All functors in \eqref{eq:commdiag} admit right adjoints. 
        \item The functors $F$ and $G$ are $\AA_1$-schobers. 
        \item The Beck--Chevalley map
            \[
                H I^* \to G^* F
            \]
            is an equivalence. 
        \item The Beck--Chevalley map
            \[
                I H^* \to F^* G
            \]
            associated to the transpose of the square \eqref{eq:commdiag} is an equivalence. 
    \end{enumerate}
\end{defi}

\begin{rem}
    \label{rem:twistscommute}
    Given an $\AA_1 \times \AA_1$-schober, denote by $T$ and $S$ the twist
    functors associated to the functors $F$ and $G$, respectively. The
    composite $TS$ is equivalent to the total fiber of the square
    \[
    \begin{tikzcd}
        FF^* GG^* \ar{r}\ar{d} & G G^* \ar{d}\\
        FF^* \ar{r} & \id_{\D}
    \end{tikzcd}
    \]
    which using the Beck--Chevalley map $IH^* \overset{\simeq}{\to} F^* G$ may be identified with the square
    \[
    \begin{tikzcd}
        FI (GH)^* \ar{r}\ar{d} & G G^* \ar{d}\\
        FF^* \ar{r} & \id_{\D}.
    \end{tikzcd}
    \]
    Similarly, the composite $ST$ is given as the total fiber of the square
    \[
    \begin{tikzcd}
        GG^* FF^* \ar{r}\ar{d} & G G^* \ar{d}\\
        FF^* \ar{r} & \id_{\D}
    \end{tikzcd}
    \]
    which may be identified via the Beck--Chevalley map 
    \[
        H I^* \to G^* F
    \]
    with 
    \[
    \begin{tikzcd}
        GH (FI)^* \ar{r}\ar{d} & G G^* \ar{d}\\
        FF^* \ar{r} & \id_{\D}.
    \end{tikzcd}
    \]
    Finally the identification $FI \simeq GH$ yields a canonical equivalence
    \[
        TS \simeq ST
    \]
    so that the twist functors commute coherently. Examples of $\AA^1 \times
    \AA^1$-schobers and higher--dimensional cubical variants, have already been
    studied in \cite{cdw:complexes}.
\end{rem}

\begin{rem}
    \label{rem:decategorificatoin}
    Upon passage to Grothendieck groups, an $\AA_1 \times \AA_1$-schober gives rise to
    abelian groups $A,B,C,D$ along with additive maps 
    \[
        \begin{tikzcd}[sep=8ex]
        A \ar[bend left=10]{r}{i} \ar[bend left=10]{d}{h} & \ar[bend left=10]{l}{i^*}\ar[bend left=10]{d}{f}  B \\
        C \ar[bend left=10]{u}{h^*} \ar[bend left=10]{r}{g}  & \ar[bend left=10]{u}{f^*} \ar[bend left=10]{l}{g^*}  D
    \end{tikzcd}
    \]
    satisfying the relations
    \begin{enumerate}
        \item $fi = gh$ and $i^*f^* = h^* g^*$,
        \item $t := ff^* - \id$ is an automorphism, 
        \item $s := gg^* - \id$ is an automorphism, 
        \item $hi^* = g^* f$,
        \item $ih^* = f^* g$.
    \end{enumerate}
    Such a datum is known \cite{GMP85} to classify a perverse sheaf
    on $\CC^2$ with respect to the stratification 
    \[
        \CC^2 = \{(0,0)\} \cup H \setminus \{0\} \cup \CC^2 \setminus H
    \]
    determined by the singular hypersurface
    \[
        H = \{ (z_1,z_2) \in \CC^2 \; | \; z_1 z_2 = 0 \} \subset \CC^2.
    \]
\end{rem}

\subsection{\texorpdfstring{$\AA_2$}{A2}-schobers}
\label{subsec:a2schober}

Variants of the conditions entering the following definition (in the context of
``enhanced triangulated categories'') were introduced by Anno-Logvinenko
\cite{AL:skein} when defining skein-triangulated categorical representations of
generalized braid groups. While the definition of loc. cit. is motivated by the
analysis of specific examples of such categorical representations (Nil-Hecke
algebras and cotangent bundles of flag varieties), our definition was found
independently and arises from the desire to categorify the classification of
perverse sheaves on the cuspidal cubic from \cite{gm:rebroussement}, as
explained below. 

\begin{defi}
    \label{defi:a2schober}
    A coherent square as in \eqref{eq:commdiag} is called an {\em
    $\AA_2$-schober} if the following conditions hold:
    \begin{enumerate}[label=(A2.\arabic *)]
        \item All functors in \eqref{eq:commdiag} admit right adjoints. 
        \item The functors $F$ and $G$ define $\AA_1$-schobers with twist functors denoted $T$ and $S$, respectively.
        \item The fiber $\alpha$ of the Beck-Chevalley map 
            \begin{equation}
                \label{eq:bcalpha}
                H I^* \to G^* F
            \end{equation}
            is an equivalence $\alpha: \B \to \C$. 
        \item The fiber $\beta$ of the Beck-Chevalley map 
            \begin{equation}
                \label{eq:bcbeta}
                I H^* \to F^* G
            \end{equation}
            of the transpose square is an equivalence $\beta: \C \to \B$. 
        \item\label{a2:5} The {\em Beck--Chevalley square} 
            \begin{equation}
                \label{eq:bcsquare1}
                \begin{tikzcd}
                    II^* \ar{r} \ar{d} & F^* G G^* F \ar{d} \\
                    \id_{\B} \ar{r} & F^* F 
                \end{tikzcd}
            \end{equation}
            is a biCartesian square of endofunctors of $\B$. Here, the maps are the following:
            \begin{enumerate}
                \item left vertical: the counit of $I \dashv I^*$,
                \item bottom horizontal: the unit of $F \dashv F^*$,
                \item right vertical: the counit of $G \dashv G^*$, whiskered with $F^*$ and $F$ from the
            left and right, respectively,
                \item top horizontal: the composite
                    \[
                        I I^* \to I H^* H I^* \to F^* G G^* F
                    \]
                    of the (whiskering of the) unit of $H \dashv H^*$ and the
                    Beck--Chevalley maps \eqref{eq:bcalpha} and
                    \eqref{eq:bcbeta}.
            \end{enumerate} 
            At this point, it is left to the reader to verify that the square commutes up to canonical 
            homotopy yielding the coherent square \eqref{eq:bcsquare1}, in \S \ref{subsec:higherbc}
            we will give a more systematic account of the description of such
            squares and higher--dimensional generalizations.
        \item\label{a2:6} The Beck--Chevalley square
            \begin{equation}
                \label{eq:bcsquare2}
                \begin{tikzcd}
                    HH^* \ar{r} \ar{d} & G^* F F^* G \ar{d} \\
                    \id_{\C} \ar{r} & G^* G 
                \end{tikzcd}
            \end{equation}
            defined analogously to \eqref{eq:bcsquare1} is a biCartesian square
            of endofunctors of $\C$.
\end{enumerate}
\end{defi}

The classical work \cite{gm:rebroussement} provides a remarkable elementary
description of the category of perverse sheaves on the quotient variety
\[
    X = \{ (z_1, z_2, z_3 \in \CC \; | \; z_1 + z_2 + z_3 = 0 \} / S_3
\]
which can be regarded as the space of unordered configurations of three points
in $\CC$ (possibly with multiplicities) with center of mass $0$. Here the
perversity condition is understood with respect to the discriminant
stratification with strata given by configurations of points with
multiplicities at $1$, $2$, and $3$, respectively. In terms of the
elemtary symmetric coordinates
\[
    X \cong \CC^2, (z_1, z_2, z_3) \mapsto (x = z_1 z_2 + z_1 z_3 + z_2 z_3, y = z_1z_2z_3)
\]
the discriminant corresponds to the cuspidal cubic $4 x^3 + 27 y^2 = 0$, a singularity of Dynkin type $\AA_2$. 

\begin{prop}
    \label{rem:decata2}
    Upon passage to Grothendieck groups, an $\AA_2$-schober defines a perverse
    sheaf on $X$ with respect to the discriminant stratification.
\end{prop}
\begin{proof}
    Passing to Grothendieck groups, we obtain abelian groups $A,B,C,D$ along
    with additive maps 
    \begin{equation}
        \label{eq:a2sheaf}
        \begin{tikzcd}[sep=8ex]
            A \ar[bend left=10]{r}{i} \ar[bend left=10]{d}{h} & \ar[bend left=10]{l}{i^*}\ar[bend left=10]{d}{f}  B \\
            C \ar[bend left=10]{u}{h^*} \ar[bend left=10]{r}{g}  & \ar[bend left=10]{u}{f^*} \ar[bend left=10]{l}{g^*}  D
       \end{tikzcd}
    \end{equation}
    satisfying the relations
    \begin{enumerate}[label=(\arabic *)]
        \item $fi = gh$ and $i^*f^* = h^* g^*$,
        \item $t := ff^* - \id$ is an automorphism of $D$, 
        \item $s := gg^* - \id$ is an automorphism of $D$, 
        \item $a = h i^* - g^* f$ is an isomorphism from $B$ to $C$,
        \item $b = i h^* - f^* g$ is an isomorphism from $C$ to $B$,
        \item \label{eq:6} $h h^* - g^* f f^* g - \id + g^* g = 0$,
        \item  \label{eq:7} $i i^* - f^* g g^* f - \id + f^* f = 0$.
    \end{enumerate}
    From this, we may deduce
    \begin{align*}
        tsf & = f f^* gg^* f - ff^*f - gg^*f + f = fii^* - f - gg^* f + f \\
             & = ghi^* - gg^*f = ga
    \end{align*}
    and similarly the equations $af^* = g^*ts$, $fb = stg$, and $bg^* = f^*st$.
    Using these relations, we further obtain
    \begin{enumerate}
        \item $hh^* = \id + af^*s^{-1}g$,
        \item $ii^* = \id + bg^*t^{-1}f$.
    \end{enumerate}
    We have now verified all relations from \cite[II.6]{gm:rebroussement} so
    that the main theorem in loc. cit. implies that the data \eqref{eq:a2sheaf}
    does indeed define a perverse sheaf. One can easily show that the relations
    given in \cite[II.6]{gm:rebroussement} are in fact {\em equivalent} to our
    decategorified schober relations (1) -- (7).
\end{proof}

\begin{rem}
    \label{rem:braid} Note that the top stratum of $X$ is homotopy equivalent
    to the unordered configuration of points in $\CC^2$ (without
    multiplicities) which is a classifying space for the braid group
    $\on{Br}_3$ on $3$ strands. In particular, since any perverse sheaf defines
    a (shifted) local system of abelian groups on the top stratum, it gives
    rise to a representation of the braid group via monodromy. In terms of the
    above classification data, the underlying abelian group of this
    representation is $D$, the automorphisms $t$ and $s$ correspond to the
    actions of the two Artin generators of the braid group, and a
    straightforward computation, using the relations \ref{eq:6} and \ref{eq:7}
    shows that, indeed, the braid relation
    \[
        tst = sts
    \]
    is satisfied. 
\end{rem}

In Theorem \ref{thm:a2braid}, we will show that the braid group action of
Remark \ref{rem:braid} in fact has a categorical counterpart: for any
$\AA_2$-schober, we obtain a coherent action of the braid group $\on{Br}_3$ on
the stable $\infty$-category $\D$. This generalizes analogous braid actions
announced in the context of enhanced triangulated categories in
\cite{AL:skein}.
As we will see below, this result, and its generalization to $\AA_n$-schobers,
provides a ``geometric'' explanation for many of the categorical braid group
actions that have been considered in the literature, including
$\AA_n$-configurations of spherical objects in homological mirror symmetry, and
Rouqier complexes in link homology theory. 

\section{\texorpdfstring{$\AA_n$}{An}-schobers}
\label{sec:anschobers}

We will now introduce the concept of an $\AA_n$-schober for general $n \ge 1$.
While our definition of an $\AA_2$-schober given in \S \ref{subsec:a2schober} was
based on an ad-hoc categorification of the classifying data for perverse
sheaves on $\Sym^3(\CC)$ of \cite{gm:rebroussement}, our definition of
an $\AA_n$-schober categorifies the remarkable graded bialgebra relations that
were identified by Kapranov--Schechtman \cite{2102.13321} as classifying data
for perverse sheaves on $\Sym^n(\CC)$. 
As it turns out, these bialgebra relations admit canonical categorifications
that can be described systematically in terms of certain higher--dimensional
Beck--Chevalley cubes that we will introduce now.

\subsection{Higher Beck--Chevalley cubes}
\label{subsec:higherbc}

The conditions of Definition \ref{defi:a2schober}, especially conditions
\ref{a2:5} and \ref{a2:6}, may seem a bit random and unmotivated (in fact, they
were found by a rather naive attempt to ``categorify'' the equations of
\cite[II.6]{gm:rebroussement}). As it turns out, we may in fact think of {\em
all} conditions appearing in the definition of an $\AA_2$-schober (and more
generally $\AA_n$-schobers, as we will see later) in terms of
higher--dimensional cubical variants of Beck--Chevalley maps and their total
fibers. In this section, we develop the relevant concepts of higher
Beck--Chevalley cubes and defects. 

Let us begin by rigorously constructing the Beck--Chevalley map
associated to a coherent square 
\begin{equation}
    \label{eq:commsquare}
    \begin{tikzcd}
        \A \ar{r}{I}\ar[swap]{d}{H} & \B \ar{d}{F}\\
        \C \ar{r}{G} & \D
    \end{tikzcd}
\end{equation}
of stable $\infty$-categories. Denoting by $\Catinfty$ the $\infty$-category of
$\infty$-categories, then such a square is given by a functor $[1]
\times [1] \to \Catinfty$ or, equivalently, via the Grothendieck
construction (aka unstraightening) as a coCartesian fibration $\pi: \X \to
[1]^2$ (cf. \cite[3.2]{lurie:htt}). The assumption that all functors in
\eqref{eq:commsquare} admit right adjoints then amounts to the condition that
$\pi$ is also {\em Cartesian} (\cite[5.2.2.5]{lurie:htt}). We will refer to fibrations which are both
Cartesian and coCartesian as {\em biCartesian fibrations}.

The effect of passing to the Grothendieck construction $\X$ is that it
provides very efficient access to
\begin{enumerate}[label = \arabic *.]
    \item the functors in \eqref{eq:commsquare},
    \item their right adjoints, 
    \item the unit and counit maps associated to the various adjunctions, 
    \item higher--dimensional coherence relations among the above,
\end{enumerate}
in terms of universal properties of the biCartesian fibration $\pi$. For an
introduction to this circle of ideas, we refer the reader to
\cite{DKS:spherical}, only giving one basic example as an introduction here:

\begin{exa}
    \label{exa:F}
    To reconstruct the functor $F$ from $\pi$, consider the $\infty$-category
    of coCartesian edges
    \[
        \E_F = \{  b \overset{!}{\to} d \}
    \]
    in $\X$ with $b \in \B$ and $d \in \D$, where the symbol $!$ signifies that
    the edge is required to be coCartesian. Then the projection map $\E_F \to \B$, taking an edge to
    its source, is a trivial Kan fibration so that it admits a section $s: \B
    \to \E_F$, unique up to contractible choice (\cite[4.3.2.15]{lurie:htt}). Postcomposing this section
    with the target projection functor $\E_F \to \D$ yields a funtor $\B \to
    \D$ which is equivalent to $F$. 

    Similarly, to reconstruct a right adjoint of $F$, consider the
    $\infty$-category of Cartesian edges
    \[
    \E_{F^*} = \{ b \overset{*}{\to} d \}
    \]
    in $\X$ with $b \in \B$ and $d \in \D$. Then projection to $\D$ is a
    trivial Kan fibration and the analogous construction recovers the functor
    $F^*: \D \to \B$. 
\end{exa}

\begin{construction}
    \label{construction:bc}
    We provide a construction of the Beck--Chevalley map as defined in
    \eqref{eq:bcalpha}. To this end, consider the $\infty$-category of diagrams
    \begin{equation}
        \label{eq:refcube2}
        \Y = \{ \begin{tikzcd} a \ar{r}{*}\ar{d}{!} &  b \ar{dd}{!}\\ c'\ar{d}  & \\ c  \ar{r}{*} & d  \end{tikzcd} \}
    \end{equation}
    with $a \in \A$, $b \in \B$, $c,c' \in \C$, $d \in \D$, all edges
    marked $!$ coCartesian, and all edges marked $*$ Cartesian. Note that there
    is a {\em hidden} $2$-simplex that exhibits the composition of the edges $a
    \to c'$ and $c' \to c$, this composite then forms the left vertical edge of
    he depicted (coherent) rectangle.
    Projection to $b$ defines a trivial Kan fibration $\Y \to \B$ (\cite[4.3.2.15]{lurie:htt}) and we
    choose a section $s: \B \to \Y$. Postcomposing this section with the
    projection to the edge $c' \to c$ yields a functor
    \[
        \B \to \Fun([1], \C)
    \]
    or, equivalently, an edge $e$ in $\Fun(\B,\C)$, i.e. a natural
    transformation of functors from $\B$ to $\C$. By Example \ref{exa:F}, we
    see that the source of $e$ is the functor $HI^*$ and the target $G^*F$. We
    denote the edge $e$ as the {\em Beck--Chevalley map} associated to the
    square \eqref{eq:commsquare}. Note that, as is immediate from the
    definition of $\Y$, the Beck--Chevalley map is an equivalence if and only
    if vertical coCartesian edges in $\X$ from $\B$ to $\D$ pull back to
    coCartesian edges from $\A$ to $\B$. We refer to the fiber 
    \[
        \alpha: \B \to \C
    \]
    of the Beck--Chevalley map as the {\em Beck--Chevalley defect} which can be
    regarded as a measure of the failure for coCartesian edges being stable
    under pullback in the above-mentioned sense. 
\end{construction}

\begin{exa}
    \label{exa:bcunit}
    Let $F: \A \to \B$ be an exact functor of stable $\infty$-categories and assume that $F$ admits a right adjoint. We consider the square
    \[
    \begin{tikzcd}
        \A \ar{r}{F} \ar{d}[swap]{F}& \B \ar{d}{\id} \\
        \B \ar{r}{\id} & \B
    \end{tikzcd}
    \]
    Then the Beck--Chevalley map associated to this square can be identified with the counit map
    \[
        FF^* \to \id_{\B}
    \]
    of the adjuntion $F \dashv F^*$.
    In particular, we may interpret the twist functors as Beck--Chevalley
    defects. Similarly, the Beck--Chevalley map associated to the square 
    \[
    \begin{tikzcd}
        \A \ar{r}{\id} \ar[swap]{d}{\id}& \A \ar{d}{F} \\
        \A \ar{r}{F} & \B
    \end{tikzcd}
    \]
    can be identified with the unit map
    \[
        \id_{\A} \to F^* F
    \]
    of the adjunction $F \dashv F^*$. 
\end{exa}

\begin{rem}
    \label{rem:laxmat}
    Consider a coherent square 
    \[
        \begin{tikzcd}
            \A \ar{r}{I}\ar[swap]{d}{H} & \B \ar{d}{F}\\
            \C \ar{r}{G} & \D
        \end{tikzcd}
    \]
    of stable $\infty$-categories with right adjoints. We set $K = FI = GH$ and
    let $\A \oplus_K \D$ denote the lax sum in the sense of \cite{cdw:lax}. Then we have maps described by lax matrices:
    \[
        \left(\begin{tikzcd}
        I^* \ar{d}\\ F
       \end{tikzcd}\right): \B \lra \A \oplus_K \D
    \]
    and 
    \[
        \left(\begin{tikzcd}
            H \ar{r} & G^* 
        \end{tikzcd}\right): \A \oplus_K \D \lra \C
    \]
    The lax matrix product
    \[
        \left(\begin{tikzcd}
        H \ar{r} & G^* 
        \end{tikzcd}\right) (\begin{tikzcd}
        I^* \ar{d} \\
        F 
        \end{tikzcd}) = \fib(HI^* \to G^*F) 
    \]
    yields the Beck--Chevalley defect of the square. This is a first indication that
    the lax matric calculus of \cite{cdw:lax}, in particular its
    homotopy-coherent upgrade, introduced in \cite{rush:thesis}, will serve
    as a useful tool for both concrete computations and coherence problems
    within the context of $\AA_n$-schobers as introduced in \S \ref{sec:anschobers}.
\end{rem}

We now move on to the construction of the squares from \ref{a2:5} and
\ref{a2:6}.

\begin{construction}
    \label{construction:inf-def}
    Let 
    \[
        \pi: \X \to [1]^{k}
    \]
    be a biCartesian fibration modelling a coherent cube of stable
    $\infty$-categories with exact functors admitting right adjoints. We define
    certain fundamental functors called {\em inflation} and {\em deflation}.
    Consider the map of posets
    \[
        [1] \times [1]^{k} \to [1]^{k},\; (i,j) \mapsto 
        \begin{cases} 
            0 & \text{if $i = 0$,}\\
            j & \text{if $i = 1$.}
        \end{cases}
    \]
    which, upon geometric realization, amounts to a homotopy contracting the
    cube to its initial vertex. Let 
    \[
        \Y \subset \Fun_{[1]^{k}}([1] \times [1]^{k}, \X)
    \]
    be the full subcategory given by diagrams such that all edges of the form
    $(0,x) \to (1,x)$ in $[1] \times [1]^{k}$ map to $\pi'$-coCartesian edges. 
    We have projection functors
    \[
    \begin{tikzcd}
        \Fun([1]^{k}, \X_{00 \ldots 0}) & \ar{l}{q} \Y \ar{r}{p} & \Fun_{[1]^{k}}([1]^{k}, \X)
    \end{tikzcd}
    \]
    where, by \cite[4.3.2.15]{lurie:htt}, the projection map $q$ is a trivial
    Kan fibration. We choose a section $s$ of $q$ (contractible choice) and define the {\em inflation functor}
    \[
        \inffun_{\X} := p \circ s: \Fun([1]^{k}, \X_{00\ldots0}) \to \Fun_{[1]^{k}}([1]^{k}, \X).
    \]
    Its adjoint, which we refer to as the {\em deflation functor}
    \[
        \deffun_{\X}: \Fun_{[1]^{k}}([1]^{k}, \X) \to \Fun([1]^{k}, \X_{00\ldots0})
    \]
    can be constructed explicitly via a variant of $\Y$ replacing the coCartesian edges by Cartesian ones. 
\end{construction}

The higher--dimensional schober conditions will be formulated in terms of total
fibers of cubical diagrams (see \cite[Appendix A]{djw:auslander} for some basic
background). For additional flexibility, we define total fibers for more
general diagrams:

\begin{defi}
    \label{defi:totalfiber}
    For any simplicial set $K$, we consider the cone $K^{\lhd}$ with cone point
    $\emptyset$ and define the {\em total fiber} of a diagram
    \[
        q: K^{\lhd} \to \C
    \]
    valued in an $\infty$-category $\C$, as the fiber of the canonical map
    \[
        q(\emptyset) \to \lim q|K.
    \]
\end{defi}

\begin{construction}
    \label{construction:bccube}
    We now provide the general construction of a Beck--Chevalley cube of dimension
    $n-1$ associated to an $n$-dimensional cube of exact functors of stable
    $\infty$-categories ($n \ge 2$) with right adjoints 
    \begin{equation}
        \label{eq:cubestab}
        [1]^n \to \Catinfty
    \end{equation}
    corresponding to a biCartesian fibration
    \[
        \pi: \X \to [1]^{n}.
    \]
    We compose the biCartesian fibration $\pi: \X \to [1]^n$ with the
    projection $q:[1]^n \to [1]^2$ to the first two coordinates.
    Since this latter map is a product fibration, it is biCartesian, so that
    the composite $q \pi$ is biCartesian as well. It models the square 
    \[
    \begin{tikzcd}
        \X_{00} \ar{r}\ar{d} & \X_{01} \ar{d}\\
        \X_{10} \ar{r} & \X_{11}
    \end{tikzcd}
    \]
    where $\pi': \X_{ij} \to [1]^{n-2}$ denote the Grothendieck constructions of the $(n-2)$--dimensional cube
    obtained from \eqref{eq:cubestab} by fixing the first two coordinates of $[1]^n$ to $(i,j)$.
    We thus obtain a Beck-Chevalley map
    \[
        \beta_{\X}: \X_{01} \to \Fun([1], \X_{10})
    \]
    over $[1]^{n-2}$, which is the first ingredient of our construction. Using
    left Kan extension along the includion $i$ of the initial vertex in
    $[1]^{n-2}$, along with the inflation and deflation functors from
    Construction \ref{construction:inf-def}, we form the composition 
    \begin{equation}
        \label{eq:compositebc}
        \begin{tikzcd}
        \X_{010 \ldots 0} \ar{d}{i_!} \\
        \Fun([1]^{n-2}, \X_{010 \ldots 0}) \ar{d}{\inffun_{\X_{01}}}\\
        \Fun_{[1]^{n-2}}([1]^{n-2}, \X_{01}) \ar{d}{\beta_{\X}}\\
        \Fun([1], \Fun_{[1]^{n-2}}([1]^{n-2}, \X_{10})) \ar{d}{\deffun_{\X_{10}}}\\
        \Fun([1], \Fun([1]^{n-2}, \X_{100 \ldots 0}))\ar{d}{\cong}\\
        \Fun([1]^{n-1}, \X_{100 \ldots 0}).
        \end{tikzcd}
    \end{equation}
    Finally, we repackage the composite \eqref{eq:compositebc}
    as a cube
    \begin{equation}
        \label{eq:defbccube}
        [1]^{n-1} \to \Fun(\X_{010 \cdots 0}, \X_{100 \cdots 0})
    \end{equation}
    referred to as the {\em Beck--Chevalley cube} associated to $\X$. Further,
    we refer to the total fiber of the cube \eqref{eq:defbccube} as the {\em
    higher Beck--Chevalley defect} of $\X$.
\end{construction}

\begin{exa}
    \label{exa:bcsquare}
    We unravel the construction of the Beck--Chevalley square associated to a
    coherent $3$--dimensional cube of stable $\infty$-categories with right
    adjoints explicitly. We are given a functor
    \begin{equation}
        \label{eq:cubest}
        [1]^3 \to \Catinfty,\; ijk \mapsto \X_{ijk}
    \end{equation}
    or, equivalently, as a biCartesian fibration $\pi:
    \X \to [1]^3$. In this situation, we may describe the composite \eqref{eq:compositebc} by a
    single correspondence as follows: Consider the $\infty$-category $\Y$ of
    diagrams in $\X$ of the form
        \begin{equation}
            \label{eq:higherbccube}
        \begin{tikzpicture}[baseline= (a).base]
            \node[scale=.8] (a) at (0,0){
            \begin{tikzcd}
                    x_{000} \ar{rrr}{*}\ar{ddrr}\ar{dd}{!} & & & x_{010}\ar{ddrrr}{!}\ar{ddddd}{!} & & \\
                                                  &&&&\\
                    x_{100}' \ar{dr}\ar{ddd} & & x_{001}\ar{dd}{!}\ar{rrrr}{*}&& & & x_{011} \ar{ddddd}{!}\\
                                            & x_{100}''\ar{dr}{*}\ar{ddd} & & \\
                                                    &&x_{101}'\ar{ddd}&&& \\
                    x_{100}\ar{dr}\ar{rrr}{*} & &&  x_{110}\ar{ddrrr}{!}\\
                                                    & x_{100}''' \ar{dr}{*} &&&\\
                             &&x_{101} \ar{rrrr}{*} && & & x_{111}
            \end{tikzcd}
            };
        \end{tikzpicture} 
        \end{equation}
        where the subindex of each vertex indicates its image in $[1]^3$ under
    $\pi$, the edges marked with $!$ are coCartesian, and the edges marked with
    $*$ are Cartesian. Projection to the vertex $x_{010}$ defines a trivial Kan
    fibration $\Y \to \X_{010}$ (\cite[4.3.2.15]{lurie:htt}). We choose a section and postcompose with the
    projection functor to the square
        \begin{equation}
            \label{eq:bcsquare}
        \begin{tikzcd}
            x_{100}' \ar{r}\ar{d} & x_{100}''\ar{d} \\
            x_{100} \ar{r} & x_{100}'''
        \end{tikzcd}
        \end{equation}
    in $\X_{100}$ and obtain the Beck--Chevalley square 
    \[
        q: [1]^2 \to \Fun(\X_{010}, \X_{100})
    \]
    as introduced in Construction \ref{construction:bccube}. This more direct
    description of the Beck--Chevalley square can be generalized to cubes of
    arbitrary dimension in a straightforward fashion.
\end{exa}

\begin{rem}
    \label{rem:higherbc}
    We may interpret the Beck--Chevalley square associated to the cubical
    diagram \eqref{eq:cubest} as a morphism between the Beck--Chevalley map of
    the back face and a two-sided whiskering of the Beck--Chevalley map of the front face
    of the cube. It thus induces a comparison of Beck--Chevalley defects of these
    two squares and the higher Beck--Chevalley defect may be interpreted as a
    measure for the deviation of this comparison map from being an equivalence. 
\end{rem}

\begin{rem}
    \label{rem:valuesbc}
    One can easily read off the functors from $\X_{010}$ to $\X_{100}$ that
    form the vertices of the Beck--Chevalley square by tracing the zigzags of
    coCartesian and Cartesian edges that connect the vertex $x_{010}$ to the
    respective vertex of \eqref{eq:bcsquare}. This zigzag translates to a
    composite of successive applications of either one of the original functors
    of the cubical diagram or its adjoint. Denoting each zigzag by the list of
    vertices in $[1]^3$ to which it projects written from top to bottom,
    we may describe the Beck--Chevalley square very conveniently by the
    notation
    \[
    \begin{tikzcd}
        \substack{010\\000\\100} \ar{r}\ar{d} &  \substack{010\\011\\001\\101\\100} \ar{d} \\
        \substack{010\\110\\100} \ar{r} &  \substack{010\\111\\100}.
    \end{tikzcd}
    \]
    We will use this notation extensively in the subsequent sections. 
\end{rem}

\begin{exa}
    \label{exa:a2bc}
    Given a coherent square of stable $\infty$-categories as in \eqref{eq:commsquare}, we may form the cube
    \[
        \begin{tikzcd}[sep=1ex]
        \A \ar{dd}\ar{dr} \ar{rr} &  & \B\ar{dd} \ar{dr} & \\
                            & \C \ar{rr}\ar{dd} &  & \D \ar{dd} \\
        \B \ar{rr}\ar{dr} & & \B \ar{dr} & \\
                           & \D \ar{rr}& & \D .
    \end{tikzcd}
    \]
    Then the associated Beck--Chevalley square is precisely the square
    considered in \ref{a2:5} (cf. Remark \ref{rem:valuesbc}). The square
    considered in \ref{a2:6} is the Beck--Chevalley square of a similar cube
    obtained by swapping $\B$ and and $\C$. The $\AA_2$-schober conditions
    \ref{a2:5} and \ref{a2:6} are therefore equivalent that the higher
    Beck--Chevalley defects associated to the above-mentioned cubical diagrams
    vanish (a square is biCartesian if and only if its total fiber is zero).
\end{exa}

\begin{exa}
    \label{exa:4}
    Using the conventions and notation from Remark \ref{rem:valuesbc}, the
    Beck--Chevalley cube associated to a $4$--dimensional cube of stable
    $\infty$-categories takes the following form:
    \begin{equation}
        \label{eq:3dbccube}
    \begin{tikzcd}[sep=scriptsize]
    \substack{0100\\0000\\1000} \ar{rr}\ar{dd}\ar{dr} &  & \substack{0100\\0110\\0010\\1010\\1000} \ar{dd}\ar{dr} & \\
                                               & \substack{0100\\0101\\0001\\1001\\1000}\ar[crossing over]{rr} & & \substack{0100\\0111\\0011\\1011\\1000}\ar{dd}  \\
    \substack{0100\\1100\\1000} \ar{rr}\ar{dr} &  & \substack{0100\\1110\\1000}\ar{dr} & \\
                                                               & \substack{0100\\1101\\1000}\ar{rr}\ar[from=uu, crossing over]  & & \substack{0100\\1111\\1000}
    \end{tikzcd}
    \end{equation}
\end{exa}

\subsection{Definition}

By a {\em composition} of a natural number $n \ge 0$ we mean an ordered tuple of
positive natural numbers $(n_1, n_2, \ldots, n_k)$ summing up to $n$. In the extremal case $n=0$ the empty tuple serves as the unique composition. We consider the set
$\Comp(n)$ of compositions of a fixed number $n$ as a poset with the order $<$
generated by 
\[
    (n_1, n_2, \ldots, n_k) > (n_1, \ldots, n_i + n_{i+1}, \ldots, n_k ).
\]
To avoid cluttered notation, we will typically write a composition simply as
$n_1 n_2 \ldots n_k$. 

\begin{exa}
    \label{exa:composition23}
    The map
        \begin{equation}
            \label{eq:cubicalcomp}
            \Comp(n) \to \{0,1\}^{n-1}, \; n_1 n_2 \ldots n_k \mapsto (\underbrace{0,\ldots,0}_{n_1-1},1,\underbrace{0,\ldots,0}_{n_2-1},1,0 \ldots 0,1,\underbrace{0,\ldots,0}_{n_k-1})
        \end{equation}
    is an order-preserving bijection, identifying $\Comp(n)$ with a cube of dimension $n-1$. 
    For example, the compositions of $3$ form the square
        \begin{equation}
            \label{eq:comp3}
        \begin{tikzcd}
            3 \ar{r}\ar{d} & 21 \ar{d} \\
            12 \ar{r} & 111
        \end{tikzcd}
        \end{equation}
    and the compositions of $4$ form the cube
    \begin{equation}
        \label{eq:comp4}
    \begin{tikzcd}[sep=1ex]
    4 \ar{dd}\ar{dr} \ar{rr} &  & 31\ar{dd} \ar{dr} & \\
                        & 22 \ar{rr}\ar{dd} &  & 211 \ar{dd} \\
    13 \ar{rr}\ar{dr} & & 121 \ar{dr} & \\
                       & 112 \ar{rr}& & 1111 .
    \end{tikzcd}
    \end{equation}
\end{exa}

Our definition of an $\AA_n$-schober to $n \ge 3$, is based on a more
systematic study of perverse sheaves on $\Sym^{n+1}(\CC)$ due to
Kapranov--Schechtman. In \cite{2102.13321}, they show that the linear-algebraic relations
appearing in the classification of such sheaves can be interpreted as {\em
graded bialgebra relations} in a fixed homogeneous degree $n$ valued in a
suitably defined braided monoidal category freely generated by a single object
of degree $1$: Given a graded vector space $\oplus_{n \in \NN} V_n$ equipped
with multiplication $\mu$ and comultiplication $\Delta$ of degree $0$, such
relations amount to:
\begin{itemize}
    \item For every pair $ab$, $cd$ compositions of $n$, we have
        \[
            \Delta^{n}_{cd} \circ \mu^{ab}_n = \sum (\mu_{c}^{ik} \otimes \mu^{d}_{jl}) \circ (\id_i \otimes T_{jk} \otimes \id_l) \circ  (\Delta^{a}_{ij} \otimes \Delta^{b}_{kl})
        \]
        where $T_{jk}$ denotes the braiding isomorphism $V_j \otimes V_k \cong V_k \otimes V_j$.
\end{itemize}

We will now explain how these bialgebra relations admit a very natural categorification
via higher--dimensional Beck--Chevalley cubes. 

\begin{defi}
    \label{defi:bifactorizationcube}
    We define the {\em bifactorization cube} $Q(ab,cd)$,
    associated to a pair $ab$, $cd$ of compositions of $n$ recursively by the
    following set of formulas. Here, we refer to the cubical coordinates for
    $\Comp(n)$ from \eqref{eq:cubicalcomp}.
    \begin{enumerate}
        \item We have
            \[
                Q(11,11) = \begin{tikzcd} 0\ar{d} \ar{r} & 1 \ar{d}\\ 1 \ar{r} & 1 \end{tikzcd}
            \]
            or, in terms of $\Comp(2)$: 
            \[
                Q(11,11) = \begin{tikzcd} 2\ar{d} \ar{r} & 11 \ar{d}\\ 11 \ar{r} &11. \end{tikzcd}
            \]
        \item For $a > 1$, we have
            \[
                Q(a1,1a) = \{0,1\} \times \{\overbrace{0\ldots0}^{a-1} \}\times \{0,1\}
            \]
            or, in terms of $\Comp(a+1)$: 
            \[
                Q(a1,1a) = \begin{tikzcd} a+1\ar{d} \ar{r} &a1 \ar{d}\\ 1a \ar{r}& 1(a-1)1. \end{tikzcd}
            \]
        \item For $a \ge b \ge 1$, we have 
            \[
                Q((a+1)(b+1),(b+1)(a+1)) = \bigcup_{\epsilon \in \{0,1\}} \{\epsilon\} \times Q(ab,ba) \times \{\epsilon\}
            \]
        \item For $a > b \ge 1$, we have 
            \[
                Q(a(b+1),b(a+1)) = Q(ab,ba) \times \{0,1\}
            \]
            and
            \[
                Q((a+1)b,(b+1)a) = \{0,1\} \times Q(ab,ba) 
            \]
        \item For $a > b \ge 1$ and $m \ge 1$ we have 
            \[
                Q(a(b+m+1),b(a+m+1)) = Q(a(b+m),b(a+m)) \times \{0\}
            \]
            and
            \[
                Q((a+m+1)b,(b+m+1)a) = \{0\} \times Q((a+m)b,(b+m)a) 
            \]
        \item The above relations uniquely determine $Q(ab,cd)$ for $a \ge c$.
            To describe $Q(ab,cd)$ for $a < c$, we may pass to the transpose: For $a,b,c,d
            \ge 1$, we set
            \[
                Q(ab,cd) = Q(cd,ab)^{\on{T}}.
            \]
    \end{enumerate}
\end{defi}

\begin{exa}
    \label{exa:} Let $(ab,cd)$ be a pair of compositions of $n$ with $a \ge c$ and $a <
    d$, so that there exists $m > 0$ with $d = a+m$, and therefore $b = c+m$.
    Then
    \[
        Q(a(c+m),c(a+m)) = Q(ac,ca) \times \{0,1\} \times \{0\}^{m-1}
    \]
    Setting $a = c+l$, for $l \ge 0$, we further have
    \begin{align*}
        Q(ac,ca) 
        &= \cup_{\epsilon_1, \epsilon_2, \ldots, \epsilon_{c-1}}
        \{\epsilon_1 \epsilon_2 \ldots \epsilon_{c-1} \} \times Q((l+1)1,
        1(l+1)) \times \{\epsilon_{c-1} \ldots \epsilon_2 \epsilon_{1} \}
    \end{align*}
    Similarly, the above recursive formulas uniquely determine the
    bifactorization cube for any pair $(ab,cd)$ of compositions of $n$. 
\end{exa}

\begin{defi}
    \label{defi:highertwist}
    Let 
    \[
        \X: \Comp(n+1) \to \Stab
    \]
    be a coherent diagram of exact functors of stable $\infty$-categories with
    right adjoints and let $ab$ be a composition of $n+1$. We obtain an exact
    functor
    \[
        T_{ab}: \X_{ab} \to \X_{ba}
    \]
    defined as the Beck--Chevalley defect of the pullback of $\X$ to the
    bifactorization cube $Q(ab,ba)$. We call $T_{ab}$ the {\em higher twist
    functor} associated to the composition $ab$.
\end{defi}

\begin{defi}
    \label{defi:anschober}
    A coherent diagram 
    \begin{equation}
        \label{eq:compncube}
        \X: \Comp(n+1) \to \Stab
    \end{equation}
    of exact functors of stable $\infty$-categories is called an $\AA_n$-schober if the following five conditions hold.
    \begin{enumerate}
        \item[\textbf{(Adjunctability)}] All functors of the diagram $\X$ admit right adjoints. 
        \item[\textbf{(Recursiveness)}] For every proper composition $n_1 n_2 \ldots n_k$ of $n+1$, and
            every $1 \le i \le k$, the restriction of $\X$ along
            \[
                \Comp(n_i) \to \Comp(n+1),\; t \mapsto n_1 \ldots n_{i-1}\; t\; n_{i+1} \ldots n_k
            \]
            is an $\AA_{n_i-1}$-schober. Here, an $\A_{0}$-schober is simply an object in $\St$. 
        \item[\textbf{(Far-commutativity)}]  For every composition $ab$ of $n+1$ and compositions $c_0 \le c_1$
            and $d_0 \le d_1$ of $a$ and $b$, respectively, the square
            \[
            \begin{tikzcd}
                \X_{c_0 d_0} \ar{r}\ar{d} & \ar{d}\X_{c_0 d_1}\\
                \X_{c_1 d_0} \ar{r} & \X_{c_1 d_1}
            \end{tikzcd}
            \]
            and its transpose satisfy the Beck--Chevalley condition (i.e. the
            Beck--Chevalley map is an equivalence).
        \item[\textbf{(Twist invertibility)}] For every composition $ab$ of $n+1$, the higher twist functor
            $T_{ab}: \X_{ab} \to \X_{ba}$ is an equivalence of stable
            $\infty$-categories.
        \item[\textbf{(Defect vanishing)}] For every pair of compositions $ab$, $cd$ of $n+1$ with $a \ne d$,
            the Beck--Chevalley defect $R_{ab,cd}$ associated to the pullback
            of $\X$ to the bifactorization cube $Q(ab,cd)$ is the zero functor
            $\X_{ab} \to \X_{cd}$. 
    \end{enumerate}
\end{defi}

\begin{exa}
    \label{exa:a3schober}
    An $\AA_3$-schober is a coherent diagram 
    \[
        \X: \Comp(4) \to \Stab
    \]
    where $\Comp(4)$ is the composition cube
    \[
    \begin{tikzcd}[sep=1ex]
    4 \ar{dd}\ar{dr} \ar{rr} &  & 31\ar{dd} \ar{dr} & \\
                        & 22 \ar{rr}\ar{dd} &  & 211 \ar{dd} \\
    13 \ar{rr}\ar{dr} & & 121 \ar{dr} & \\
                       & 112 \ar{rr}& & 1111 
    \end{tikzcd}
    \]
    such that 
    \begin{enumerate}
        \item All functors of the diagram $\X$ admit right adjoints. 
        \item The restrictions to 
                \[
                \begin{tikzcd}
                    31 \ar{r}\ar{d} & 211 \ar{d}\\
                    121 \ar{r} & 1111 
                \end{tikzcd}
                \quad\quad
               \text{ and}
               \quad\quad
                \begin{tikzcd}
                    13 \ar{r}\ar{d} & 121 \ar{d}\\
                    112 \ar{r} & 1111 
                \end{tikzcd}
                \]
                are $\AA_2$-schobers, in particular, the restrictions to $211
                \to 1111$, $121 \to 1111$, $112 \to 1111$ are $\AA_1$-schobers.
        \item The restriction to 
                \[
                \begin{tikzcd}
                    22 \ar{r}\ar{d} & 211 \ar{d}\\
                    112 \ar{r} & 1111 
                \end{tikzcd}
                \]
                is an $\AA_1 \times \AA_1$-schober. 
        \item The following higher twist functors are equivalences
            \begin{enumerate}
                \item $T_{13}: \X_{13} \to \X_{31}$ defined as the total fiber of
                    \[
                        \begin{tikzcd}
                            \substack{13\\4\\31} \ar{r} &  \substack{13\\121\\31} 
                        \end{tikzcd}
                    \]
                \item $T_{22}: \X_{22} \to \X_{22}$ defined as the total fiber of
                    \[
                        \begin{tikzcd}
                            \substack{22\\4\\22} \ar{r}\ar{d} &\ar{d}  \substack{22\\1111\\121\\1111\\22} \\
                            \substack{22} \ar{r} &  \substack{22\\1111\\22} 
                        \end{tikzcd}
                    \]
                \item $T_{31}: \X_{31} \to \X_{13}$ defined as the total fiber of
                    \[
                        \begin{tikzcd}
                            \substack{31\\4\\13} \ar{r} &  \substack{31\\121\\13} 
                        \end{tikzcd}
                    \]
            \end{enumerate}
        \item The Beck-Chevalley squares corresponding to the factorization cubes $Q(13,13)$ and $Q(31,31)$ are biCartesian, these are:
                    \[
                        \begin{tikzcd}
                            \mathrm{(a)} &
                            \substack{13\\4\\13} \ar{r}\ar{d} &\ar{d}  \substack{13\\112\\22\\112\\13} \\
                           & \substack{13} \ar{r} &  \substack{13\\112\\13} 
                        \end{tikzcd}
                        \quad\quad
                        \text{ and}
                        \quad\quad
                        \begin{tikzcd}
                            \mathrm{(b)}  &\substack{31\\4\\31} \ar{r}\ar{d} &\ar{d}  \substack{31\\211\\22\\211\\31} \\
                            &\substack{31} \ar{r} &  \substack{31\\211\\31} 
                        \end{tikzcd}
                     \]
    \end{enumerate}
\end{exa}

\subsection{Categorical bialgebra relations}
\label{sec:grbialg}

We explain how the defining conditions of an $\AA_n$-schober from Definition
\ref{defi:anschober} are in fact ``equivalent'' to certain categorical analogs of bialgebra
relations of \cite{2102.13321}. The phenomenon can be best seen by looking at a specific example
which already captures the general argument:

\begin{exa}
    \label{exa:43-25} 
    Let $\X$ be an $\AA_6$-schober and consider the pair $(43,25)$ of compositions
    of $7$. The bifactorization cube $Q(43,25)$ is a $4$-dimensional cube with back face
    
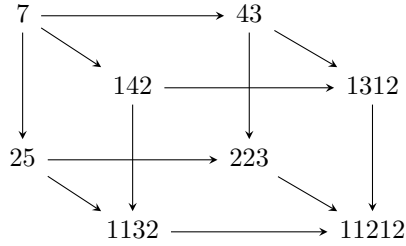
\[
        \begin{tikzcd}[sep=scriptsize]
        7 \ar{rr}\ar{dr}\ar{dd} & & 43 \ar{dd}\ar{dr}& \\
                         &142\ar{dd}\ar{rr} &&1312 \ar{dd}\\
        25 \ar{rr}\ar{dr} & & 223\ar{dr} & \\
                          &1132\ar{rr}&& 11212 
    \end{tikzcd}
    \]
    and front face 
    
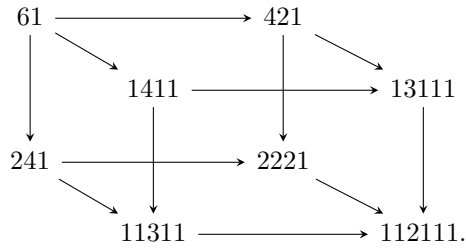
\[
    \begin{tikzcd}[sep=scriptsize]
        61 \ar{rr}\ar{dr}\ar{dd} & & 421\ar{dd}\ar{dr}& \\
                         &1411\ar{dd}\ar{rr} &&13111 \ar{dd}\\
        241\ar{rr}\ar{dr} & & 2221\ar{dr} & \\
                          &11311\ar{rr}&& 112111.
    \end{tikzcd}
    \]
    Using the conventions and notation from Remark \ref{rem:valuesbc}, the
    associated Beck--Chevalley cube (cf. Example \ref{exa:4}) is
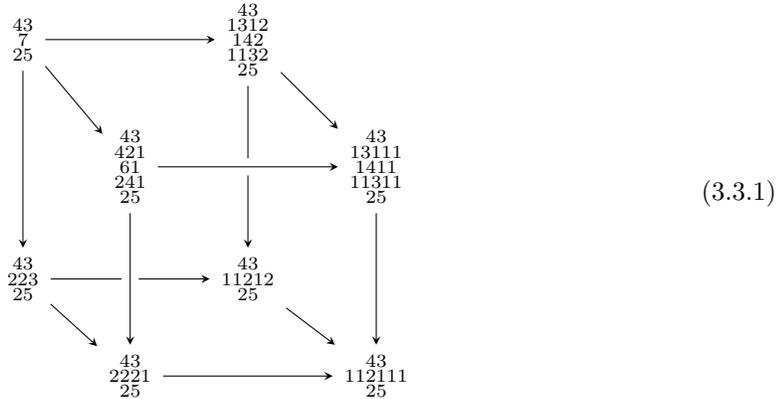
    \begin{equation}
        \label{eq:3dbccubeexa}
    \begin{tikzcd}[sep=scriptsize]
    \substack{43\\7\\25} \ar{rr}\ar{dd}\ar{dr} &  & \substack{43\\1312\\142\\1132 \\25} \ar{dd}\ar{dr} & \\
                                               & \substack{43\\421\\61\\241\\25}\ar[crossing over]{rr} & & \substack{43\\13111\\1411\\11311\\25}\ar{dd}  \\
    \substack{43\\223\\25} \ar{rr}\ar{dr} &  & \substack{43\\11212\\25}\ar{dr} & \\
                                                               & \substack{43\\2221\\25}\ar{rr}\ar[from=uu, crossing over]  & & \substack{43\\112111\\25}
    \end{tikzcd}
    \end{equation}
    Unravelling the definitions, we observe that for the back right vertical map, we have
    \[
        \fib(\substack{43\\1312\\142\\1132 \\25} \to \substack{43\\11212\\25}) \simeq \substack{1132\\25} \circ (\id_{1} \otimes T_{31} \otimes \id_2) \circ \substack{43\\1312} 
    \]
    and the total fiber of the front face of \eqref{eq:3dbccubeexa} has total fiber equivalent to 
    \[
         \substack{241\\25} \circ (T_{42} \otimes \id_1) \circ \substack{43\\421}.
    \]
    The $\AA_6$-schober condition that the total fiber $R_{43,25}$ of the
    Beck--Chevalley cube \eqref{eq:3dbccubeexa} is zero (so that the cube is
    biCartesian), therefore implies that we obtain a canonical Waldhausen cell (all squares are biCartesian)
    \begin{equation}
        \label{eq:waldhausen}
    \begin{tikzcd}
        \substack{241\\25} \circ (T_{42} \otimes \id_1) \circ \substack{43\\421} \ar{r}\ar{d} & x \ar{r}\ar{d} & \substack{43\\7\\25} \ar{d} \\
            0  \ar{r}  & \ar{d}\ar{r} \substack{1132\\25} \circ (\id_{1} \otimes T_{31} \otimes \id_2) \circ \substack{43\\1312}  & \ar{d} y\\
                       & 0        \ar{r}  & \substack{43\\223\\25} 
    \end{tikzcd}
    \end{equation}
    where $x$ is the fiber of the map
    \[
            \substack{43\\7\\25} \to \substack{43\\223\\25} 
    \]
    and $y$ is the limit of the punctured back face of \eqref{eq:3dbccubeexa}. In other words, the functor 
    \[
        \substack{43\\7\\25}: \X_{43} \to \X_{25}
    \]
    is exhibited as an extension of the diagonal terms of \eqref{eq:waldhausen}, yielding a categorification of the graded bialgebra relation
    \begin{align*}
        \Delta^{7}_{2,5} \circ \mu^{4,3}_7 & =  (\mu_{2}^{02} \otimes \mu_{5}^{41}) \circ (T_{42} \otimes \id_1) \circ  (\Delta^{4}_{04} \otimes \Delta^{3}_{21})\\
                                           & + (\mu_{2}^{11} \otimes \mu_{5}^{32}) \circ (\id_1 \otimes T_{31} \otimes \id_2) \circ  (\Delta^{4}_{13} \otimes \Delta^{3}_{12})\\
                                           & + (\mu_{2}^{20} \otimes \mu_{5}^{23}) \circ (\Delta^{4}_{22} \otimes \Delta^{3}_{03})
    \end{align*}
\end{exa}

The type of diagram appearing in \eqref{eq:waldhausen} is an instance of a cell
in Waldhausen's $\S$-construction which, in general, can be described as
follows: For $n \ge 0$ and a stable $\infty$-category $\C$, denote by 
\[
    \wS_n(\C) \subset \Fun(\Fun([1],[n]), \C)
\]
the full subcategory of diagrams $X$ such that
\begin{enumerate}
    \item the diagonal terms $X_{ii}$ are zero objects and 
    \item for every $0 \le i < j < k \le n$, the square
        \[
        \begin{tikzcd}
            X_{ij} \ar{r}\ar{d}  & X_{ik} \ar{d} \\
            X_{jj} \ar{r} & X_{jk}
        \end{tikzcd}
        \]
        is biCartesian. 
\end{enumerate}
As already alluded to in Example \ref{exa:43-25}, we may think of the datum of
$X \in \wS_n(\C)$ as a ``filtration''
\[
    X_{01} \to X_{02} \to \cdots \to X_{0n}
\]
exhibiting $X_{0n}$ as an extension of its associated graded terms $X_{n-1,n}$,
$X_{n-2,n-1}$, ..., $X_{01}$. We refer to the elements of $\wS_n(\C)$ as
Waldhausen diagrams of length $n$. 

To formulate the categorical bialgebra relations in the most uniform way it is
useful to reformulate Waldhausen diagrams in terms of coherent complexes. In
particular, this will be useful when connecting to examples coming from
differential graded categories where coherent complexes can be nicely modelled
as twisted complexes. 
We define a {\em coherent complex} (of length $n$) in a stable
$\infty$-category $\C$ to be a cubical coherent diagram
\[
    Y: [1]^n \to \C
\]
such that the only vertices with nonzero values are the $(n+1)$ vertices of the
form $(0,\ldots, 0, 1, \ldots, 1)$. We denote the $\infty$-category of
complexes of length $n$ by $\Ch^n(\C)$. To save dimensions, we will schematically
depict a coherent complex by simply writing its nonzero terms
\[
    Y_0 \to Y_1 \to \ldots \to Y_n,
\]
keeping the coherent system of zero homotopies implicit. A coherent complex is
called {\em exact} if its total fiber vanishes. The full subcategory of
$\Ch^n(\C)$ spanned by the exact complexes will be denoted by
$\Ch^n_{\on{exa}}(\C)$.

\begin{exa}
    \label{exa:coherentcomplex}
    A coherent complex of length $2$ is a square of the form
    \[
    \begin{tikzcd}
        Y_{00} \ar{r}\ar{d} & Y_{01} \ar{d}\\
        0 \ar{r} & Y_{11}.        
    \end{tikzcd}
    \]
    It is exact if and only if it is biCartesian. A coherent complex of length
    $3$ is a cube of the form 
    \[
    \begin{tikzcd}[sep=scriptsize]
        Y_{000} \ar{rr}\ar{dr}\ar{dd} & & Y_{001}\ar{dd}\ar{dr}& \\
                         &0\ar{dd}\ar{rr} &&0 \ar{dd}\\
        0\ar{rr}\ar{dr} & & Y_{011} \ar{dr} & \\
                        &0 \ar{rr}&& Y_{111}.
    \end{tikzcd}
    \]
\end{exa}

\begin{prop}
    \label{prop:koszul}
    Let $n \ge 0$ and let $\C$ be a stable $\infty$-category. Then there is an
    equivalence of stable $\infty$-categories
    \[
        \wS_n(\C) \lra \Ch^n_{\on{ex}}(\C)
    \]
    taking a Waldhausen diagram $\{X_{ij}\}$ to the (canonically defined)
    exact coherent complex 
    \[
        X_{0,n} \to X_{n-1,n} \to X_{n-2,n-1}[1] \to \ldots \to X_{0,1}[n-1].
    \]
\end{prop}
\begin{proof}
    This is an generalization of the classical derived equivalence of representations
    of the $A_n$-quiver and the $A_n$-quiver with zero relations from the
    representation theory of finite dimensional algebras. In the stated form,
    it can be deduced from the discussion of Koszul complexes in
    \cite{cdw:complexes} or the proof based on derivators in \cite{Bec18}, or the
    forthcoming Bachelor thesis of Rio Haeussler Albi for a direct proof of a
    more refined statement describing semiorthogonal decompositions of higher
    length.
\end{proof}

\begin{exa}
    \label{exa:42complex}
    The Waldhausen diagram \eqref{eq:waldhausen} corresponds to the exact coherent complex
    \[
           \substack{43\\7\\25} \to 
           \substack{43\\223\\25} \to 
           \substack{1132\\25} \circ (\id_{1} \otimes T_{31} \otimes \id_2) \circ \substack{43\\1312}[1] \to 
           \substack{241\\25} \circ (T_{42} \otimes \id_1) \circ \substack{43\\421} [2]
    \]
    of length $3$.
\end{exa}

\begin{thm}
    \label{thm:schoberbialgebra}
    Let 
    \[
        \X: \Comp(n+1) \to \Stab
    \]
    be an $\AA_n$-schober. Then, for every pair $(ab,cd)$ of compositions of
    $n+1$, there exists a canonical exact complex of the form
    \begin{equation}
        \label{eq:coherentbialgebra}
        \substack{ab\\n+1\\cd} \to \ldots \to \substack{ikjl\\cd} \circ (\id_i \otimes T_{jk} \otimes \id_l) \circ \substack{ab\\ijkl}[m] \to \ldots
    \end{equation}
    where the positive terms of the complex correspond to the tuples $(i,j,k,l)$ such
    that $a = i+j$, $b = k+l$, $c = i+k$, and $d=j+l$, ordered by the index
    $j$, and the $p$th term of the complex is shifted by $m = p-1$. 
    In particular, for the pair $(ab,ba)$, we obtain an exact complex of the form
    \begin{equation}
        \label{eq:twistbialgebra}
        \substack{ab\\n+1\\ba} \to \cdots \to T_{ab}[m].
    \end{equation}
\end{thm}
\begin{proof}
    The proof of Example \ref{exa:43-25} provides a blueprint. The recursive
    nature (cf. Definition \ref{defi:bifactorizationcube}) of the
    bifactorization cube $Q(ab,cd)$ leads to a filtration of the
    Beck--Chevalley cube resulting in a Waldhausen diagram analogous to
    \eqref{eq:waldhausen}. An application of Proposition \ref{prop:koszul} then
    produces the desired exact complex \eqref{eq:coherentbialgebra}.
\end{proof}

\begin{rem}
    \label{rem:equivalent}
    The relations ``generated by'' the categorical bialgebra relations of Theorem
    \ref{thm:schoberbialgebra} (relations form a two-sided ideal $\sim$
    recursiveness!) together with far--commutativity relations are
    ``equivalent'' to our $\AA_n$-schober conditions from Definition
    \ref{defi:anschober}. In particular, the exact complexes of the form
    \eqref{eq:twistbialgebra} can be solved for $T_{ab}$ so that they provide a
    ``formula'' for the twists obtained from their participation in the
    bialgebra relations. This was our approach for finding another form of the
    bialgebra relations which can be formulated in immediate terms from the
    coherent composition diagram. 

    In contrast, it is {\em not} immediate to produce the coherent complexes
    \eqref{eq:coherentbialgebra} (specifically the differentials and coherent
    system of zero homotopies) directly from the composition cube $\X$. The
    proof of Theorem \ref{thm:schoberbialgebra} explains how to mediate between
    the Beck--Chevalley conditions from Definition \ref{defi:anschober} and the
    categorical bialgebra relations \eqref{eq:coherentbialgebra}, but it would
    be tricky to give a direct definition in terms of
    \eqref{eq:coherentbialgebra} without referring to the conditions of
    Definition \ref{defi:anschober}.
\end{rem}

\subsection{Framed \texorpdfstring{$\AA_n$}{An}-schobers}
\label{subsec:framed}

As already alluded to in Remark \ref{rem:spherical}, our concept of
$\AA_n$-schober only involves twist functors and not the cotwist functors that
are typically part of the definition of a spherical functor. We are now in a
position to explain the relevance of including cotwists. 

\begin{defi}
    \label{defi:cotwist}
    Let $n \ge 2$ and let 
    \[
        \X: \Comp(n) \to \Stab
    \]
    be a coherent diagram of exact functors of stable $\infty$-categories with
    right adjoints. 
    We define the {\em cotwist}
    \[
        T_{n}: \X_n \to \X_n
    \]
    as the total cofiber of the composite
    \begin{equation}
        \X_{0 \ldots 0}  \xrightarrow{i_!} 
        \Fun([1]^{n-1}, \X_{0\ldots 0})\xrightarrow{\inffun_{\X}}
        \Fun_{[1]^{n-1}}([1]^{n-1}, \X)  \xrightarrow{\deffun_{\X}}
        \Fun([1]^{n-1}, \X_{0 \ldots 0}).
    \end{equation}
    Where $i_!$ is relative left Kan extension along the inclusion $i$ of the
    initial vertex of $[1]^{n-1}$, and the functors $\inffun_{\X}$ and
    $\deffun_{\X}$ are the inflation and deflation functors introduced in 
    Definition \ref{construction:inf-def}.
    We call $T_{n}$ the {\em cotwist functor} in degree $n$. For $n = 1$, we
    set $T_{1} = \id$. 
\end{defi}

\begin{exa}
    \label{exa:cotwist}
    The cotwist $T_2$ is the cofiber of the unit 
    \[
        \substack{2} \to \substack{2\\11\\2}
    \]
    which, up to shift conventions, agrees with the cotwist for spherical functors. 
    The cotwist $T_3$ is the total cofiber of the square
    \[
    \begin{tikzcd}
        \substack{3} \ar{r}\ar{d} & \substack{3\\21\\3} \ar{d} \\ 
        \substack{3\\12\\3} \ar{r} & \substack{3\\111\\3}.
    \end{tikzcd}
    \]
\end{exa}

\begin{defi}
    \label{defi:highercotwist}
    An $\AA_n$-schober is called {\em framed} if, for every composition $n_1
    n_2 \ldots n_k$ of $n+1$, and every $1 \le i \le k$, the cotwist of the restriction of $\X$ along
    \[
        \Comp(n_i) \to \Comp(n+1),\; t \mapsto n_1 \ldots n_{i-1}\; t\; n_{i+1} \ldots n_k
    \]
    is an autoequivalence. Here, we include the trivial composition $n+1$, so
    that the cotwist $T_{n+1}$ is an equivalence. 
\end{defi}

\begin{rem}
    \label{rem:balancing}
    The Soergel schober defined in \S \ref{sec:soergelschober} is framed.
    Motivated by this example, we expect that the cotwist autoequivalences of a
    factorizing family $\X_{\bullet}$ of $\AA_n$-schobers should be
    interpreted as a {\em half-balancing} on the free braided monoidal
    $(\infty,2)$-category associated to $\X$; see
    \cite{MR2579396,enriquez2010halfbalancedbraidedmonoidalcategories,Bae} for
    related pre-categorified developments. Further, the categorical bialgbera object is required
    to satisfy certain compatibility contraints with respect to the
    half-balancing (cf. Example \ref{exa:framed}). In this work, we will not
    discuss framed $\AA_n$-schobers systematically, we plan to get back to this
    in future work.  
\end{rem}

\begin{exa}
    \label{exa:framed}
    A framed $\AA_1$-schober is simply a spherical adjunction 
    \[
        \substack{2\\11}: \X_{2} \longleftrightarrow \X_{11}: \substack{11\\2}
    \]
    in the usual sense where twist 
    \[
        T_{11} = \fib( \substack{11\\2\\11} \to \substack{11}): \X_{11} \to \X_{11}
    \]
    and cotwist
    \[
        T_{2} = \cofib( \substack{2} \to \substack{2\\11\\2}): \X_2 \to \X_2
    \]
    are equivalences. 
    We have a coherent diagram 
    \[
    \begin{tikzcd}
        0 \ar{r}\ar{d} & \substack{11\\2} \circ T_{11}  \ar{r}\ar{d} & 0 \ar{d}\\
        \substack{11\\2} \ar{r}\ar{d} &\ar{d}  \substack{11\\2\\11\\2} \ar{r} & \substack{11\\2} \ar{d}\\
        0 \ar{r} & T_{2} \circ \substack{11\\2} \ar{r} &  0
    \end{tikzcd}
    \]
    with all squares biCartesian and where the middle horizontal composite
    is the identity due to the snake relation of the adjunction. In
    particular, we obtain an equivalence (since the right-hand rectangle is
    biCartesian) 
    \begin{equation}
        \label{eq:balance2}
        \substack{11\\2} \circ T_{11}  \simeq T_{2} \circ \substack{11\\2}.
    \end{equation}
    Also taking into account $T_1 = \id$, this translates to the identity:
    \[
        T_2 \circ \mu^{11}_2 \simeq \mu^{11}_2 \circ T_{11} \circ T_1 \otimes T_1
    \]
\end{exa}

\subsection{Braid group actions from \texorpdfstring{$\AA_n$}{An}-schobers}

We begin by reformulating the definition of an $\AA_2$-schober from \S
\ref{subsec:a2schober} in terms of the notation introduced in Remark
\ref{rem:valuesbc}. It amounts to a coherent square of stable
$\infty$-categories parametrized by the composition poset $\Comp(3)$:
\[
\X: \begin{tikzcd}
    3 \ar{r}\ar{d}  & 21 \ar{d}\\
    12 \ar{r} & 111
\end{tikzcd} \quad \lra \quad \Stab
\]
such that
\begin{enumerate}
    \item The functor 
        \[
            T = \fib(\substack{111\\21\\111} \to \substack{111})
        \]
        is an autoequivalence $\X_{111} \to \X_{111}$. 
    \item The functor 
        \[
            S = \fib(\substack{111\\21\\111} \to \substack{111})
        \]
        is an autoequivalence $\X_{111} \to \X_{111}$. 
    \item The functors
        \[
            T_{21}: \fib(\substack{21\\3\\12} \to \substack{21\\111\\12})
        \]
        and 
        \[
            T_{12}: \fib(\substack{12\\3\\21} \to \substack{12\\111\\21})
        \]
        are equivalences $\X_{21} \to \X_{12}$ and $\X_{12} \to \X_{21}$, respectively.
    \item The squares 
    \[
                    \begin{tikzcd}
                        \substack{21\\3\\21} \ar{r}\ar{d} &\ar{d}  \substack{21\\111\\12\\111\\21} \\
                        \substack{21} \ar{r} &  \substack{21\\111\\21} 
                    \end{tikzcd} \quad\quad  \text{and} \quad\quad
                    \begin{tikzcd}
                        \substack{12\\3\\12} \ar{r}\ar{d} &\ar{d}  \substack{12\\111\\21\\111\\12} \\
                        \substack{12} \ar{r} &  \substack{12\\111\\12}
                    \end{tikzcd} 
    \]
    of functors $\X_{21} \to \X_{21}$ and $\X_{12} \to \X_{12}$, respectively, are biCartesian.
\end{enumerate}

\begin{thm}
    \label{thm:a2braid}
    Any $\AA_2$-schober defines a canonical coherent $\Br_3$-action on the underlying stable $\infty$-category $\X_{111}$. 
\end{thm}
\begin{proof}
    We denote 
    \[
        T = \fib(\substack{111\\21\\111} \to \substack{111})
    \]
    and 
    \[
        S = \fib(\substack{111\\12\\111} \to \substack{111}).
    \]
    Then $TST$ and $STS$ are given by the total fibers of the cubes
    \[
    \begin{tikzcd}[sep=tiny]
    \substack{111\\21\\111\\12\\111\\21\\111} \ar{rr}\ar{dd}\ar{dr} &  & \substack{111\\12\\111\\21\\111} \ar{dd}\ar{dr} & \\
                                               & \substack{111\\21\\111\\21\\111}\ar[crossing over]{rr} & & \substack{111\\21\\111}\ar{dd}  \\
    \substack{111\\21\\111\\12\\111} \ar{rr}\ar{dr} &  & \substack{111\\12\\111}\ar{dr} & \\
                                                               & \substack{111\\21\\111}\ar{rr}\ar[from=uu, crossing over]  & & \substack{111}
    \end{tikzcd}\quad\quad \text{and} \quad\quad
    \begin{tikzcd}[sep=tiny]
    \substack{111\\12\\111\\21\\111\\12\\111} \ar{rr}\ar{dd}\ar{dr} &  & \substack{111\\21\\111\\12\\111} \ar{dd}\ar{dr} & \\
                                               & \substack{111\\12\\111\\12\\111}\ar[crossing over]{rr} & & \substack{111\\12\\111}\ar{dd}  \\
    \substack{111\\12\\111\\21\\111} \ar{rr}\ar{dr} &  & \substack{111\\21\\111}\ar{dr} & \\
                                                               & \substack{111\\12\\111}\ar{rr}\ar[from=uu, crossing over]  & & \substack{111}
    \end{tikzcd} \text{, respectively.}
    \]
    The $\AA_2$-schober conditions \ref{a2:5} and \ref{a2:6} translate, using
    our current terminology from Remark \ref{rem:valuesbc} to the biCartesian
    squares
    \[
                    \begin{tikzcd}
                        \substack{21\\3\\21} \ar{r}\ar{d} &\ar{d}  \substack{21\\111\\12\\111\\21} \\
                        \substack{21} \ar{r} &  \substack{21\\111\\21} 
                    \end{tikzcd} \quad\quad  \text{and} \quad\quad
                    \begin{tikzcd}
                        \substack{12\\3\\12} \ar{r}\ar{d} &\ar{d}  \substack{12\\111\\21\\111\\12} \\
                        \substack{12} \ar{r} &  \substack{12\\111\\12}. 
                    \end{tikzcd} 
    \]
    Using the first of these squares whiskered by $111 \to 21$ and $21 \to 111$, we may combine it with the square totalizing to $TST$ to obtain the prism
    \[
    \begin{tikzcd}[sep=tiny]
        \substack{111\\3\\111}\ar{rr}\ar{dr}\ar{dd}& &\substack{111\\21\\111\\12\\111\\21\\111} \ar{rr}\ar{dd}\ar{dr} &  & \substack{111\\12\\111\\21\\111} \ar{dd}\ar{dr} & \\
                                                   & \substack{111\\21\\111}\ar{rr}\ar{dd} &        & \substack{111\\21\\111\\21\\111}\ar[crossing over]{rr} & & \substack{111\\21\\111}\ar{dd}  \\
        \substack{111\\21\\111\\12\\111}\ar{dr}\ar{rr}   & & \substack{111\\21\\111\\12\\111} \ar{rr}\ar{dr} &  & \substack{111\\12\\111}\ar{dr} & \\
                                                         &\substack{111\\21\\111}\ar{rr} &                           & \substack{111\\21\\111}\ar{rr}\ar[from=uu, crossing over]  & & \substack{111}
    \end{tikzcd}
    \]
    where the left cube is biCartesian since its top and bottom face are. Thus,
    by $2/3$ for cubes we obtain an equivalence between the total fiber of the
    outer prism and the right hand side cube, which is $TST$. 
    Now the total fiber of this cube
    \[
    \begin{tikzcd}[sep=tiny]
        \substack{111\\3\\111}\ar{rr}\ar{dr}\ar{dd}&  & \substack{111\\12\\111\\21\\111} \ar{dd}\ar{dr} & \\
                                                   & \substack{111\\21\\111}\ar[crossing over]{rr}\ar{dd} &    & \substack{111\\21\\111}\ar{dd}  \\
        \substack{111\\21\\111\\12\\111}\ar{dr}\ar{rr}   &   & \substack{111\\12\\111}\ar{dr} & \\
                                                         &\substack{111\\21\\111}\ar{rr} &      & \substack{111}
    \end{tikzcd}
    \]
    can, via right Kan extension, be canonically identified with the total fiber of the globular diagram
    
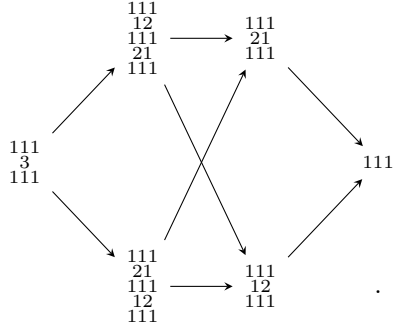
\[
    \begin{tikzcd}
        & \substack{111\\12\\111\\21\\111} \ar{r}\ar{ddr}& \substack{111\\21\\111} \ar{dr}& \\
        \substack{111\\3\\111} \ar{ur}\ar{dr} & &  & \substack{111}\\
                                              &\substack{111\\21\\111\\12\\111}\ar{r}\ar{uur}& \substack{111\\12\\111}\ar{ur} & .
    \end{tikzcd}
    \]
    A similar argument leads to a canonical identification of $STS$ with the
    total fiber of the same diagram, thus producing a preferred identification $TST = STS$. 
\end{proof}

We note that this proof extracts the essence of a pattern of similar
arguments concerning the invariance of link homology constructions under the
third Reidemeister move, see e.g. \cite[Figures 43 and 44]{Kho3} or \cite[Figures 53 and 54]{KR}.

\begin{cor}
    \label{cor:anbraid}
    Any $\AA_n$-schober $\X$ defines a canonical $\Br_{n+1}$-action up to homotopy on
    the underlying stable $\infty$-category $\X_{11 \cdots 1}$. 
\end{cor}
\begin{proof}
    This is a direct consequence of Theorem \ref{thm:a2braid} and Remark \ref{rem:twistscommute}.
\end{proof}

Of course, the datum of a local system of $\infty$-categories on the open
stratum of $\Sym^{n+1}(\CC)$ corresponds to a coherent braid group action. We expect
that this can indeed be obtained, but it involves a much more elaborate
construction which we plan to get back to in future work. For now, we leave it
as a conjecture. 

\begin{conj}
    \label{conj:anbraidcoherent} Any $\AA_n$-schober $\X$ defines a canonical
    coherent $\Br_{n+1}$-action on the underlying stable $\infty$-category $\X_{11
    \cdots 1}$. 
\end{conj}

\subsection{Example: \texorpdfstring{$\AA_n$}{An}-configurations of spherical objects}
\label{sec:sphericalconf}

We provide a first class of examples of $\AA_n$-schobers: namely
$\AA_n$-configurations of spherical objects as introduced by Seidel and Thomas.
To connect to the classical context of spherical objects let $k$ be a field and
let $\Modk$ be the monoidal $\infty$-category of $k$-module spectra. Let $\D$
be a presentable $k$-linear stable $\infty$-category, i.e., a stable
$\infty$-category left tensored over $\Modk$. In this setup, any compact
object $E$ of $\D$ determines an adjunction
\[
    - \otimes E : \Modk \longleftrightarrow \D : \Hom(E,-)
\]
of $k$-linear functors. The twist functor is then given by
\[
    T_E: \D \to \D,\; X \mapsto \fib(\Hom(E,X) \otimes E \to X).
\]

\begin{exa}
    \label{exa:A1A1}
    Let $E,E'$ be compact objects of $\D$. Then the diagram 
    \[
    \begin{tikzcd}
        0 \ar{r} \ar{d}& \Modk \ar{d}{-\otimes E'} \\
        \Modk \ar{r}[swap]{-\otimes E} & \D 
    \end{tikzcd}
    \]
    is an $\AA_1 \times \AA_1$-schober if and only if
    \begin{enumerate}
        \item The twist functors $T_E$ and $T_{E'}$ are autoequivalences. 
        \item $\Hom(E,E') = 0$ and $\Hom(E',E) = 0$. 
    \end{enumerate}
    In particular, any pair of orthogonal spherical objects $E,E'$ defines an $\AA_1 \times \AA_1$-schober. 
\end{exa}

\begin{exa}
    \label{exa:A2}
    Let $E,E'$ be compact objects of $\D$. Then the diagram 
    \begin{equation}
        \label{eq:a2schob}
    \begin{tikzcd}
        0 \ar{r} \ar{d}& \Modk \ar{d}{-\otimes E'} \\
        \Modk \ar[swap]{r}{-\otimes E} & \D 
    \end{tikzcd}
    \end{equation}
    is an $\AA_2$-schober if and only if
    \begin{enumerate}
        \item The twist functors $T_E$ and $T_{E'}$ are autoequivalences. 
        \item\label{en:2} The functors
            \[
                \Modk \to \Modk, X \mapsto \Hom(E,E') \otimes X
            \]
            and
            \[
                \Modk \to \Modk, X \mapsto \Hom(E',E) \otimes X
            \]
            are equivalences. This implies that $\Hom(E.E') \simeq k[m]$ and
            $\Hom(E',E) \simeq k[n]$. 
        \item The maps
            \[
                k \oplus (\Hom(E,E') \otimes \Hom(E',E)) \lra \Hom(E,E), (\lambda, (f,g)) \mapsto \lambda \id + gf
            \]
            and
            \[
                k \oplus (\Hom(E',E) \otimes \Hom(E,E')) \lra \Hom(E',E'), (\lambda, (f,g)) \mapsto \lambda \id + gf
            \]
            are equivalences so that, together with \ref{en:2}, we obtain
            \[\Hom(E,E) \simeq k \oplus k[n+m] \quad \text{and} \quad \Hom(E',E') \simeq k \oplus
            k[n+m].\]
    \end{enumerate}
    In particular, any $\AA_2$-configuration of spherical objects $E$ and $E'$
    in the sense of Seidel-Thomas defines an $\AA_2$-schober. In fact, if we
    also force the dual $\AA_2$-schober conditions, then $\AA_2$-schobers of
    the form \eqref{eq:a2schob} correspond precisely to $\AA_2$-configurations of spherical objects. 
\end{exa}

\begin{exa}
    \label{exa:ancollection}
    Combining Example \ref{exa:A1A1} and Example \ref{exa:A2}, we see that
    $\AA_n$-schobers $\X$ with values in $k$-linear stable $\infty$-categories
    such that $\X(11 \dots1) = \D$, $\X(1 \ldots 121 \ldots 1) = \Modk$, and
    $\X(m_1 \ldots m_k) = 0$ else, correspond to $\AA_n$-collections of
    spherical objects in $\D$.
    In particular, Corollary \ref{cor:anbraid} recovers the celebrated main
    result of Seidel-Thomas stating that any $\AA_n$-collection of spherical
    objects in $\D$ defines an action of $\Br_{n+1}$ on $\D$.
\end{exa}

\begin{rem}
    \label{rem:e2}
    It is formal to replace the field $k$ in the above examples by an
    $\EE_2$-algebra. In particular, one could set $k$ to be the sphere
    spectrum, to obtain $\Mod_k \simeq \Sp$. 
\end{rem}

\section{Factorizing families of \texorpdfstring{$\AA_n$}{An}-schobers}
\label{sec:monoidal}

The classification data for perverse sheaves on $\Sym^n(\CC)$ according to
\cite{2102.13321} correspond to $\NN$-graded bialgebra relations seen within a
fixed homogeneous degree $n$. Of course, this strongly suggests that one ought
to combine all symmetric products to obtain an actual $\NN$-graded bialgebra
given as a sequence of perverse sheaves on $\{\Sym^n(\CC),\; n \ge 0\}$
equipped with suitable identification data connecting the sheaves on the
various symmetric products. In \cite{2102.13321}, this is realized in terms
of a suitably defined PROB. In our stably categorified context, this turns out to be simpler,
since multiplication and comultiplication are adjoint to one another so that we
only need to specify one of them. 

We simply have to note that the disjoint union of all composition posets
\[
    \Comp := \amalg_{n \in \NN} \Comp(n)
\]
forms a (strict) monoidal category where the tensor product is given by
concatenation of compositions, i.e.
\[
    (n_1 \ldots n_k) \otimes (m_1 \ldots m_l) := (n_1 \ldots n_k m_1 \ldots m_l).
\]
Here, we include the empty composition $\Comp(0) = \{ \emptyset \}$ as the
unit. Further, we equip the $\infty$-category of presentable stable
$k$-linear $\infty$-categories $\Pr^{\on{st}}_k$ where $k$ is an
$\EE_{\infty}$-algebra (for the examples in this paper, $k$ will be a field)
with Lurie's tensor product (\cite[5.1.3]{lurie:ha}) making it a symmetric
monoidal $\infty$-category. 

\begin{defi}
    \label{defi:multschober}
    A {\em factorizing family of $k$-linear $\AA_n$-schobers} is a monoidal functor
    \[
        \X_{\bullet}: \Comp \lra \Pr^{\on{st}}_k
    \]
    such that, for every $n \ge 2$, the restriction $\X_{n+1}$ of
    $\X_{\bullet}$ to $\Comp(n+1)$ is an $\AA_n$-schober.
\end{defi}

\begin{rem}
    \label{rem:simplefact}
    Note that it follows immediately from the definitions that, for a
    factorizing family $\X_{\bullet}$ of $\AA_n$-schobers, the conditions
    (Recursiveness) and (Far-commutativity) in Definition \ref{defi:anschober}
    are automatic as they are induced from lower--dimensional schober
    conditions via the monoidality of the functor $\X_{\bullet}$ (using the
    exactness of the tensor product on $\Pr^{\on{st}}_k$). 
\end{rem}

\section{Soergel schobers}
\label{sec:soergelschober}

\subsection{Soergel cubes}
Consider the standard presentation of $S_{n}$ as a Coxeter group and the set of parabolic subgroups, partially ordered by inclusion. This defines a coherent diagram
\begin{equation}\label{eq:parabolics}
\X^{\Par}\colon \; \Comp(n)^\op \to \Grp 
\end{equation}
mapping a composition $n_1n_2\dots n_k$ of $n$ to the parabolic subgroup 
\[S_{n_1n_2\dots n_k}:=S_{n_1}\times S_{n_2}\times \cdots \times S_{n_k}\subset S_{n}.\]

\begin{definition}
    A (commutative) Frobenius extension is an inclusion of unital rings
    $\iota\colon A\hookrightarrow B$, such that $B$ is free and
    finitely-generated as an $A$-module, together with the data of an $A$-linear
    \emph{trace} $\del=\del_{A}^B\colon B \to A$ that is \emph{non-degenerate}
    in the sense that $B$ can be equipped with a pair of $A$-linear bases
    $\{b_i\}$ and $\{b^*_i\}$, such that $\del(b_i b^*_j)=\delta_{i,j}$. 
\end{definition}

It is straightforward to check that every ring admits a unique identity
Frobenius extension and that the composition of Frobenius extensions
$A\hookrightarrow B$ and $B\hookrightarrow C$ yields a Frobenius extension
$A\hookrightarrow C$ with $\del_{A}^C=\del_{A}^B\del_{B}^C$. We denote by
$\FrobExt$ the category with objects given by unital rings and morphisms by
Frobenius extensions.

Cubical diagrams of Frobenius extensions were systematically studied in
\cite{ESW}. We consider the following example, see \cite[Example 1.5]{ESW}, and will throughout work over the rationals $\Q$ for concreteness, although most aspects work in significantly greater generality.
\begin{construction}
    For $n\geq 0$ we consider the graded polynomial ring $R:=\Q[x_1,\dots,x_{n}]$ with
    all variables $x_i$ of degree two. The symmetric group $S_{n}$ acts on $R$
    by permutation of variables. By passing to rings of invariants under parabolic
    subgroups
    \[
    R_{n_1n_2\dots n_k}:= R^{S_{n_1n_2\dots n_k}},
    \]
    the cube of parabolic subgroups $\X^{\Par}$ from \eqref{eq:parabolics} yields a coherent diagram
    \begin{equation}\label{eq:soergelcube}
\X^{\Frob}\colon \; \Comp(n) \to \FrobExt
    \end{equation}
of Frobenius extensions of graded $\Q$-algebras, the \emph{Soergel cube} of
$S_{n}$. 

We now describe the requisite system of Frobenius traces explicitly. For
parabolic subgroups $I\subset J \subset S_{n}$ we write $w_I$ and $w_J$ for
the corresponding longest elements and $\ell(w_I),\ell(w_J)$ for the respective
lengths. Associated to each simple transposition $s=(i,i+1)$ we have the
Demazure operator $D_{s}\colon R\to R$ acting on $p \in R$ by 
    \[ D_{s}(p) = \frac{p - s(p)}{x_i - x_{i+1}}.\] By \cite{MR342522}, these
    operators satisfy the type $A$ braid relations as well as $D_s^2=0$. Any
    choice of reduced expression $w_Jw_I^{-1}=s_1\cdots s_k$ in terms of simple
    transpositions $s_i\in S_{n}$ will thus define the same endomorphism
\[ D_I^J = \del_{s_1}\circ\cdots\circ\del_{s_k} \] of $R$ which, furthermore,
restricts to a non-degenerate trace $\del_I^J\colon R^I\to R^J$ of degree
$2\ell(w_I)-2\ell(w_J)$. The coherence of the system of traces $\Soe(S_{n})$
follows from a similar argument.
\end{construction}

\begin{rem}
    \label{rem:cohflags}
    The Soergel cubes $\X^{\Frob}$ are a convenient way to package various maps
    obtained in $\mathrm{GL}$-equivariant cohomology by pullback and pushforward
    along forgetful maps between partial flag varieties, with Frobenius traces
    appearing as instances of Poincar\'e duality, see \cite{elias2024demazureoperatorsdoublecosets}. 
\end{rem}

Recall the Morita bicategory $\Mor$, whose objects are unital rings,
$\Hom_{\Mor}(A,B)$ is the category of finitely-generated $(B,A)$-bimodules and
horizontal composition is given by relative tensor product. We write $\Mor^\Z$
for the analogous bicategory of $\Z$-graded unital rings, finitely-generated
graded $(B,A)$-bimodules, and grading-preserving bimodule homomorphisms. 

For any inclusion $\iota\colon A \to B$ of unital rings, we consider the
induction bimodule ${}_{B}B_A$ and the restriction bimodule ${}_{A}B_B$ as 1-morphisms in $\Mor$ between which we have the adjunction
\[
    {}_{B}B_A\dashv
    {}_{A}B_B
\]
with counit and unit induced by the multiplication and unit of $B$ respectively. The data of a Frobenius extension upgrades this to an ambidextrous adjunction \cite{MR190183}
\[
    {}_{B}B_A\dashv
    {}_{A}B_B\dashv
    {}_{B}B_A
\]
with the counit of the new adjunction induced by the Frobenius trace $\del_A^B$. In the $\Z$-graded setting with a trace homogeneous of degree $-d$, the double right adjoint of ${}_{B}B_A$ in $\Mor^\Z$ is isomorphic up to shift, namely $\qdeg^{-d}{}_{B}B_A $.

\begin{construction}
    The Soergel cube \eqref{eq:soergelcube} gives rise to a coherent diagram
    \begin{equation}\label{eq:moritacube}
        \X^{\Mor}\colon \; \Comp(n) \to \Mor^\Z
            \end{equation}
            in which every bimodule has a right adjoint. Moreover, for every
            composition $ab$ of $n$ and compositions $c_0 \le c_1$ and $d_0
            \le d_1$ of $a$ and $b$, respectively, the square
            \[
            \begin{tikzcd}
                \X^{\Mor}_{c_0 d_0} \ar{r}\ar{d} & \ar{d}\X^{\Mor}_{c_0 d_1}\\
                \X^{\Mor}_{c_1 d_0} \ar{r} & \X^{\Mor}_{c_1 d_1}
            \end{tikzcd}
            \]
            and its transpose satisfy the Beck--Chevalley condition, simply
            because the relevant bimodules are canonically isomorphic. 

            Next, we can lift this
            diagram from the world of abelian categories of graded bimodules to
            the corresponding derived $\infty$-categories. These then act by
            derived tensor product on appropriate stable $\infty$-categories of
            graded modules over partially symmetric polynomial rings. In
            summary, we obtain a coherent diagram      
            \begin{equation}\label{eq:stablecube}
                \X_{n}\colon \; \Comp(n) \to \mathrm{st}_{\Q}^{B\Z}
            \end{equation}
         in the $\infty$-category $\mathrm{st}_{\Q}^{B\Z}$ of small stable
         idempotent-complete $\Q$-linear $\infty$-categories, equipped with a
         compatible $\Z$-action, see \cite[\S4]{LMGRSW}. The Beck-Chevalley
         conditions on squares parametrized by compositions $c_0 \le c_1$ and
         $d_0 \le d_1$ of $a$ and $b$ with $a+b=n$ is inherited from \eqref{eq:moritacube}.
        \end{construction}

        \begin{defi}
            The 2-category of \emph{singular Bott--Samelson bimodules} $\sBSBim_{n}$ is the sub
            2-category of $\Mor^\Z$ whose objects are the graded $\Q$-algebras
            on the vertices of the Soergel cube \eqref{eq:moritacube} and whose
            $\Hom$-categories are full subcategories of graded bimodules
            generated under grading shift and horizontal composition by bimodules
            on the edges\footnote{Restricting to edges makes the horizontal
            composition strictly associative.} of the Soergel cube. Its hom-wise idempotent completion defines the 2-category of \emph{singular Soergel bimodules} $\sSBim_{n}$.
        \end{defi}

        \begin{rem}
          The Morita bicategories $\Mor$ and $\Mor^\Z$ are prototypical examples
          of monoidal bicategories with the monoidal structure $\boxtimes$ induced on
          objects and 1-morphisms by tensor product over $\Q$. The cubes
          $\X^{\Mor} \colon \Comp(n) \to \Mor^\Z$ from \eqref{eq:moritacube}
          for all $n\geq 0$ form a factorizing family 
          \begin{equation}
            \label{eq:moritacubefamily}
          \X^{\Mor}_\bullet \colon \Comp \to \Mor^\Z
        \end{equation}
          with
          respect to this structure. We give two related perspectives.

          As in \cite{stroppel2024braidingtypesoergelbimodules}, one
          can show that $\sBSBim_{n}$ respectively $\sSBim_{n}$ for all $n\geq
          0$ form the hom categories in semistrict monoidal 2-categories
          $\sBSBim$ and $\sSBim$ respectively, whose set of objects are the
          partitions of all integers $n\geq 0$ and tensor product is given by
          concatenation.
          
          Alternatively, following \cite[Definition 6.2.6]{LMGRSW}, $\sSBim$ can
          be modelled as $\EE_1$-monoidal $(2,2)$-category enriched over
          $\mathrm{add}_{\Q}^{B\Z}$, the presentably symmetric monoidal
          $\infty$-category of small additive idempotent-complete $\Q$-linear
          $\infty$-categories equipped with a $\Z$-action. Along these lines,
          \eqref{eq:stablecube} form a factorizing family 
          \begin{equation}
            \label{eq:stablecubefamily}
          \X_\bullet \colon \Comp \to \mathrm{st}_{\Q}^{B\Z}
        \end{equation}
          with respect to the
          presentably symmetric monoidal structure on $\mathrm{st}_{\Q}^{B\Z}$
          constructed in \cite[\S4]{LMGRSW}.
        \end{rem}

        \begin{rem}[Graphical language]
            \label{rem:graphicallanguage} All bimodules on the edges of the
        Soergel cubes \eqref{eq:moritacube} are generated under whiskering using
        $\boxtimes$ by basic induction bimodules, namely those associated to an
        edge from the initial vertex in some cube. These are parametrized by
        pairs $a,b\in \N$ and graphically denoted by \emph{split webs}
\begin{equation}\label{eq:split}
	\begin{tikzpicture}[scale =.75, smallnodes,rotate=270,anchorbase]
		\draw[very thick] (0.25,0) node[left,xshift=2pt]{$b$} to [out=90,in=210] (.5,.75);
		\draw[very thick] (.75,0) node[left,xshift=2pt]{$a$} to [out=90,in=330] (.5,.75);
		\draw[very thick] (.5,.75) to (.5,1.5) node[right,xshift=-2pt]{$a{+}b$};
	\end{tikzpicture}
	\; := \; 
            \phantom{()}_{R_{ab}}(R_{ab})_{R_{a+b}}
    \end{equation}
    We read graphical representations of $1$-morphisms as directed from right to
    left. Objects are compositions read from bottom to top. It is an established
    convention that the analogous \emph{merge web} represents a grading-shifted version of the right-adjoint restriction bimodule:
    \begin{equation}\label{eq:merge}
    \begin{tikzpicture}[scale =.75, smallnodes, rotate=90,anchorbase]
		\draw[very thick] (0.25,0) node[right,xshift=-2pt]{$a$} to [out=90,in=210] (.5,.75);
		\draw[very thick] (.75,0) node[right,xshift=-2pt]{$b$} to [out=90,in=330] (.5,.75);
		\draw[very thick] (.5,.75) to (.5,1.5) node[left,xshift=2pt]{$a{+}b$};
	\end{tikzpicture}\,
	\; := \; 
    {\qdeg^{-ab}} \!\!\!\!\phantom{()}_{R_{a+b}}(R_{ab})_{R_{ab}}
	\end{equation}
The counits and units for the two sides of the ambidextrous adjunction (up to shifts) are homogeneous bimodule homomorphisms that can be visualized as foams
    
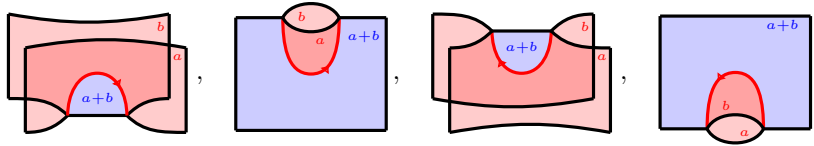
\[
        \begin{tikzpicture} [scale=.45,fill opacity=0.2,anchorbase,tinynodes]
            \path [fill=red] (4.25,-.5) to (4.25,2) to [out=170,in=10] (-.5,2) to (-.5,-.5) to 
                [out=0,in=225] (.75,0) to [out=90,in=180] (1.625,1.25) to [out=0,in=90] 
                    (2.5,0) to [out=315,in=180] (4.25,-.5);
            \path [fill=red] (3.75,.5) to (3.75,3) to [out=190,in=350] (-1,3) to (-1,.5) to 
                [out=0,in=135] (.75,0) to [out=90,in=180] (1.625,1.25) to [out=0,in=90] 
                    (2.5,0) to [out=45,in=180] (3.75,.5);
            \path[fill=blue] (.75,0) to [out=90,in=180] (1.625,1.25) to [out=0,in=90] (2.5,0);
            \draw [very thick] (2.5,0) to (.75,0);
            \draw [very thick] (.75,0) to [out=135,in=0] (-1,.5);
            \draw [very thick] (.75,0) to [out=225,in=0] (-.5,-.5);
            \draw [very thick] (3.75,.5) to [out=180,in=45] (2.5,0);
            \draw [very thick] (4.25,-.5) to [out=180,in=315] (2.5,0);
            \draw [very thick, red, directed=.75] (.75,0) to [out=90,in=180] (1.625,1.25)
                to [out=0,in=90] (2.5,0);
            \draw [very thick] (3.75,3) to (3.75,.5);
            \draw [very thick] (4.25,2) to (4.25,-.5);
            \draw [very thick] (-1,3) to (-1,.5);
            \draw [very thick] (-.5,2) to (-.5,-.5);
            \draw [very thick] (4.25,2) to [out=170,in=10] (-.5,2);
            \draw [very thick] (3.75,3) to [out=190,in=350] (-1,3);
            \node [blue, opacity=1]  at (1.625,.5) {$a{+}b$};
            \node[red, opacity=1] at (3.5,2.6) {$b$};
            \node[red, opacity=1] at (4,1.75) {$a$};		
        \end{tikzpicture}
        \, , \quad
        \begin{tikzpicture} [scale=.5,fill opacity=0.2,anchorbase,tinynodes]
            \path[fill=red] (-.75,4) to [out=270,in=180] (0,2.5) to [out=0,in=270] 
                (.75,4) .. controls (.5,4.5) and (-.5,4.5) .. (-.75,4);
            \path[fill=blue] (-.75,4) to [out=270,in=180] (0,2.5) to [out=0,in=270] (.75,4) -- 
                (2,4) -- (2,1) -- (-2,1) -- (-2,4) -- (-.75,4);
            \path[fill=red] (-.75,4) to [out=270,in=180] (0,2.5) to [out=0,in=270] 
                (.75,4) .. controls (.5,3.5) and (-.5,3.5) .. (-.75,4);
            \draw[very thick] (2,1) -- (-2,1);
            \path (.75,1) .. controls (.5,.5) and (-.5,.5) .. (-.75,1); 
            \draw [very thick, red, directed=.65] (-.75,4) to [out=270,in=180] 
                (0,2.5) to [out=0,in=270] (.75,4);
            \draw[very thick] (2,4) -- (2,1);
            \draw[very thick] (-2,4) -- (-2,1);
            \draw[very thick] (2,4) -- (.75,4);
            \draw[very thick] (-.75,4) -- (-2,4);
            \draw[very thick] (.75,4) .. controls (.5,3.5) and (-.5,3.5) .. (-.75,4);
            \draw[very thick] (.75,4) .. controls (.5,4.5) and (-.5,4.5) .. (-.75,4);
            \node [blue, opacity=1]  at (1.4,3.5) {$a{+}b$};
            \node[red, opacity=1] at (.25,3.4) {$a$};
            \node[red, opacity=1] at (-.25,4) {$b$};
        \end{tikzpicture}
        \, , \quad
        \begin{tikzpicture} [scale=.45,fill opacity=0.2,anchorbase,tinynodes]
            \path [fill=red] (4.25,2) to (4.25,-.5) to [out=170,in=10] (-.5,-.5) to (-.5,2) to
                [out=0,in=225] (.75,2.5) to [out=270,in=180] (1.625,1.25) to [out=0,in=270] 
                    (2.5,2.5) to [out=315,in=180] (4.25,2);
            \path [fill=red] (3.75,3) to (3.75,.5) to [out=190,in=350] (-1,.5) 
                to (-1,3) to [out=0,in=135] (.75,2.5) to [out=270,in=180] (1.625,1.25) 
                    to [out=0,in=270] (2.5,2.5) to [out=45,in=180] (3.75,3);
            \path[fill=blue] (2.5,2.5) to [out=270,in=0] (1.625,1.25) to [out=180,in=270] (.75,2.5);
            \draw [very thick] (4.25,-.5) to [out=170,in=10] (-.5,-.5);
            \draw [very thick] (3.75,.5) to [out=190,in=350] (-1,.5);
            \draw [very thick, red, directed=.75] (2.5,2.5) to [out=270,in=0] (1.625,1.25) 
                to [out=180,in=270] (.75,2.5);
            \draw [very thick] (3.75,3) to (3.75,.5);
            \draw [very thick] (4.25,2) to (4.25,-.5);
            \draw [very thick] (-1,3) to (-1,.5);
            \draw [very thick] (-.5,2) to (-.5,-.5);
            \draw [very thick] (2.5,2.5) to (.75,2.5);
            \draw [very thick] (.75,2.5) to [out=135,in=0] (-1,3);
            \draw [very thick] (.75,2.5) to [out=225,in=0] (-.5,2);
            \draw [very thick] (3.75,3) to [out=180,in=45] (2.5,2.5);
            \draw [very thick] (4.25,2) to [out=180,in=315] (2.5,2.5);
            \node [blue, opacity=1]  at (1.625,2) {$a{+}b$};
            \node[red, opacity=1] at (3.5,2.6) {$b$};
            \node[red, opacity=1] at (4,1.75) {$a$};		
        \end{tikzpicture}
        \, , \quad 
        \begin{tikzpicture} [scale=.5,fill opacity=0.2,anchorbase,tinynodes]
            \path[fill=red] (-.75,-4) to [out=90,in=180] (0,-2.5) to [out=0,in=90] 
                (.75,-4) .. controls (.5,-4.5) and (-.5,-4.5) .. (-.75,-4);
            \path[fill=blue] (-.75,-4) to [out=90,in=180] (0,-2.5) to [out=0,in=90] (.75,-4) -- 
                (2,-4) -- (2,-1) -- (-2,-1) -- (-2,-4) -- (-.75,-4);
            \path[fill=red] (-.75,-4) to [out=90,in=180] (0,-2.5) to [out=0,in=90] 
                (.75,-4) .. controls (.5,-3.5) and (-.5,-3.5) .. (-.75,-4);
            \draw[very thick] (2,-1) -- (-2,-1);
            \path (.75,-1) .. controls (.5,-.5) and (-.5,-.5) .. (-.75,-1); 
            \draw [very thick, red, directed=.65] (.75,-4) to [out=90,in=0] 
                (0,-2.5) to [out=180,in=90] (-.75,-4);
            \draw[very thick] (2,-4) -- (2,-1);
            \draw[very thick] (-2,-4) -- (-2,-1);
            \draw[very thick] (2,-4) -- (.75,-4);
            \draw[very thick] (-.75,-4) -- (-2,-4);
            \draw[very thick] (.75,-4) .. controls (.5,-3.5) and (-.5,-3.5) .. (-.75,-4);
            \draw[very thick] (.75,-4) .. controls (.5,-4.5) and (-.5,-4.5) .. (-.75,-4);
            \node [blue, opacity=1]  at (1.25,-1.25) {$a{+}b$};
            \node[red, opacity=1] at (-.25,-3.4) {$b$};
            \node[red, opacity=1] at (.25,-4.1) {$a$};
        \end{tikzpicture}
\]
of degrees $ab$, $-ab$, $ab$, and $-ab$ respectively. Squares containing an
initial vertex and edges in two adjacent coordinate axes have on their boundary
two isomorphic composites of induction bimodules. These isomorphisms and their
analogues for restriction bimodules can be visualized as
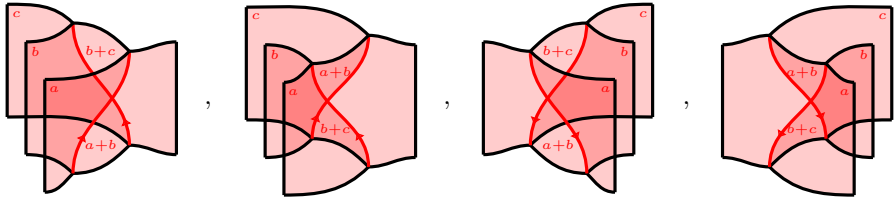
\[\begin{tikzpicture} [scale=.5,fill opacity=0.2,anchorbase]
	\path[fill=red] (-2.5,4) to [out=0,in=135] (-.75,3.5) to [out=270,in=90] (.75,.25)
		to [out=135,in=0] (-2.5,1);
	\path[fill=red] (-.75,3.5) to [out=270,in=125] (.29,1.5) to [out=55,in=270] (.75,2.75) 
		to [out=135,in=0] (-.75,3.5);
	\path[fill=red] (-.75,-.5) to [out=90,in=235] (.29,1.5) to [out=315,in=90] (.75,.25) 
		to [out=225,in=0] (-.75,-.5);
	\path[fill=red] (-2,3) to [out=0,in=225] (-.75,3.5) to [out=270,in=125] (.29,1.5)
		to [out=235,in=90] (-.75,-.5) to [out=135,in=0] (-2,0);
	\path[fill=red] (-1.5,2) to [out=0,in=225] (.75,2.75) to [out=270,in=90] (-.75,-.5)
		to [out=225,in=0] (-1.5,-1);
	\path[fill=red] (2,3) to [out=180,in=0] (.75,2.75) to [out=270,in=55] (.29,1.5)
		to [out=305,in=90] (.75,.25) to [out=0,in=180] (2,0);
	\draw[very thick] (2,0) to [out=180,in=0] (.75,.25);
	\draw[very thick] (.75,.25) to [out=225,in=0] (-.75,-.5);
	\draw[very thick] (.75,.25) to [out=135,in=0] (-2.5,1);
	\draw[very thick] (-.75,-.5) to [out=135,in=0] (-2,0);
	\draw[very thick] (-.75,-.5) to [out=225,in=0] (-1.5,-1);
	\draw[very thick, red, rdirected=.85] (-.75,3.5) to [out=270,in=90] (.75,.25);
	\draw[very thick, red, rdirected=.75] (.75,2.75) to [out=270,in=90] (-.75,-.5);	
	\draw[very thick] (-1.5,-1) -- (-1.5,2);	
	\draw[very thick] (-2,0) -- (-2,3);
	\draw[very thick] (-2.5,1) -- (-2.5,4);	
	\draw[very thick] (2,3) -- (2,0);
	\draw[very thick] (2,3) to [out=180,in=0] (.75,2.75);
	\draw[very thick] (.75,2.75) to [out=135,in=0] (-.75,3.5);
	\draw[very thick] (.75,2.75) to [out=225,in=0] (-1.5,2);
	\draw[very thick]  (-.75,3.5) to [out=225,in=0] (-2,3);
	\draw[very thick]  (-.75,3.5) to [out=135,in=0] (-2.5,4);
	\node[red,opacity=1] at (-2.25,3.75) {\tiny$c$};
	\node[red,opacity=1] at (-1.75,2.75) {\tiny$b$};	
	\node[red,opacity=1] at (-1.25,1.75) {\tiny$a$};
	\node[red,opacity=1] at (0,2.75) {\tiny$_{b+c}$};
	\node[red,opacity=1] at (0,.25) {\tiny$_{a+b}$};
\end{tikzpicture}
\quad , \quad
\begin{tikzpicture} [scale=.5,fill opacity=0.2,anchorbase]
	\path[fill=red] (-2.5,4) to [out=0,in=135] (.75,3.25) to [out=270,in=90] (-.75,.5)
		 to [out=135,in=0] (-2.5,1);
	\path[fill=red] (-.75,2.5) to [out=270,in=125] (-.35,1.5) to [out=45,in=270] (.75,3.25) 
		to [out=225,in=0] (-.75,2.5);
	\path[fill=red] (-.75,.5) to [out=90,in=235] (-.35,1.5) to [out=315,in=90] (.75,-.25) 
		to [out=135,in=0] (-.75,.5);	
	\path[fill=red] (-2,3) to [out=0,in=135] (-.75,2.5) to [out=270,in=125] (-.35,1.5) 
		to [out=235,in=90] (-.75,.5) to [out=225,in=0] (-2,0);
	\path[fill=red] (-1.5,2) to [out=0,in=225] (-.75,2.5) to [out=270,in=90] (.75,-.25)
		to [out=225,in=0] (-1.5,-1);
	\path[fill=red] (2,3) to [out=180,in=0] (.75,3.25) to [out=270,in=45] (-.35,1.5) 
		to [out=315,in=90] (.75,-.25) to [out=0,in=180] (2,0);				
	\draw[very thick] (2,0) to [out=180,in=0] (.75,-.25);
	\draw[very thick] (.75,-.25) to [out=135,in=0] (-.75,.5);
	\draw[very thick] (.75,-.25) to [out=225,in=0] (-1.5,-1);
	\draw[very thick]  (-.75,.5) to [out=225,in=0] (-2,0);
	\draw[very thick]  (-.75,.5) to [out=135,in=0] (-2.5,1);	
	\draw[very thick, red, rdirected=.75] (-.75,2.5) to [out=270,in=90] (.75,-.25);
	\draw[very thick, red, rdirected=.85] (.75,3.25) to [out=270,in=90] (-.75,.5);
	\draw[very thick] (-1.5,-1) -- (-1.5,2);	
	\draw[very thick] (-2,0) -- (-2,3);
	\draw[very thick] (-2.5,1) -- (-2.5,4);	
	\draw[very thick] (2,3) -- (2,0);
	\draw[very thick] (2,3) to [out=180,in=0] (.75,3.25);
	\draw[very thick] (.75,3.25) to [out=225,in=0] (-.75,2.5);
	\draw[very thick] (.75,3.25) to [out=135,in=0] (-2.5,4);
	\draw[very thick] (-.75,2.5) to [out=135,in=0] (-2,3);
	\draw[very thick] (-.75,2.5) to [out=225,in=0] (-1.5,2);
	\node[red,opacity=1] at (-2.25,3.75) {\tiny $c$};
	\node[red,opacity=1] at (-1.75,2.75) {\tiny $b$};	
	\node[red,opacity=1] at (-1.25,1.75) {\tiny $a$};
	\node[red,opacity=1] at (-.125,2.25) {\tiny $_{a+b}$};
	\node[red,opacity=1] at (-.125,.75) {\tiny $_{b+c}$};
\end{tikzpicture}
\quad, \quad
\begin{tikzpicture} [xscale=-.5,yscale=.5,fill opacity=0.2,anchorbase]
	\path[fill=red] (-2.5,4) to [out=0,in=135] (-.75,3.5) to [out=270,in=90] (.75,.25)
		to [out=135,in=0] (-2.5,1);
	\path[fill=red] (-.75,3.5) to [out=270,in=125] (.29,1.5) to [out=55,in=270] (.75,2.75) 
		to [out=135,in=0] (-.75,3.5);
	\path[fill=red] (-.75,-.5) to [out=90,in=235] (.29,1.5) to [out=315,in=90] (.75,.25) 
		to [out=225,in=0] (-.75,-.5);
	\path[fill=red] (-2,3) to [out=0,in=225] (-.75,3.5) to [out=270,in=125] (.29,1.5)
		to [out=235,in=90] (-.75,-.5) to [out=135,in=0] (-2,0);
	\path[fill=red] (-1.5,2) to [out=0,in=225] (.75,2.75) to [out=270,in=90] (-.75,-.5)
		to [out=225,in=0] (-1.5,-1);
	\path[fill=red] (2,3) to [out=180,in=0] (.75,2.75) to [out=270,in=55] (.29,1.5)
		to [out=305,in=90] (.75,.25) to [out=0,in=180] (2,0);
	\draw[very thick] (2,0) to [out=180,in=0] (.75,.25);
	\draw[very thick] (.75,.25) to [out=225,in=0] (-.75,-.5);
	\draw[very thick] (.75,.25) to [out=135,in=0] (-2.5,1);
	\draw[very thick] (-.75,-.5) to [out=135,in=0] (-2,0);
	\draw[very thick] (-.75,-.5) to [out=225,in=0] (-1.5,-1);
	\draw[very thick, red, directed=.85] (-.75,3.5) to [out=270,in=90] (.75,.25);
	\draw[very thick, red, directed=.75] (.75,2.75) to [out=270,in=90] (-.75,-.5);	
	\draw[very thick] (-1.5,-1) -- (-1.5,2);	
	\draw[very thick] (-2,0) -- (-2,3);
	\draw[very thick] (-2.5,1) -- (-2.5,4);	
	\draw[very thick] (2,3) -- (2,0);
	\draw[very thick] (2,3) to [out=180,in=0] (.75,2.75);
	\draw[very thick] (.75,2.75) to [out=135,in=0] (-.75,3.5);
	\draw[very thick] (.75,2.75) to [out=225,in=0] (-1.5,2);
	\draw[very thick]  (-.75,3.5) to [out=225,in=0] (-2,3);
	\draw[very thick]  (-.75,3.5) to [out=135,in=0] (-2.5,4);
	\node[red,opacity=1] at (-2.25,3.75) {\tiny$c$};
	\node[red,opacity=1] at (-1.75,2.75) {\tiny$b$};	
	\node[red,opacity=1] at (-1.25,1.75) {\tiny$a$};
	\node[red,opacity=1] at (0,2.75) {\tiny$_{b+c}$};
	\node[red,opacity=1] at (0,.25) {\tiny$_{a+b}$};
\end{tikzpicture}
\quad , \quad
\begin{tikzpicture} [xscale=-.5,yscale=.5,fill opacity=0.2,anchorbase]
	\path[fill=red] (-2.5,4) to [out=0,in=135] (.75,3.25) to [out=270,in=90] (-.75,.5)
		 to [out=135,in=0] (-2.5,1);
	\path[fill=red] (-.75,2.5) to [out=270,in=125] (-.35,1.5) to [out=45,in=270] (.75,3.25) 
		to [out=225,in=0] (-.75,2.5);
	\path[fill=red] (-.75,.5) to [out=90,in=235] (-.35,1.5) to [out=315,in=90] (.75,-.25) 
		to [out=135,in=0] (-.75,.5);	
	\path[fill=red] (-2,3) to [out=0,in=135] (-.75,2.5) to [out=270,in=125] (-.35,1.5) 
		to [out=235,in=90] (-.75,.5) to [out=225,in=0] (-2,0);
	\path[fill=red] (-1.5,2) to [out=0,in=225] (-.75,2.5) to [out=270,in=90] (.75,-.25)
		to [out=225,in=0] (-1.5,-1);
	\path[fill=red] (2,3) to [out=180,in=0] (.75,3.25) to [out=270,in=45] (-.35,1.5) 
		to [out=315,in=90] (.75,-.25) to [out=0,in=180] (2,0);				
	\draw[very thick] (2,0) to [out=180,in=0] (.75,-.25);
	\draw[very thick] (.75,-.25) to [out=135,in=0] (-.75,.5);
	\draw[very thick] (.75,-.25) to [out=225,in=0] (-1.5,-1);
	\draw[very thick]  (-.75,.5) to [out=225,in=0] (-2,0);
	\draw[very thick]  (-.75,.5) to [out=135,in=0] (-2.5,1);	
	\draw[very thick, red, directed=.75] (-.75,2.5) to [out=270,in=90] (.75,-.25);
	\draw[very thick, red, directed=.85] (.75,3.25) to [out=270,in=90] (-.75,.5);
	\draw[very thick] (-1.5,-1) -- (-1.5,2);	
	\draw[very thick] (-2,0) -- (-2,3);
	\draw[very thick] (-2.5,1) -- (-2.5,4);	
	\draw[very thick] (2,3) -- (2,0);
	\draw[very thick] (2,3) to [out=180,in=0] (.75,3.25);
	\draw[very thick] (.75,3.25) to [out=225,in=0] (-.75,2.5);
	\draw[very thick] (.75,3.25) to [out=135,in=0] (-2.5,4);
	\draw[very thick] (-.75,2.5) to [out=135,in=0] (-2,3);
	\draw[very thick] (-.75,2.5) to [out=225,in=0] (-1.5,2);
	\node[red,opacity=1] at (-2.25,3.75) {\tiny $c$};
	\node[red,opacity=1] at (-1.75,2.75) {\tiny $b$};	
	\node[red,opacity=1] at (-1.25,1.75) {\tiny $a$};
	\node[red,opacity=1] at (-.125,2.25) {\tiny $_{a+b}$};
	\node[red,opacity=1] at (-.125,.75) {\tiny $_{b+c}$};
\end{tikzpicture}
\] Analogous graphical interpretations are available for isomorphisms relating
boundaries of squares formed by edges in non-adjacent coordinate axes.
        \end{rem}

        Alternating compositions of induction and restriction functors satisfy relations
that are controlled by so-called \emph{square switch isomorphisms}. The following
proposition is a consequence of the {S}to{\v{s}i\'c} formula in the
categorified quantum group for $\slnn{2}$ \cite[Theorem 5.6]{KLMS} acting on
singular Soergel bimodules via foams \cite{QR}, see \cite[Appendix A]{HRW1} for
a discussion.

\begin{prop}[Square switch]\label{prop:squareswitch} For all non-negative integers
$a,b,c,d$ with $a+b=c+d$ we have isomorphisms
\begin{align*}
	\begin{tikzpicture}[smallnodes,rotate=90,baseline=.1em,scale=.66]
		\draw[very thick] (0,.25) to [out=150,in=270] (-.25,1) node[left,xshift=2pt]{$c$};
		\draw[very thick] (.5,.5) to (.5,1) node[left,xshift=2pt]{$d$};
		\draw[very thick] (0,.25) to  (.5,.5);
		\draw[very thick] (0,-.25) to (0,.25);
		\draw[very thick] (.5,-.5) to [out=30,in=330]  (.5,.5);
		\draw[very thick] (0,-.25) to  (.5,-.5);
		\draw[very thick] (.5,-1) node[right,xshift=-2pt]{$b$} to (.5,-.5);
		\draw[very thick] (-.25,-1)node[right,xshift=-2pt]{$a$} to [out=90,in=210] (0,-.25);
		\node at (0.2, -.65) {$s$};
		\node at (0.2, .65) {$r$};
		\end{tikzpicture}
		&\cong
		\bigoplus_{t=0}^{\min(r,s)} 
		\left(\begin{tikzpicture}[smallnodes,rotate=90,xscale=-1,baseline=-.4em,scale=.66]
			\draw[very thick] (0,.25) to [out=150,in=270] (-.25,1) node[left,xshift=2pt]{$d$};
			\draw[very thick] (.5,.5) to (.5,1) node[left,xshift=2pt]{$c$};
			\draw[very thick] (0,.25) to  (.5,.5);
			\draw[very thick] (0,-.25) to (0,.25);
			\draw[very thick] (.5,-.5) to [out=30,in=330](.5,.5);
			\draw[very thick] (0,-.25) to  (.5,-.5);
			\draw[very thick] (.5,-1) node[right,xshift=-2pt]{$a$} to (.5,-.5);
			\draw[very thick] (-.25,-1)node[right,xshift=-2pt]{$b$} to [out=90,in=210] (0,-.25);
			\node at (0.2, .85) {$s{-}t$};
			\node at (0.2, -.85) {$r{-}t$};
			\end{tikzpicture}
		 \right)^{\oplus \genfrac[]{0pt}{2}{b-c}{t}}\qquad \text{if } b\geq c \\
		 \begin{tikzpicture}[smallnodes,rotate=90,xscale=-1,baseline=-.4em,scale=.66]
			\draw[very thick] (0,.25) to [out=150,in=270] (-.25,1) node[left,xshift=2pt]{$d$};
			\draw[very thick] (.5,.5) to (.5,1) node[left,xshift=2pt]{$c$};
			\draw[very thick] (0,.25) to  (.5,.5);
			\draw[very thick] (0,-.25) to (0,.25);
			\draw[very thick] (.5,-.5) to [out=30,in=330](.5,.5);
			\draw[very thick] (0,-.25) to  (.5,-.5);
			\draw[very thick] (.5,-1) node[right,xshift=-2pt]{$a$} to (.5,-.5);
			\draw[very thick] (-.25,-1)node[right,xshift=-2pt]{$b$} to [out=90,in=210] (0,-.25);
			\node at (0.2, .65) {$r$};
			\node at (0.2, -.65) {$s$};
			\end{tikzpicture}
			&\cong
			\bigoplus_{t=0}^{\min(r,s)} 
			\left(\begin{tikzpicture}[smallnodes,rotate=90,baseline=.1em,scale=.66]
				\draw[very thick] (0,.25) to [out=150,in=270] (-.25,1) node[left,xshift=2pt]{$c$};
				\draw[very thick] (.5,.5) to (.5,1) node[left,xshift=2pt]{$d$};
				\draw[very thick] (0,.25) to  (.5,.5);
				\draw[very thick] (0,-.25) to (0,.25);
				\draw[very thick] (.5,-.5) to [out=30,in=330]  (.5,.5);
				\draw[very thick] (0,-.25) to  (.5,-.5);
				\draw[very thick] (.5,-1) node[right,xshift=-2pt]{$b$} to (.5,-.5);
				\draw[very thick] (-.25,-1)node[right,xshift=-2pt]{$a$} to [out=90,in=210] (0,-.25);
				\node at (0.2, -.85) {$r{-}t$};
				\node at (0.2, .85) {$s{-}t$};
				\end{tikzpicture}
			 \right)^{\oplus\genfrac[]{0pt}{2}{c-b}{t}}\qquad \text{if } b\leq c 
\end{align*}
with graded multiplicities given by balanced $\qdeg$-binomial coefficients.
\end{prop}

\begin{exa}By setting $a=c=0$ and $d=b$ in Proposition~\ref{prop:squareswitch} we obtain the following isomorphisms.
	\begin{equation}
		\label{eq:digon}
		\begin{tikzpicture}[smallnodes,rotate=90,anchorbase,scale=.66]
			\draw[very thick] (.5,.5) to (.5,1) node[left,xshift=2pt]{$b$};
			\draw[very thick] (.5,-.5) to [out=30,in=330]  (.5,.5);
			\draw[very thick] (.5,-.5) to [out=150,in=210]  (.5,.5);
			\draw[very thick] (.5,-1) node[right,xshift=-2pt]{$b$} to (.5,-.5);
			\node at (0.4, 0) {$s$};
			\end{tikzpicture}	
			\cong 
			\left(\begin{tikzpicture}[smallnodes,rotate=90,anchorbase,scale=.66]
				\draw[very thick] (.5,.5) to (.5,1) node[left,xshift=2pt]{$b$};
				\draw[very thick] (.5,-1) node[right,xshift=-2pt]{$b$} to (.5,.5);
				\end{tikzpicture}
				\right)^{\oplus\genfrac[]{0pt}{2}{b}{s}}
			\cong
			\qdeg^{-s(b-s)}H^*(\mathrm{Gr}(s,b))\otimes  
			\left(\begin{tikzpicture}[smallnodes,rotate=90,anchorbase,scale=.66]
				\draw[very thick] (.5,.5) to (.5,1) node[left,xshift=2pt]{$b$};
				\draw[very thick] (.5,-1) node[right,xshift=-2pt]{$b$} to (.5,.5);
				\end{tikzpicture}
				\right)
			\end{equation}
The appearance of the Grassmannian cohomology illustrates Remark~\ref{rem:cohflags}.
        \end{exa}

\subsection{Rickard complexes}
By the main result of \cite{LMGRSW}, the $\infty$-categorical homotopy categories of Soergel bimodules $\mathbf{K}^b(\SBim_n)$ for all $n\geq 0$ assemble into an $\EE_2$-monoidal locally stable graded $\Q$-linear $(\infty,2)$-category with an $\EE_2$-monoidal functor to $\mathrm{st}_{\Q}^{B\Z}$:
\[
    H_{\mathrm{loc}}\colon \mathbf{K}_{\mathrm{loc}}^b(\SBim)\to \mathrm{st}_{\Q}^{B\Z}
\] The braiding on $\mathbf{K}_{\mathrm{loc}}^b(\SBim)$ is implemented on the level of 1-morphisms by
Rouquier complexes \cite{0409593}, which can be conveniently described hom-category-wise in the stable graded $\Q$-linear monoidal dg categories of chain complexes of Soergel bimodules $\Ch^b(\SBim_n)$ for all $n\geq 0$ \cite{stroppel2024braidingtypesoergelbimodules}.

On general grounds, we expect all of these results to extend to the setting of
singular Soergel bimodules with braiding 1-morphisms generated by \emph{Rickard
complexes}, which we are now going to describe as $1$-morphisms in the locally stable graded $\Q$-linear dg 2-categories of chain complexes of singular Soergel bimodules $\Ch_{\mathrm{loc}}^b(\sSBim_n)$. For two compositions $c,c'$ of $n$, we use the notation ${}_{c'}\sSBim_c$ for the hom-category of morphisms from $c$ to $c'$ in $\sSBim_n$ and 
\[{}_{c'}\CS_c:= {}_{c'}\Ch^b(\sSBim_n)_c\] for the corresponding hom dg category in $\Ch_{\mathrm{loc}}^b(\sSBim_n)$.


\begin{definition} Let $ab$ and $cd$ be two compositions of $n\geq 0$ and set
    $\ell = a-d= c-b$. The \emph{$\ell$-shifted Rickard complex} ${}_{cd}\sRick_{ab}\in {}_{cd}\CS_{ab}$ is the bounded complex of singular Soergel bimodules
    \begin{align*}
        {}_{cd}\sRick_{ab}
    :=& 
    \left( \begin{tikzpicture}[smallnodes,rotate=90,baseline=.1em,scale=.66]
    \draw[very thick] (0,.25) to [out=150,in=270] (-.25,1) node[left,xshift=2pt]{$c$};
        \draw[very thick] (.5,.5) to (.5,1) node[left,xshift=2pt]{$d$};
    \draw[very thick] (0,.25) to  (.5,.5);
    \draw[very thick] (0,-.25) to (0,.25);
    \draw[dotted] (.5,-.5) to [out=30,in=330] node[above=-2pt]{$0$} (.5,.5);
    \draw[very thick] (0,-.25) to  (.5,-.5);
    \draw[very thick] (.5,-1) node[right,xshift=-2pt]{$b$} to (.5,-.5);
    \draw[very thick] (-.25,-1)node[right,xshift=-2pt]{$a$} to [out=90,in=210] (0,-.25);
    \end{tikzpicture}
    \xrightarrow{\; \chi_0^+}
    \qdeg^{-(a-d+1)} \tdeg
    \begin{tikzpicture}[smallnodes,rotate=90,baseline=.1em,scale=.66]
    \draw[very thick] (0,.25) to [out=150,in=270] (-.25,1) node[left,xshift=2pt]{$c$};
    \draw[very thick] (.5,.5) to (.5,1) node[left,xshift=2pt]{$d$};
    \draw[very thick] (0,.25) to  (.5,.5);
    \draw[very thick] (0,-.25) to (0,.25);
    \draw[very thick] (.5,-.5) to [out=30,in=330] node[above=-2pt]{$1$} (.5,.5);
    \draw[very thick] (0,-.25) to  (.5,-.5);
    \draw[very thick] (.5,-1) node[right,xshift=-2pt]{$b$} to (.5,-.5);
    \draw[very thick] (-.25,-1)node[right,xshift=-2pt]{$a$} to [out=90,in=210] (0,-.25);
    \end{tikzpicture}
    \xrightarrow{\; \chi_0^+}
    \qdeg^{-2(a-d+1)} \tdeg^2
    \begin{tikzpicture}[smallnodes,rotate=90,baseline=.1em,scale=.66]
    \draw[very thick] (0,.25) to [out=150,in=270] (-.25,1) node[left,xshift=2pt]{$c$};
    \draw[very thick] (.5,.5) to (.5,1) node[left,xshift=2pt]{$d$};
    \draw[very thick] (0,.25) to  (.5,.5);
    \draw[very thick] (0,-.25) to (0,.25);
    \draw[very thick] (.5,-.5) to [out=30,in=330] node[above=-2pt]{$2$} (.5,.5);
    \draw[very thick] (0,-.25) to  (.5,-.5);
    \draw[very thick] (.5,-1) node[right,xshift=-2pt]{$b$} to (.5,-.5);
    \draw[very thick] (-.25,-1)node[right,xshift=-2pt]{$a$} to [out=90,in=210] (0,-.25);
    \end{tikzpicture}
    \xrightarrow{\; \chi_0^+}
     \cdots \right)
    \end{align*}
    which terminates after $\min(b,d)+1$ steps. The bimodule homomorphisms $\chi_0^+$
that serve as components of the differential are uniquely determined up to a
unit scalar by their degrees and their non-vanishing. They can be constructed
from the units and counits of the ambidextrous adjunctions and succinctly
described by the foam in Figure~\ref{fig:diffslices}, see \cite[Section 3.4 and
Figure 1]{HRW2}.

    For $\ell=0$, i.e. if $(c,d)=(b,a)$, the shifted Rickard complex
    coindices with the (ordinary) Rickard complex, see e.g. \cite[Definition
    2.23]{HRW1}, for which we use the diagrammatic notation
    \[
        \left\llbracket
        \begin{tikzpicture}[rotate=90,scale=.5,smallnodes,anchorbase,yscale=1]
        \draw[very thick] (1,0) node[right,xshift=2pt]{$b$} \pu (0,2) node[left,xshift=-2pt]{$b$};
        \draw[line width=5pt,color=white] (0,0) \pu (1,2) ;
        \draw[very thick] (0,0) node[right,xshift=2pt]{$a$} \pu (1,2) node[left,xshift=-2pt]{$a$};
        \end{tikzpicture}
        \right\rrbracket	
        :=
        {}_{ba}\Rick_{ab} :=
        {}_{ba}\sRick_{ab}
    \]
    in which the double brackets serve as a reminder that crossing are to be parsed as a chain complex of bimodules. We also use double brackets to evaluate composite diagrams of crossings, merge webs, and split webs to chain complexes of singular Soergel bimodules by using the natural extension of the horizontal composition and tensor product $\boxtimes$ of $\sSBim$ to chain complexes thereover, see \eqref{eq:forkslide} for an example.  
    
\end{definition}

\begin{figure}[h]
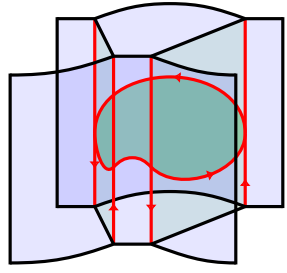

	\[\diffslices\]
\caption{The type of foam corresponding to $\chi_0^+$.}
\label{fig:diffslices}
\end{figure}
    
\begin{rem}\label{rem:Rickardbraidandforkslide}
    The (unshifted) Rickard complexes are invertible up to homotopy and yield an
    action of the $\N$-colored braid group \cite[Proposition 2.25]{HRW1}. They
    also satisfy certain naturality properties expected of braiding 1-morphisms in an $\mathbb{E}_2$-monoidal $(\infty,2)$-category, such as

    \begin{equation}\label{eq:forkslide}
        \left\llbracket
\begin{tikzpicture}[rotate=90,scale=.5,smallnodes,anchorbase,yscale=1]
\draw[very thick] (-.5,-1) node[right,xshift=2pt]{$b{+}c$} to [out=90,in=270] (.5,1);
\draw[very thick] (.5,1) to [out=30,in=270] (1,2) node[left,xshift=-2pt]{$c$};
\draw[very thick] (.5,1) to [out=150,in=270] (0,2) node[left,xshift=-2pt]{$b$};
\draw[line width=5pt,color=white] (.5,-1) to [out=90,in=270] (-1,2);
\draw[very thick] (.5,-1) node[right,xshift=2pt]{$a$} to [out=90,in=270] 
    (-1,2) node[left,xshift=-2pt]{$a$};
\end{tikzpicture}
\right\rrbracket	
\simeq
\left\llbracket
\begin{tikzpicture}[rotate=90,scale=.5,smallnodes,anchorbase,yscale=1]
\draw[very thick] (-.5,-1) node[right,xshift=2pt]{$b{+}c$} to (-.5,-.25);
\draw[very thick] (-.5,-.25) to [out=30,in=270] (1,2) node[left,xshift=-2pt]{$c$};
\draw[very thick] (-.5,-.25) to [out=150,in=270] (0,2) node[left,xshift=-2pt]{$b$};
\draw[line width=5pt,color=white] (.5,-1) to [out=90,in=270] (-1,2);
\draw[very thick] (.5,-1) node[right,xshift=2pt]{$a$} to [out=90,in=270] 
    (-1,2) node[left,xshift=-2pt]{$a$};
\end{tikzpicture}
\right\rrbracket
    \end{equation}
    see \cite[Proposition 2.27]{HRW1}. In fact, such \emph{fork-sliding} relations are the main ingredient for a lift of the main result of \cite{LMGRSW} to singular Soergel bimodules, the details of which will appear elsewhere.

Additional relations
\begin{equation}\label{eq:twistzipper}
	\left\llbracket
\begin{tikzpicture}[rotate=90,scale=.5,smallnodes,anchorbase]
	\draw[very thick] (1,-1) node[right,xshift=-2pt]{$b$} to [out=90,in=270] (0,1)
		to [out=90,in=210] (.5,2);
	\draw[line width=5pt,color=white] (0,-1) to [out=90,in=270] (1,1);
	\draw[very thick] (0,-1) node[right,xshift=-2pt]{$a$} to [out=90,in=270] (1,1)
		to [out=90,in=330] (.5,2);
	\draw[very thick] (.5,2) to (.5,2.75);
\end{tikzpicture}
\right\rrbracket		
\simeq
\qdeg^{a b}
\left\llbracket
\begin{tikzpicture}[rotate=90,scale=.5,smallnodes,anchorbase]
	\draw[very thick] (0,1) node[right,xshift=-2pt]{$a$} to [out=90,in=210] (.5,2);
	\draw[very thick] (1,1) node[right,xshift=-2pt]{$b$} to [out=90,in=330] (.5,2);
	\draw[very thick] (.5,2) to (.5,2.75);
\end{tikzpicture}
\right\rrbracket
\quad, \quad
\left\llbracket
\begin{tikzpicture}[rotate=90,scale=-.5,smallnodes,anchorbase]
	\draw[very thick] (1,-1) node[left,xshift=2pt]{$a$} to [out=90,in=270] (0,1)
		to [out=90,in=210] (.5,2);
	\draw[line width=5pt,color=white] (0,-1) to [out=90,in=270] (1,1);
	\draw[very thick] (0,-1) node[left,xshift=2pt]{$b$} to [out=90,in=270] (1,1)
		to [out=90,in=330] (.5,2);
	\draw[very thick] (.5,2) to (.5,2.75);
\end{tikzpicture}
\right\rrbracket		
\simeq
\qdeg^{a b}
\left\llbracket
\begin{tikzpicture}[rotate=90,scale=-.5,smallnodes,anchorbase]
	\draw[very thick] (0,1) node[left,xshift=2pt]{$b$} to [out=90,in=210] (.5,2);
	\draw[very thick] (1,1) node[left,xshift=2pt]{$a$} to [out=90,in=330] (.5,2);
	\draw[very thick] (.5,2) to (.5,2.75);
\end{tikzpicture}
\right\rrbracket	
\end{equation}
furthermore suggest that the $\mathbb{E}_2$-structure can be upgraded to an $f\mathbb{E}_2$-structure with the categorified balancing acting on the object $n$ by the equivalence $\qdeg^{n(n-1)}\tdeg^{n-1} \id_n$.
\end{rem}

The shifted Rickard complexes can be expressed using composite diagrams following \cite[Proposition 2.31]{HRW1}, which we recall here.

\begin{prop} For all natural numbers $a,b,c,d$ such that $a+b=c+d$, there exists a homotopy equivalences of complexes of singular Soergel bimodules

\begin{equation}
	\label{eqn:MCSred}
    {}_{cd}\sRick_{ab} 
    \; \simeq \;
    \left\llbracket
\begin{tikzpicture}[scale=.4,smallnodes,anchorbase,rotate=270]
\draw[very thick] (1,-1) to [out=150,in=270] (0,1) to (0,2) node[right=-2pt]{$b$}; 
\draw[line width=5pt,color=white] (0,-2) to (0,-1) to [out=90,in=210] (1,1);
\draw[very thick] (0,-2) node[left=-2pt]{$d$} to (0,-1) to [out=90,in=210] (1,1);
\draw[very thick] (1,1) to (1,2) node[right=-2pt]{$a$};
\draw[very thick] (1,-2) node[left=-2pt]{$c$} to (1,-1); 
\draw[very thick] (1,-1) to [out=30,in=330] node[below=-1pt]{$a-d$} (1,1); 
\end{tikzpicture}
\right\rrbracket 
\end{equation}
where the diagram on the right-hand side is understood as zero in case $\ell=a-d< 0$.
\end{prop}
As in \cite{HRW1}, we will use the suggestive notation ${}_{cd}\MCS_{ab}$ for the right-hand side of \eqref{eqn:MCSred}: $M$ for \emph{merge}, $C$ for
\emph{crossing}, $S$ for \emph{split}. For the left-hand side we similarly set ${}_{cd}\MCSmin_{ab}:={}_{cd}\sRick_{ab}$, letting $\textsf{sans serif font}$ indicate a  homotopy equivalent ``simplified'' complex.

We record the special case for $a=c=0$ and $n=b=d>0$ and hence $\ell=-n<0$.

\begin{cor}\label{cor:digoncomplex} Let $n\geq 1$, then the \emph{complex of digons}
    \begin{align*}
        {}_{0n}\sRick_{0n}
    =& 
    \left( 
        \begin{tikzpicture}[smallnodes,rotate=0,baseline=.1em,scale=.66]
            \draw[very thick] (-1,0) to (-.5,0);
            \draw[dotted] (-.5,0) to [out=60,in=120] node[above]{$0$} (.5,0);
            \draw[very thick] (-.5,0) to [out=300,in=240] node[below]{$n$} (.5,0);
            \draw[very thick] (.5,0) to (1,0);
            \end{tikzpicture}
        \xrightarrow{\; \chi_0^+}\cdots \xrightarrow{\; \chi_0^+}
    \qdeg^{s(n-1)} \tdeg^s\;
    \begin{tikzpicture}[smallnodes,rotate=0,baseline=.1em,scale=.66]
    \draw[very thick] (-1,0) to (-.5,0);
    \draw[very thick] (-.5,0) to [out=60,in=120] node[above]{$s$} (.5,0);
    \draw[very thick] (-.5,0) to [out=300,in=240] node[below]{$n-s$} (.5,0);
    \draw[very thick] (.5,0) to (1,0);
    \end{tikzpicture}
    \xrightarrow{\; \chi_0^+}\cdots \xrightarrow{\; \chi_0^+}
    \qdeg^{n(n-1)} \tdeg^n\;
    \begin{tikzpicture}[smallnodes,rotate=0,baseline=.1em,scale=.66]
    \draw[very thick] (-1,0) to (-.5,0);
    \draw[very thick] (-.5,0) to [out=60,in=120] node[above]{$n$} (.5,0);
    \draw[dotted] (-.5,0) to [out=300,in=240] node[below]{$0$} (.5,0);
    \draw[very thick] (.5,0) to (1,0);
    \end{tikzpicture} 
    \right) 
    \end{align*}
    is contractible.    
\end{cor}

Let $\Comp_m(n)\subset \Comp(n)$ denote the set of
compositions of $n$ with exactly (positive) $m$ parts. For $\underline{k}\in C$ we consider
the singular Bott--Samelson bimodule corresponding to the (shifted) web
\begin{equation}
	\label{eq:explodedweb}
{}_{n} {W^{\underline{k}}}_{n} :=  \qdeg^{\sum_{i=1}^{m-1}{(k_1+\cdots+k_i)k_{i+1}}}\; 
\begin{tikzpicture}[smallnodes,rotate=0,baseline=.1em,scale=1]
    \draw[very thick] (-1,0) node[left,xshift=-1pt]{$n$} to (-.5,0);
    \draw[very thick] (-.5,0) to [out=60,in=120] node[above]{$k_m$} (.5,0);
    \draw[very thick] (-.45,0.05) to [out=300,in=240] (.45,.05);
    \draw[very thick] (-.4,0.1) to [out=300,in=240] node[above]{${\tiny\vdots}$} (.4,.1);
    \draw[very thick] (-.5,0) to [out=300,in=240] node[below]{$k_1$} (.5,0);
    \draw[very thick] (.5,0) to (1,0) node[right,xshift=1pt]{$n$};
    \end{tikzpicture}
\end{equation}
If $1\leq i<m$ and $\partial_i \underline{k} := (k_1,\dots, k_i+k_{i+1},
k_{i+2}, \dots, k_m)$, there exists a unique morphism 
\[
	F_{i,\underline{k}}:= {}_{n} {W^{\partial_i  \underline{k}}}_{n} \to {}_{n} {W^{\underline{k}}}_{n}
\]
composed of adjunction units and (co)associativity morphisms. These
morphisms assemble into a commutative diagram 
\[
	{}_{n} {W^{(a+b)}}_{n} 
	\to \bigoplus_{\underline{k}\in \Comp_2(n)}	{}_{n} {W^{\underline{k}}}_{n} 
	\to \cdots
	\to \bigoplus_{\underline{k}\in \Comp_{n-1}(n)}{}_{n} {W^{\underline{k}}}_{n} 
	\to {}_{n} {W^{(1,\dots, 1)}}_{n} 		
\]
whose total cofiber we denote by ${\Expl}_{n}$.

\begin{prop}\label{prop:expl}
Let $n\geq 1$, then we have an equivalence ${\Expl}_{n} \simeq \qdeg^{n(n-1)} 
\id_n$.
\end{prop}
\begin{proof}
By induction on $n$, using Corollary~\ref{cor:digoncomplex}.
\end{proof}

\subsection{Soergel schober dictionary}

\begin{thm}[Soergel schobers]
    The family of Soergel cubes \eqref{eq:stablecubefamily} defines a
    factorizing family of (graded $\Q$-linear) framed $\AA_n$-schobers in the
    sense of Definitions~\ref{defi:anschober}, \ref{defi:highercotwist},
    \ref{defi:multschober}. The higher twist functors and cotwists are equivalent to grading shifted
    Rickard complexes and grading shifted identity morphisms, respectively:
    \[T_{ab} \simeq \qdeg^{ab}{}_{ba}\Rick_{ab}\quad, \quad  T_{n}\simeq \qdeg^{n(n-1)}\id_n\]

\end{thm}
\begin{proof}[Proof outline]
    We have already observed that the data \eqref{eq:stablecube} satisfies the
    Adjunctability and Far-commutativity conditions of Definition~\ref{defi:anschober} and Recursiveness follows by
    induction on $n$ using the monoidality of the Soergel cubes. 
    We outsource the checks for the Twist invertibility and Defect vanishing conditions into Proposition~\ref{prop:highertwistRickard} and Proposition~\ref{prop:defectzero} respectively. The invertibility of cotwist has already been checked in Proposition~\ref{prop:expl}.
\end{proof}

\begin{prop}
    \label{prop:highertwistRickard}
Let $ab$ be a composition of $n$, then the higher twist functor $T_{ab}$ for
the Soergel cube \eqref{eq:stablecube} is equivalent to tensoring with the
grading shifted Rickard complex $\qdeg^ab{}_{ba}\Rick_{ab}$ and, thus, an
equivalence.
\end{prop}

\begin{prop}
    \label{prop:defectzero}
    For every pair of compositions $ab$, $cd$ of $n$ with $a \ne d$,
            the Beck--Chevalley defect $R_{ab,cd}$ associated to the pullback
            of the Soergel cube \eqref{eq:stablecube} to the bifactorization cube $Q(ab,cd)$ is the zero functor.
\end{prop}
The proof of Propositions~\ref{prop:highertwistRickard} and \ref{prop:defectzero} occupies \S\ref{sec:Soergelschoberbialgebra} and \S\ref{sec:proofbialgebra}. 

\begin{example}
    After a global grading shift, the equivalences \eqref{eq:twistzipper}
    together with $T_{ab}\simeq \qdeg^{ab}{}_{ba}\Rick_{ab}$ and $T_{n}\simeq \qdeg^{n(n-1)}\tdeg^{n-1}\id_n$ yield
    \[
    \mu^{ba}_{a+b}\circ T_{ab} \circ (T_a\otimes T_b) \simeq T_{a+b} \circ \mu^{ab}_{a+b} 
        \quad, \quad
        (T_a\otimes T_b)\circ T_{ba}\circ \Delta_{ba}^{a+b}  \simeq \Delta_{ab}^{a+b} \circ T_{a+b} 
    \]
\end{example}
\begin{remark}
The braiding on complexes of Soergel bimodules from \cite{LMGRSW} is generated
by a Rickard complex normalized by a different shift convention, namely
$\qdeg^{-1}{}_{11}\Rick_{11}\simeq \qdeg^{-2} T_{11}$. With this convention, the
braiding descends to the standard symmetric braiding upon proceeding to derived
categories, see \cite[Remark 2.2.6]{LMGRSW}. For singular Soergel
bimodules, one would expend the braiding to be generated by the shifted Rickard complexes
$\qdeg^{-ab}{}_{ba}\Rick_{ab}\simeq \qdeg^{-2ab} T_{ab}$.
\end{remark}

\begin{example}
    \label{exa:sbima1}
   We check that the Soergel cube $\X_{2}$ from
   Construction~\ref{eq:soergelcube} defines a framed $\AA_1$-schober in the
   sense of Definitions~\ref{defi:a1schober} and~\ref{defi:cotwist}, indeed a
   spherical adjunction~\ref{exa:framed}. The only edge in $\X_{2}$ is labelled
   by the induction bimodule $\!\!\!\!\phantom{()}_{R_{11}}(R_{11})_{R_{2}}$
   with right adjoint given by the restriction bimodule
   $\!\!\!\!\phantom{()}_{R_{2}}(R_{11})_{R_{11}}$, so (A1.1) is satisfied. 
   
   The corresponding twist, the fiber
   of the counit of the adjunction, is given by tensoring with
   \[
   \left( \qdeg^1 \uwave{
    \begin{tikzpicture}[smallnodes,anchorbase, scale=.7]
        \Indc{0} 
        \Resc{1} 
        \node[xshift=-3pt] at (0,.3) {$1$}; 
        \node[xshift=-3pt] at (0,.6) {$1$};
        \node[xshift=3pt] at (2,.3) {$1$}; 
        \node[xshift=3pt] at (2,.6) {$1$};
    \end{tikzpicture} 
   }
        \to	
        \begin{tikzpicture}[smallnodes,scale=.7,anchorbase]
            \draw[very thick] (-1,.66) to (1,.66);
            \draw[very thick] (-1,.33) to (1,.33) ;
            \node[xshift=-3pt] at (-1,.3) {$1$}; 
            \node[xshift=-3pt] at (-1,.6) {$1$};
            \node[xshift=3pt] at (1,.3) {$1$}; 
            \node[xshift=3pt] at (1,.6) {$1$};
            \end{tikzpicture} 
            \right)
            \simeq 
            \qdeg^{1}{}_{11}\Rick_{11}
   \]
   which is a shift of a Rickard complex, in this case also known as a Rouquier
   complex, and thus defines an equivalence (A1.2). The cofiber of the unit of the adjunction, on the other hand, is
   \[
    \left( \begin{tikzpicture}[smallnodes,scale=.7]
        \draw[very thick] (-1,0) node[left,xshift=-1pt]{$2$} to (1,0) node[right,xshift=1pt]{$2$};
        \end{tikzpicture}    
         \to	
         \qdeg^1 
         \uwave{
            \begin{tikzpicture}[smallnodes,anchorbase, scale=.7]
                \Indc{1} 
                \Resc{0} 
                \node[xshift=-3pt] at (0,.5) {$2$}; 
                \node[xshift=3pt] at (2,.5) {$2$}; 
            \end{tikzpicture} 
            }
            \right)
             \simeq 
             \qdeg^{2} \id_2
    \]
    a shift of the identity as easy special case of Proposition~\ref{prop:expl},
    and thus an equivalence.
\end{example}

\begin{example}
    \label{exa:sbima2}
    Considering Soergel cube $\X_{3}$ from
    Construction~\ref{eq:soergelcube}, we find that it defines a framed $\AA_2$-schober.
The edges of the cube correspond to the induction bimodules
\[
I:=\begin{tikzpicture}[smallnodes,anchorbase, scale=.7]
    \Indc{0} 
    \node[xshift=-3pt] at (0,.33) {$2$}; 
    \node[xshift=-3pt] at (0,.66) {$1$};
\end{tikzpicture}
\quad,\quad	
H:=\begin{tikzpicture}[smallnodes,anchorbase, scale=.7]
    \Indc{0} 
    \node[xshift=-3pt] at (0,.33) {$1$}; 
    \node[xshift=-3pt] at (0,.66) {$2$};
\end{tikzpicture}	
\quad,\quad	
F:=\begin{tikzpicture}[smallnodes,anchorbase, scale=.7]
    \Inda{0}
    \node[xshift=-3pt] at (0,0) {$1$}; 
    \node[xshift=-3pt] at (0,.5) {$1$};
    \node[xshift=-3pt] at (0,1) {$1$};
\end{tikzpicture}
\quad,\quad	
G:=\begin{tikzpicture}[smallnodes,anchorbase, scale=.7]
    \Indb{0}
    \node[xshift=-3pt] at (0,0) {$1$}; 
    \node[xshift=-3pt] at (0,.5) {$1$};
    \node[xshift=-3pt] at (0,1) {$1$};
\end{tikzpicture}
\]
The corresponding restriction bimodules are right adjoint (A2.1), $F$ and $G$
each define an $\AA_1$-schober as in Example~\ref{exa:sbima1} (A2.2). The fibers of the Beck-Chevalley map $H I^* \to G^* F$ and $I H^* \to F^* G$ agree on the nose with the shifted Rickard complexes $\qdeg{2}{}_{12}C_{21}$ and $\qdeg{2}{}_{21}C_{12}$, respectively, hence are equivalences (A2.3-4). Finally, for (A2.5-6) we consider the Beck-Chevalley square
\[
        \begin{tikzcd}
            \qdeg^{2} \begin{tikzpicture}[smallnodes,anchorbase, scale=.7]
                \Indc{0} 
                \Resc{1} 
                \node[xshift=-3pt] at (0,.3) {$1$}; 
                \node[xshift=-3pt] at (0,.6) {$2$};
                \node[xshift=3pt] at (2,.3) {$1$}; 
                \node[xshift=3pt] at (2,.6) {$2$};
            \end{tikzpicture} 
            \ar{r} \ar{d} & 
            \qdeg^{2} \begin{tikzpicture}[smallnodes,anchorbase, scale=.7]
                \Indb{3} 
                \Resa{2}
                \Inda{1}
                \Resb{0} 
                \node[xshift=-3pt] at (0,.3) {$1$}; 
                \node[xshift=-3pt] at (0,.6) {$2$};
                \node[xshift=3pt] at (4,.3) {$1$}; 
                \node[xshift=3pt] at (4,.6) {$2$};
            \end{tikzpicture}
            \ar{d} 
            \\
            \begin{tikzpicture}[smallnodes,scale=.7,anchorbase]
                \draw[very thick] (-1,.66) to (1,.66);
                \draw[very thick] (-1,.33) to (1,.33) ;
                \node[xshift=-3pt] at (-1,.3) {$1$}; 
                \node[xshift=-3pt] at (-1,.6) {$2$};
                \node[xshift=3pt] at (1,.3) {$1$}; 
                \node[xshift=3pt] at (1,.6) {$2$};
                \end{tikzpicture} 
            \ar{r} & 
            \qdeg^{2}\begin{tikzpicture}[smallnodes,anchorbase, scale=.7]
                \Indb{1} 
                \Resb{0} 
                \node[xshift=-3pt] at (0,.3) {$1$}; 
                \node[xshift=-3pt] at (0,.6) {$2$};
                \node[xshift=3pt] at (2,.3) {$1$}; 
                \node[xshift=3pt] at (2,.6) {$2$};
            \end{tikzpicture}
        \end{tikzcd}
\]
and its symmetric analog. These can be shown to be biCartesian by inspecting the
morphisms implementing the square switch relations from 
Proposition~\ref{prop:squareswitch}. The cotwists are seen to be shifted
identity morphisms by Proposition~\ref{prop:expl}.
\end{example}

\begin{example}
    \label{exa:sbimaa1}
Exactly as in Example~\ref{exa:sbima2} it is clear from Definition~\ref{defi:bifactorizationcube}.2 that the higher twist
functors $T_{a1}$ for the Soergel schobers are manifestly (shifted) Rickard complexes:
\[
    \left( \qdeg^a \uwave{
        \begin{tikzpicture}[smallnodes,anchorbase, scale=.7]
        \Indc{0} 
        \Resc{1} 
        \node[xshift=-3pt] at (0,.3) {$1$}; 
        \node[xshift=-3pt] at (0,.6) {$a$};
        \node[xshift=3pt] at (2,.3) {$a$}; 
        \node[xshift=3pt] at (2,.6) {$1$};
    \end{tikzpicture} 
    }
         \to	
         \qdeg^{a-1}
         \begin{tikzpicture}[smallnodes,anchorbase, scale=.7]
            \Inda{1}
            \Resb{0} 
            \node[xshift=-3pt] at (0,.3) {$1$}; 
            \node[xshift=-3pt] at (0,.6) {$a$};
            \node[xshift=3pt] at (2,.3) {$a$}; 
            \node[xshift=3pt] at (2,.6) {$1$};
        \end{tikzpicture}
        \right)
             \simeq 
             \qdeg^{a}{}_{1a}\Rick_{a1}
    \]
The first Rickard complex of length greater than one appears as shift of the twist $T_{22}$. Comparing with Example~\ref{exa:a3schober}.4.(b), we consider the square:
\[
        \begin{tikzcd}
            \qdeg^{4} \begin{tikzpicture}[smallnodes,anchorbase, scale=.7]
                \draw[very thick] (0,.25) to [out=0,in=225] (0.5,.5) to (1,.5);
                \draw[very thick] (0,.75) to [out=0,in=135] (0.5,0.5);
                \draw[very thick] (2,.25) to [out=180,in=315] (0.5+1,.5) to (0+1,.5);
                \draw[very thick] (2,.75) to [out=180,in=45] (0.5+1,0.5);
                \node[xshift=-3pt] at (0,.2) {$1$}; 
                \node[xshift=-3pt] at (0,.7) {$2$};
                \node[xshift=3pt] at (2,.2) {$1$}; 
                \node[xshift=3pt] at (2,.7) {$2$};
            \end{tikzpicture} 
            \ar{r} \ar{d} & 
            \qdeg^{3} \begin{tikzpicture}[smallnodes,anchorbase, scale=.7]
                \draw[very thick] (1.5,.33) to (1,.33) to [out=180,in=315] (0.5,.5) to (0,.5);
                \draw[very thick] (1.5,.75) to (1,.75) to [out=225,in=45] (0.5,0.5);
                \draw[very thick] (1,.75) to [out=135,in=315] (0.5,1) to (0,1);
                \draw[very thick] (1.5,1.16) to (1,1.16) to [out=180,in=45] (0.5,1);
                \draw[very thick] (0+1.5,.33) to [out=0,in=225] (0.5+1.5,.5) to (1+1.5,.5);
                \draw[very thick] (0+1.5,.75) to [out=315,in=135] (0.5+1.5,0.5);
                \draw[very thick] (0+1.5,.75) to [out=45,in=225] (0.5+1.5,1) to (1+1.5,1);
                \draw[very thick] (0+1.5,1.16) to [out=0,in=135] (0.5+1.5,1);
                \node[xshift=-3pt] at (0,.4) {$2$}; 
                \node[xshift=-3pt] at (0,.9) {$2$};
                \node[xshift=3pt] at (2.5,.4) {$2$}; 
                \node[xshift=3pt] at (2.5,.9) {$2$};
            \end{tikzpicture}
            \ar{d} 
            \\
            \begin{tikzpicture}[smallnodes,scale=.7,anchorbase]
                \draw[very thick] (-1,.25) to (1,.25);
                \draw[very thick] (-1,.75) to (1,.75) ;
                \node[xshift=-3pt] at (-1,.2) {$2$}; 
                \node[xshift=-3pt] at (-1,.7) {$2$};
                \node[xshift=3pt] at (1,.2) {$2$}; 
                \node[xshift=3pt] at (1,.7) {$2$};
                \end{tikzpicture} 
            \ar{r} & 
            \qdeg^{2}\begin{tikzpicture}[smallnodes,anchorbase, scale=.7]
                \Indc{1} 
                \Resc{0} 
                \node[xshift=-3pt] at (0,.4) {$2$}; 
                \node[xshift=3pt] at (2,.4) {$2$}; 
                \begin{scope}[shift={(0,.5)}]
                \Indc{1} 
                \Resc{0} 
                \node[xshift=-3pt] at (0,.4) {$2$}; 
                \node[xshift=3pt] at (2,.4) {$2$};
                \end{scope} 
            \end{tikzpicture}
        \end{tikzcd}
\]
The total fiber is equivalent to $\qdeg^{4}{}_{22}\Rick_{22}$. To see this, one can apply the square
switch relations from Proposition~\ref{prop:squareswitch} in the right column and then contract terms along the bottom horizontal and right vertical arrows.
\end{example}

In the following section we find a more convenient way of identifying the higher
twists with shifted Rickard complexes.

\subsection{Categorified graded bialgebra relations for Soergel bimodules}
\label{sec:Soergelschoberbialgebra}

We prove Propositions~\ref{prop:highertwistRickard} and ~\ref{prop:defectzero}
by employing the recursion scheme from Theorem~\ref{thm:schoberbialgebra}, see
also Remark~\ref{rem:equivalent}. We first show that Rickard complexes satisfy
the same type of categorified graded bialgebra relations as the higher twists
$T_{ab}$. Starting from the base cases in Examples~\ref{exa:sbima1},
\ref{exa:sbima2}, \ref{exa:sbimaa1}, induction then establishes the equivalences
$T_{ab}\simeq \qdeg^{ab}{}_{ba}\Rick_{ab}$ and $R_{ab,cd}\simeq 0$ for $a\neq d$.

\begin{thm}[Categorified graded bialgebra structure]
	\label{thm:main}
	Let $a,b,c,d$ be non-negative integers, $a+b=c+d$ and (wlog) $b\leq a,c,d$. Then
	there exists a homotopy equivalence of twisted complexes of singular Soergel
	bimodules
	\begin{equation}
	\label{eq:compat}
	\tw_{D}\left(\bigoplus_{s=0}^b 
	\qdeg^{-s(s+a-d)}
	\left\llbracket
	\begin{tikzpicture}[scale=.4,smallnodes,anchorbase,rotate=270]
	\draw[very thick] (1,-1) to [out=150,in=330] (0,1) to (0,2) node[right=-2pt]{$b$}; 
	\draw[line width=5pt,color=white] (0,-2) to (0,-1) to [out=30,in=210] (1,1);
	\draw[very thick] (0,-2) node[left=-2pt]{$d$} to (0,-1) to [out=30,in=210] (1,1);
	\draw[very thick] (1,1) to (1,2) node[right=-2pt]{$a$};
	\draw[very thick] (1,-2) node[left=-2pt]{$c$} to (1,-1); 
	\draw[very thick] (1,-1) to [out=30,in=330]  (1,1); 
	\draw[very thick] (0,-1) to [out=150,in=210]node[above=-1pt]{$s$} (0,1); 
	\end{tikzpicture}
	\right\rrbracket
	\right) 
	\simeq 
	\left\llbracket
	\begin{tikzpicture}[scale=.4,smallnodes,anchorbase,rotate=270]
		\draw[very thick] (.5,.5) to [out=150,in=270] (0,1) to (0,2) node[right=-2pt]{$b$}; 
		\draw[very thick] (0,-2) node[left=-2pt]{$d$} to (0,-1) to [out=90,in=210] (.5,-.5) to (.5,.5);
		\draw[very thick] (.5,.5) to [out=30,in=270]  (1,1) to (1,2) node[right=-2pt]{$a$};
		\draw[very thick] (1,-2) node[left=-2pt]{$c$} to (1,-1) to [out=90,in=330] (.5,-.5); 
		\end{tikzpicture}
		\right\rrbracket
	\end{equation}
	where the twist $D$ is strictly decreasing in $s$. 
\end{thm}
\begin{rem}
    The grading shifts on the left hand side of \eqref{eq:compat} are an effect
    of the shift convention \eqref{eq:merge} for the merge vertex. There would
    be no grading shifts when working with unshifted restriction bimodules.
    \end{rem}

The proof of Theorem~\ref{thm:main} requires two types of dg functors that we now describe.

\begin{definition}
    Let $a,b,c,d,s$ be non-negative integers and $a+b=c+d$. We denote by $\I^{(s)}\colon {}_{cd}\CS_{ab} \rightarrow {}_{c(d+s)}\CS_{a(b+s)}$ 
   the dg functor defined in diagrammatic terms by:
    \[
    \I^{(s)}(X):=
    \begin{tikzpicture}[scale=.5,smallnodes,rotate=90,anchorbase]
        \draw[very thick] (1,2.25) to (1,3) node[left]{$d{+}s$};
        \draw[very thick] (1,-3) node[right]{$b{+}s$} to (1,-2.25); 
        \draw[very thick] (1,-2.25) to [out=30, in=270] (1.75,-1.6) 
            to (1.75,0) node[below,yshift=1pt]{$s$}to (1.75,1.6)  to  [out=90,in=330] (1,2.25); 
        \node[yshift=-2pt] at (0,0) {\normalsize$X$};
        \draw[very thick] (.75,1.6) rectangle (-.75,-1.6);
        \draw[very thick] (1,-2.25) to [out=150,in=270] 
            node[right,xshift=-1pt,yshift=-1pt]{$b$} (.25,-1.6);
        \draw[very thick] (1,2.25) to [out=210,in=90] 
            node[left,xshift=1pt,yshift=-1pt]{$d$} (.25,1.6);
        \draw[very thick] (-.25,3) node[left]{$c$} to (-.25,1.6);
        \draw[very thick] (-.25,-3) node[right]{$a$} to (-.25,-1.6);
    \end{tikzpicture} \, .
    \]
\end{definition}

\begin{example} The terms on the left-hand side of \eqref{eq:compat} arise as
images of complexes ${}_{c(d-s)}\MCS_{a(b-s)}$, illustrated on the right-hand side of \eqref{eqn:MCSred}, under $\I^{(s)}$:
\[{}_{cd}\IMCS^s_{ab} :=   \left\llbracket
\begin{tikzpicture}[scale=.4,smallnodes,anchorbase,rotate=270]
\draw[very thick] (1,-1) to [out=150,in=330] (0,1) to (0,2) node[right=-2pt]{$b$}; 
\draw[line width=5pt,color=white] (0,-2) to (0,-1) to [out=30,in=210] (1,1);
\draw[very thick] (0,-2) node[left=-2pt]{$d$} to (0,-1) to [out=30,in=210] (1,1);
\draw[very thick] (1,1) to (1,2) node[right=-2pt]{$a$};
\draw[very thick] (1,-2) node[left=-2pt]{$c$} to (1,-1); 
\draw[very thick] (1,-1) to [out=30,in=330]  (1,1); 
\draw[very thick] (0,-1) to [out=150,in=210]node[above=-1pt]{$s$} (0,1); 
\end{tikzpicture}
\right\rrbracket
= 
\I^{(s)} \left({}_{c(d-s)}\MCS^s_{a(b-s)}\right)
.\]

\end{example}

In  \cite[Definition 3.1]{HRW1} a dg
functor $K$ was constructed, which takes a complex of singular Soergel bimodules $X\in {}_{ab}\CS_{ab}$ to a Koszul complex associated to the action of certain polynomials. Here we need a trivial variation of this construction.

\begin{definition}\label{def:Koszul cx} 
	For each $X \in {}_{cd}\CS_{ab}$, 
	let $K(X)$ denote the Koszul complex associated to the action of 
	$1,0,\dots,0$ on $X$.
	Explicitly, we consider the bigraded $\Q$-vector space $\largewedge[\xi_1,\dots,\xi_b]$ 
in which the $\xi_i$ are exterior variables with $\wt(\xi_i)=\qdeg^{2i-2b} \tdeg\inv$
and define bimodules
	\[
	K(X) := \tw_{ \id \otimes \xi_b^\ast}\left(X\otimes \largewedge[\xi_1,\dots,\xi_b]\right)  \, .
	\]
	Here, $\xi_i^\ast$ is the endomorphism (in fact, derivation) of
	$\largewedge[\xi_1,\dots,\xi_b]$ with $\wt(\xi_i^{\ast})=\qdeg^{2b-2i} \tdeg^{1}$
	defined by
	\[
	\xi_i^\ast(\xi_i)=1
	\, , \quad
	\xi_i^\ast(\xi_j)=0\quad(i\neq j)
	\, , \quad
	\xi_i^\ast(\eta \nu)=\xi_i^\ast(\eta) \nu + (-1)^{|\eta|} \eta \xi_i^\ast(\nu) \, .
	\]
    As in \cite[Proposition 3.3]{HRW1}, the assignment $X\to K(X)$
	tautologically extends to a dg functor.
	\end{definition}

    In fact, $K(X)$ is contractible for any $X$. Nevertheless, we will extract useful information from this construction by means of filtrations.
	
	\begin{remark}\label{rem:sign}
		Before turning on the Koszul differential we have
		\[
		X\otimes \largewedge[\xi_1,\dots,\xi_b] = \bigoplus_{l=0}^b \bigoplus_{i_1<\cdots<i_l} X\otimes \xi_{i_1}\cdots\xi_{i_l},
		\]
		where each $X\otimes \xi_{i_1}\cdots\xi_{i_l}$ 
		denotes a copy of $X$ (appropriately shifted).
		The usual Koszul sign conventions tell us that the differential on 
		$X\otimes \xi_{i_1}\cdots\xi_{i_l}$ coincides with $\d_X$ \emph{with no sign}, 
		since the monomial in $\xi$'s appears on the right. 
		\end{remark}

        \begin{definition}
            \label{def:KCmin} We set ${}_{cd}\KCmin_{ab}:=K({}_{cd}\Rick_{ab})$.
            \end{definition}
            
            \begin{conv}
                We now establish language for discussing complexes in ${}_{cd}\CS_{ab}$ and their chain
                objects. Consider the bimodules appearing in the shifted Rickard complex ${}_{cd}\sRick_{ab}$
                \[
                W_k := 
                \begin{tikzpicture}[smallnodes,rotate=90,anchorbase,scale=.625]
                    \draw[very thick] (0,.25) to [out=150,in=270] (-.25,1) node[left,xshift=2pt]{$c$};
                    \draw[very thick] (.5,.5) to (.5,1) node[left,xshift=2pt]{$d$};
                    \draw[very thick] (0,.25) to (.5,.5);
                    \draw[very thick] (0,-.25) to (0,.25);
                    \draw[very thick] (.5,-.5) to [out=30,in=330] (.5,.5);
                    \draw[very thick] (0,-.25) to node[right,xshift=-2pt,yshift=-2pt]{$k$} (.5,-.5);
                    \draw[very thick] (.5,-1) node[right,xshift=-2pt]{$b$} to (.5,-.5);
                    \draw[very thick] (-.25,-1)node[right,xshift=-2pt]{$a$} to [out=90,in=210] (0,-.25);
                \end{tikzpicture}
                \]
                for $\max(0,b-d) \leq k \leq b$. Then we have
        \begin{equation}
            {}_{cd}\KCmin_{ab} = \Big(K(W_b)\xrightarrow{\d^H} \qdeg^{a-d+1}\tdeg K(W_{b-1})\xrightarrow{\d^H} 
            \cdots \xrightarrow{\d^H} \qdeg^{b(a-d+1)}\tdeg^b K(W_{0})\Big),
        \end{equation}
        where $\d^H = K(\chi^+_0) \colon K(W_k)\rightarrow K(W_{k-1})$.
        The differential internal to each $K(W_k)$ will be denoted $\d^v$, 
        and referred to as the \emph{vertical differential}.
        The differential $\d^H$ will be referred to as the 
        \emph{horizontal differential}.

                \end{conv}

\subsection{Example}
Before proceeding with the proof of Theorem~\ref{thm:main}, we illustrate the strategy in a prototypical example.
\begin{exa} We consider $a=b=c=d=2$ where
${}_{22}\Rick_{22}={}_{22}\MCS_{22}={}_{22}\MCSmin_{22}$ and diagrammatically
	\begin{equation*}
         \left \llbracket
		\begin{tikzpicture}[rotate=90,scale=.5,smallnodes,anchorbase]
			\draw[very thick] (1,-1) node[right,xshift=-2pt]{$2$} to [out=90,in=270] (0,1);
			\draw[line width=5pt,color=white] (0,-1) to [out=90,in=270] (1,1);
			\draw[very thick] (0,-1) node[right,xshift=-2pt]{$2$} to [out=90,in=270] (1,1);
		\end{tikzpicture}
				\right\rrbracket =
		\left(
		\begin{tikzpicture}[smallnodes,rotate=90,anchorbase,scale=.66]
			\draw[very thick] (0,.25) to [out=150,in=270] (-.25,1) node[left,xshift=2pt]{$2$};
			\draw[very thick] (.5,.5) to (.5,1) node[left,xshift=2pt]{$2$};
			\draw[very thick] (0,.25) to node[left,xshift=2pt,yshift=-1pt]{$2$} (.5,.5);
			\draw[very thick] (0,-.25) to (0,.25);
			\draw[dotted] (.5,-.5) to [out=30,in=330] (.5,.5);
			\draw[very thick] (0,-.25) to node[right,xshift=-2pt,yshift=-1pt]{$2$} (.5,-.5);
			\draw[very thick] (.5,-1) node[right,xshift=-2pt]{$2$} to (.5,-.5);
			\draw[very thick] (-.25,-1)node[right,xshift=-2pt]{$2$} to [out=90,in=210] (0,-.25);
		\end{tikzpicture}
		\to  \qdeg\inv \tdeg
		\begin{tikzpicture}[smallnodes,rotate=90,anchorbase,scale=.66]
			\draw[very thick] (0,.25) to [out=150,in=270] (-.25,1) node[left,xshift=2pt]{$2$};
			\draw[very thick] (.5,.5) to (.5,1) node[left,xshift=2pt]{$2$};
			\draw[very thick] (0,.25) to node[left,xshift=2pt,yshift=-1pt]{$1$} (.5,.5);
			\draw[very thick] (0,-.25) to (0,.25);
			\draw[very thick] (.5,-.5) to [out=30,in=330] (.5,.5);
			\draw[very thick] (0,-.25) to node[right,xshift=-2pt,yshift=-1pt]{$1$} (.5,-.5);
			\draw[very thick] (.5,-1) node[right,xshift=-2pt]{$2$} to (.5,-.5);
			\draw[very thick] (-.25,-1)node[right,xshift=-2pt]{$2$} to [out=90,in=210] (0,-.25);
		\end{tikzpicture}
		\to  \qdeg^{-2} \tdeg^2
		\begin{tikzpicture}[smallnodes,rotate=90,anchorbase,scale=.66]
			\draw[very thick] (-.25,-1) node[right,xshift=-2pt]{$2$} to (-.25,1) node[left,xshift=2pt]{$2$};
			\draw[very thick] (.5,-1) node[right,xshift=-2pt]{$2$} to (.5,1) node[left,xshift=2pt]{$2$};
		\end{tikzpicture}
		\right) \, .
	\end{equation*}
 The Koszul complex $K({}_{22}C_{22})$ has the following schematic form:
 \begin{equation*}
    \begin{tikzpicture}[anchorbase]
        \node[scale=1] at (5,-2.5){$\MCSmin_{2,2} \otimes \wedge$};
        \node[scale=1] at (0,0.5){
    \begin{tikzcd}[row sep=2em,column sep=-1.5em]
        & 
        W_{2}\otimes \xi_1\xi_2
        \arrow[ddl]
        \arrow[dr,dotted,dash ]
        \arrow[rrr,"\chi^+_0"] & & & 
        W_{1}\otimes \xi_1\xi_2
        \arrow[dr,dotted,dash] 
        \arrow[rrr,"\chi^+_0"] & & & 
        W_{0}\otimes \xi_1\xi_2
        \arrow[from=dr,dotted,dash] & 
        \\
        & &
        W_{2}\otimes \xi_2	
        \arrow[ldd]
        \arrow[rrr,"\chi^+_0", near end] & & & 
        W_{1}\otimes \xi_2
        \arrow[ldd] 
        \arrow[rrr,"\chi^+_0", near end]  & & & 
        W_{0}\otimes \xi_2	
        \arrow[ldd]
        \\
        W_{2}\otimes \xi_1	
        \arrow[dr,dotted,dash]
        \arrow[rrr,crossing over,"\chi^+_0"] 
        & & & 
        W_{1}\otimes \xi_1
        \arrow[dr,dotted,dash]
        \arrow[rrr,crossing over,"\chi^+_0"] 
        \arrow[from=uur,crossing over] 
        & & & 
        W_{0}\otimes \xi_1	
        \arrow[from=dr,dotted,dash]
        \arrow[from=uur, crossing over]
         & & 
        \\
        &
         W_{2}\otimes 1	 \arrow[rrr,"\chi^+_0"]& & & 
         W_{1}\otimes 1	 \arrow[rrr,"\chi^+_0"]& & & 
         W_{0}\otimes 1	 &
    \end{tikzcd}
        };
    \end{tikzpicture}
    \end{equation*}
    After changing to different bases $\zeta^{(k)}_{*}$ in the Koszul columns we obtain the form:
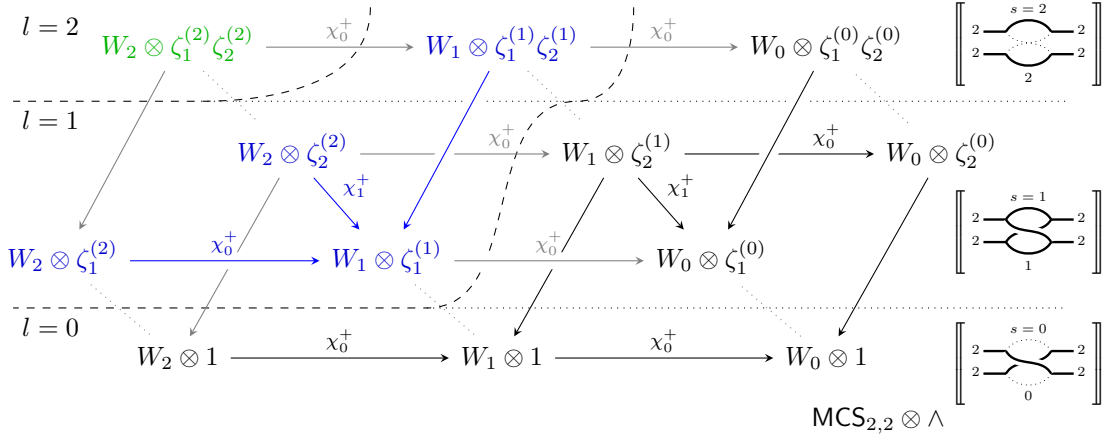
\begin{equation*}
\begin{tikzpicture}[anchorbase]
	\node[scale=1] at (5,-2.5){$\MCSmin_{2,2} \otimes \wedge$
	};
	\draw[dashed] (1.75,3) to (1.75,3) to [out=270,in=0] (1,1.75) to [out=180,in=0](-1,-1) to (-6.5,-1);
	\draw[dotted] (-4,1.75) to (7.5,1.75);
	\draw[dotted] (-1,-1) to (7.5,-1);
	\draw[dashed] (-1.75,3) to (-1.75,3) to [out=270,in=0] (-4,1.75) to (-6.5,1.75);
	\node at (-6,2.75) {$l=2$};
	\node at (-6,1.5) {$l=1$};
	\node at (-6,-1.25) {$l=0$};
	\node[scale=.75] at (7,2.5) {$\left\llbracket
	\begin{tikzpicture}[scale=.4,smallnodes,anchorbase,rotate=270]	
		\draw[very thick] (0,-2) node[left=-2pt]{$2$} to (0,-1) (0,1) to (0,2) node[right=-2pt]{$2$};
		\draw[very thick] (1,-2) node[left=-2pt]{$2$} to (1,-1) (1,1) to (1,2) node[right=-2pt]{$2$};
		\draw[dotted] (1,-1) to [out=150,in=330] (0,1); 
		\draw[dotted] (0,-1) to [out=30,in=210] (1,1);
		\draw[very thick] (1,-1) to [out=30,in=330]node[below=1pt]{$2$}  (1,1); 
		\draw[very thick] (0,-1) to [out=150,in=210]node[above=-1pt]{$s=2$} (0,1); 
		\end{tikzpicture}
	\right\rrbracket$
	};
	\node[scale=.75] at (7,0) {$\left\llbracket
\begin{tikzpicture}[scale=.4,smallnodes,anchorbase,rotate=270]
	\draw[very thick] (0,-2) node[left=-2pt]{$2$} to (0,-1) (0,1) to (0,2) node[right=-2pt]{$2$};
	\draw[very thick] (1,-2) node[left=-2pt]{$2$} to (1,-1) (1,1) to (1,2) node[right=-2pt]{$2$};
	\draw[very thick] (1,-1) to [out=150,in=330] (0,1); 
	\draw[line width=5pt,color=white] (0,-1) to [out=30,in=210] (1,1);
	\draw[very thick] (0,-1) to [out=30,in=210] (1,1);
	\draw[very thick] (1,-1) to [out=30,in=330]node[below=1pt]{$1$}  (1,1); 
	\draw[very thick] (0,-1) to [out=150,in=210]node[above=-1pt]{$s=1$} (0,1); 
\end{tikzpicture}
\right\rrbracket$
};
	\node[scale=.75] at (7,-1.75) {$\left\llbracket
	\begin{tikzpicture}[scale=.4,smallnodes,anchorbase,rotate=270]
		\draw[very thick] (0,-2) node[left=-2pt]{$2$} to (0,-1) (0,1) to (0,2) node[right=-2pt]{$2$};
		\draw[very thick] (1,-2) node[left=-2pt]{$2$} to (1,-1) (1,1) to (1,2) node[right=-2pt]{$2$};
		\draw[very thick] (1,-1) to [out=150,in=330] (0,1); 
		\draw[line width=5pt,color=white] (0,-1) to [out=30,in=210] (1,1);
		\draw[very thick] (0,-1) to [out=30,in=210] (1,1);
		\draw[dotted] (1,-1) to [out=30,in=330]node[below=1pt]{$0$}  (1,1); 
		\draw[dotted] (0,-1) to [out=150,in=210]node[above=-1pt]{$s=0$} (0,1); 
		\end{tikzpicture}
	\right\rrbracket$
	};
	\node[scale=1] at (0,0.5){
\begin{tikzcd}[row sep=2em,column sep=-1.5em]
	& 
	\GREEN{W_{2}\otimes \zeta^{(2)}_1\zeta^{(2)}_2}	
	\arrow[ddl, gray]
	\arrow[dr,dotted,gray,dash ]
	\arrow[rrr,gray,"\GRAY{\chi^+_0}"] & & & 
	\BLUE{W_{1}\otimes \zeta^{(1)}_1\zeta^{(1)}_2}
	\arrow[dr,dotted,gray,dash] 
	\arrow[rrr,gray,"\GRAY{\chi^+_0}"] & & & 
	W_{0}\otimes \zeta^{(0)}_1\zeta^{(0)}_2
	\arrow[from=dr,dotted,gray,dash] & 
	\\
	& &
	\BLUE{W_{2}\otimes \zeta^{(2)}_2}	
	\arrow[ldd, gray]
	\arrow[dr,blue,"\BLUE{\chi^+_1}"] 
	\arrow[rrr,gray,"\GRAY{\chi^+_0}", near end] & & & 
	W_{1}\otimes \zeta^{(1)}_2
	\arrow[ldd] 
	\arrow[dr,"\chi^+_1"] 
	\arrow[rrr,"\chi^+_0", near end]  & & & 
	W_{0}\otimes \zeta^{(0)}_2	
	\arrow[ldd]
	\\
	\BLUE{W_{2}\otimes \zeta^{(2)}_1}	
	\arrow[dr,dotted,dash,gray]
	\arrow[rrr,crossing over,blue,"\BLUE{\chi^+_0}"] 
	& & & 
	\BLUE{W_{1}\otimes \zeta^{(1)}_1}	
	\arrow[dr,dotted,dash,gray]
	\arrow[rrr,crossing over,gray,"\GRAY{\chi^+_0}"] 
	\arrow[from=uur,crossing over, blue] 
	& & & 
	W_{0}\otimes \zeta^{(0)}_1	
	\arrow[from=dr,dotted,gray,dash]
	\arrow[from=uur, crossing over]
	 & & 
	\\
	&
	 W_{2}\otimes 1	 \arrow[rrr,"\chi^+_0"]& & & 
	 W_{1}\otimes 1	 \arrow[rrr,"\chi^+_0"]& & & 
	 W_{0}\otimes 1	 &
\end{tikzcd}
	};
\end{tikzpicture}
\end{equation*}
The coloring and the dashed lines indicate a certain filtration in terms of a
parameter $l$ that will be described in the next section. The subquotients of
the $l$-filtrates further have an $s$-filtration as indicated by the dotted
lines. The subcomplex for $l=0$ (black) retracts onto $W_2 \otimes 1$ since all
unlabelled solid arrows are identities. This is the right-hand side of the
categorified graded bialgebra relation \eqref{eq:compat}. On the other hand, the
subquotients of the $l=0$ part with respect to the $s$-filtration are homotopy
equivalent the complexes of type $\IMCS$ as shown on the right, matching the
terms on the left-hand side of the categorified graded bialgebra relation
\eqref{eq:compat}.
	\end{exa}

    \subsection{Proof of the categorified graded bialgebra relations}
    \label{sec:proofbialgebra}
    We perform a change of basis within 
    the exterior algebra tensor factor of each $K(W_k)$ in ${}_{cd}\KCmin_{ab}$, 
    i.e. we replace each \emph{column complex} $W_k\otimes
    \largewedge[\xi_1,\ldots,\xi_{b}]$ by an isomorphic Koszul complex.

    \begin{conv}
        We use the following naming convention for the
        alphabets of degree two variables assinged to edges in such webs, symmetric polynomials of which act as endomorphisms of the associated bimodules, namely:
        
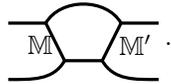
\[
        \begin{tikzpicture}[rotate=90,anchorbase]
            \draw[very thick] (0,.25) to [out=150,in=270] (-.25,1);
            \draw[very thick] (.5,.5) to (.5,1);
            \draw[very thick] (0,.25) to node[left,yshift=-1pt,xshift=3pt]{$\leftM$} (.5,.5);
            \draw[very thick] (0,-.25) to  (0,.25);
            \draw[very thick] (.5,-.5) to [out=30,in=330] (.5,.5);
            \draw[very thick] (0,-.25) to node[right,yshift=-1pt,xshift=-1pt]{$\rightM$} (.5,-.5);
            \draw[very thick] (.5,-1) to (.5,-.5);
            \draw[very thick] (-.25,-1) to [out=90,in=210] (0,-.25);
        \end{tikzpicture} \, .
        \]
        If we wish to emphasize the index $k$, we will write $\leftM^{(k)}, {\rightM}^{(k)}$, etc. 
        In particular, we note that 
        $|\leftM^{(k)}| = d-b+k,$ and $|{\rightM}^{(k)}| = k$
        for all $k$. 

        Now let $\zeta_1^{(k)},\ldots,\zeta_b^{(k)}$ denote odd variables related to $\xi_1,\ldots,\xi_{b}$ by the formulas 
        \[
        \zeta_j^{(k)} := \sum_{i=1}^j (-1)^{i-1} e_{j-i}(\leftM^{(k)}) \otimes \xi_i \; , \qquad 
        \xi_i = \sum_{j=1}^i (-1)^{j-1}h_{i-j}(\leftM^{(k)}) \otimes \zeta_j^{(k)}
        \]
        where $e_{j-i}$ and $h_{i-j}$ denote elementary and complete symmetric polynomials.
    \end{conv}
We now wish to describe ${}_{cd}\KCmin_{ab}$ in terms of the
$\zeta_*^{(k)}$-basis. The following is immediate.

\begin{lemma}\label{lemma:zeta d2}
	Consider the dg algebra $\Sym(\M|\M')\otimes \largewedge[\xi_1,\ldots,\xi_b]$ 
	with $\Sym(\M|\M')$-linear derivation $d$ defined by $d(\xi_i)  = \delta_{0,b}$ for all $1\leq i\leq b$. 
	Then the elements $\zeta_j:=\sum_{i=1}^j (-1)^{i-1} e_{j-i}(\M) \otimes \xi_i$ satisfy 	$d(\zeta_j) = \delta_{j,b}$.
	\end{lemma}

\begin{prop}\label{prop:diffKMCSexplicit}
	We have an isomorphism $K(W_k) \cong \tw_{\d^v}(W_k\otimes \largewedge[\zeta^{(k)}_1,\ldots,\zeta^{(k)}_b])$ where
	\begin{equation}\label{eq:dv}
	\d^v = 1 \otimes (\zeta_b^{(k)})^\ast \, .
	\end{equation}
	Under this isomorphism, 
	the differential $\d^H\colon K(W_k)\rightarrow K(W_{k-1})$ has a nonvanishing component
	\[
	W_k \otimes \zeta^{(k)}_{i_1}\cdots \zeta^{(k)}_{i_r}
	\xrightarrow{\d^H}
	W_{k-1}\otimes \zeta^{(k-1)}_{j_1}\cdots \zeta^{(k-1)}_{j_r}
	\]
	if and only if $i_p - j_p\in \{0,1\}$ for all $1\leq p\leq r$.
	In that case, it equals $\chi_m^+$ from \cite[Equation (14)]{HRW1} 
	where $m=\sum_{p=1}^r (i_p-j_p)$.
	\end{prop}
\begin{proof} The first statement follows from Lemma~\ref{lemma:zeta
	d2}. The second statement was proved in \cite[Proposition
3.10]{HRW1}.
\end{proof}

The bimodule homomorphism $\chi_m^+$ is best described in terms of an action of the categorified quantum group $\mathcal{U}(\mathfrak{gl}_2)$ on singular Soergel bimodules \cite[Proposition 2.18]{HRW1}. Informally, it is the foam from Figure~\ref{fig:diffslices}, but additionally decorated with the $m$th power of the variable assigned to the central disk.

\begin{definition}\label{def:Pkls}
	Set
	$
	P'_{k,l,s} :=
	\qdeg^{(k-b)(a-d+1)-2b} \tdeg^{b-k} 
	W_k
	\otimes \largewedge^{l}[\zeta^{(k)}_1,\ldots, \zeta^{(k)}_k]\otimes \largewedge^s[\zeta^{(k)}_{k+1},\ldots, \zeta^{(k)}_b]
	$.
	\end{definition}

    The following is a direct consequence of Proposition~\ref{prop:diffKMCSexplicit} and the definitions.
	\begin{proposition}\label{prop:KCdiffs}
		We have an isomorphism
		\begin{equation}
			\label{eqn:filtKCmin}
			{}_{cd}\KCmin_{ab} \cong  
		\tw_{\d^v+\d^h+\d^c}\left(\bigoplus_{0\leq l\leq k\leq b-s} P'_{k,l,s}\right) \, ,
	\end{equation}
		where the sum ranges over $k$, $l$, $s$ and $\d^v$, $\d^h$, $\d^c$ are
		pairwise anti-commuting differentials described as follows:
		\begin{itemize}
		\item the \emph{vertical differential} $\d^v \colon P'_{k,l,s}\rightarrow P'_{k,l,s-1}$
		is the direct sum of the Koszul differentials, 
		up to sign $(-1)^{b-k}$; the component
		\[
		W_k \otimes \zeta^{(k)}_{i_1}\cdots \zeta^{(k)}_{i_r} \xrightarrow{\d^v}
		W_{k}\otimes \zeta^{(k)}_{i_1}\cdots \widehat{\zeta^{(k)}_{i_j}}\cdots
		\zeta^{(k)}_{i_r}
		\]
		is $(-1)^{b-k+j-1} \delta_{i_j,b}$. 
		\item  the \emph{horizontal differential} $\d^h$ and the connecting differential
		$\d^c$ are characterized uniquely by $\d^h+\d^c=\d^H$ from Proposition
		\ref{prop:diffKMCSexplicit}, together with
		\[
		\d^h(P_{k,l,s})\subset P_{k-1,l,s} \, , \quad \d^c(P'_{k,l,s})\subset P'_{k-1,l-1,s+1}.
		\]
		\end{itemize}
		\end{proposition}
        \begin{remark}
			Proposition~\ref{prop:KCdiffs} is an analog of \cite[Proposition
			3.12]{HRW1} for ${}_{cd}\KCmin_{ab}$ instead of $\KMCSmin_{a,b}$
			The main upshot of the former was a filtration with respect to the
			$s$-parameter. Note that while ${}_{cd}\KCmin_{a,b}$ does
			\emph{not} have such an $s$-filtration, both  ${}_{cd}\KCmin_{ab}$
			and $\KMCSmin_{a,b}$ have $l$-filtrations.
		\end{remark}

		\begin{cor}\label{cor:filt} The complex ${}_{cd}\KCmin_{ab}$ as presented in
		\eqref{eqn:filtKCmin} is filtered by the parameter $l$. The part ${}_{cd}\KCmin^{l=0}_{ab}$ where $l=0$ 
		is a subcomplex, on which the connecting differential $\delta_c$
		vanishes, and it is homotopy equivalent to $W_b$:
		\[	{}_{cd}\KCmin^{l=0}_{ab}:=
		\tw_{\d^v+\d^h}\left(\bigoplus_{0\leq k\leq b-s} P'_{k,0,s}\right) \simeq W_b	\]
		Here the sum ranges over $k$ and $s$. Moreover, ${}_{cd}\KCmin^{l=0}_{ab}$ is filtered by $s$. 
		\end{cor}

        Thus we have obtained the right-hand side of the categorified graded bialgebra relation \eqref{eq:compat}. To obtain the left-hand side, we consider the $s$-filtration on ${}_{cd}\KCmin^{l=0}_{ab}$.

		\begin{definition}\label{def:IMCSmin} For any $0\leq s \leq b$ we denote the corresponding subquotient of the $s$-filtration on
		${}_{cd}\KCmin^{l=0}_{ab}$ by
	\[	{}_{cd}\IMCSmin^s_{ab} : = \qdeg^{s(s+a-d)} \tw_{\d^h}\left(\bigoplus_{0\leq k\leq b-s} P'_{k,0,s}\right).\] 
Here the sum ranges only over $k$.	
		\end{definition}

The notation for these subquotients has suggestively been chosen.


\begin{proposition}\label{prop:IMCS}
	For any $0\leq s \leq b$ we have a homotopy equivalence 
	\begin{equation*}  
     {}_{cd}\IMCS^s_{ab} =   
    \left\llbracket
\begin{tikzpicture}[scale=.4,smallnodes,anchorbase,rotate=270]
\draw[very thick] (1,-1) to [out=150,in=330] (0,1) to (0,2) node[right=-2pt]{$b$}; 
\draw[line width=5pt,color=white] (0,-2) to (0,-1) to [out=30,in=210] (1,1);
\draw[very thick] (0,-2) node[left=-2pt]{$d$} to (0,-1) to [out=30,in=210] (1,1);
\draw[very thick] (1,1) to (1,2) node[right=-2pt]{$a$};
\draw[very thick] (1,-2) node[left=-2pt]{$c$} to (1,-1); 
\draw[very thick] (1,-1) to [out=30,in=330]  (1,1); 
\draw[very thick] (0,-1) to [out=150,in=210]node[above=-1pt]{$s$} (0,1); 
\end{tikzpicture}
\right\rrbracket
= 
\I^{(s)} \left({}_{c(d-s)}\MCS^s_{a(b-s)}\right) \simeq {}_{cd}\IMCSmin^s_{ab}
.\end{equation*}
\end{proposition}
The proof will be based on the following lemma, which is a minor variation of
\cite[Lemma 3.28]{HRW1}.
\begin{lemma}\label{lemma:I of Wk}
	For each $0\leq s\leq b$ and each $0\leq k\leq b-s$, 
	we have an isomorphism of degree $\qdeg^{s(k-b+1)} \tdeg^{-s}$ :
	\begin{equation}\label{eq:mu}
	\mu_k \colon  
	\begin{tikzpicture}[smallnodes, scale=.75,rotate=90,anchorbase]
		\draw[very thick,] (0,.25) to [out=150,in=270] (-.25,1.25) node[left,xshift=2pt]{$c$};
		\draw[very thick,] (.5,.5) to [out=90,in=210] (.75,1);
		\draw[very thick] (.75,1) to (.75,1.25) node[left,xshift=2pt]{$d$};
		\draw[very thick] (0,.25) to (.5,.5);
		\draw[very thick] (0,-.25) to (0,.25);
		\draw[very thick] (.5,-.5) to [out=30,in=330] (.5,.5);
		\draw[very thick] (0,-.25) to node[right,xshift=-2pt,yshift=-1pt]{$k$} (.5,-.5);
		\draw[very thick] (.75,-1) to [out=150,in=270] (.5,-.5);
		\draw[very thick] (.75,-1.25) node[right,xshift=-2pt]{$b$} to (.75,-1);
		\draw[very thick] (-.25,-1.25) node[right,xshift=-2pt]{$a$} to [out=90,in=210] (0,-.25);
		\draw[very thick] (.75,-1) to [out=30,in=330] node[above,yshift=-2pt]{$s$} (.75,1);
	\end{tikzpicture}
	\xrightarrow{\cong} 
	W_{k} \otimes \largewedge^s[\zeta^{(k)}_{k+1},\ldots,\zeta^{(k)}_{b}] \, .
	\end{equation}
	Moreover, for each integer $m\geq 0$, these isomorphisms fit into a commutative diagram 
	\[
	\begin{tikzcd}[column sep=huge]
	\begin{tikzpicture}[smallnodes, scale=.75,rotate=90,anchorbase]
		\draw[very thick,] (0,.25) to [out=150,in=270] (-.25,1.25) node[left,xshift=2pt]{$c$};
		\draw[very thick,] (.5,.5) to [out=90,in=210] (.75,1);
		\draw[very thick] (.75,1) to (.75,1.25) node[left,xshift=2pt]{$d$};
		\draw[very thick] (0,.25) to (.5,.5);
		\draw[very thick] (0,-.25) to (0,.25);
		\draw[very thick] (.5,-.5) to [out=30,in=330] (.5,.5);
		\draw[very thick] (0,-.25) to node[right=-3pt,yshift=-1pt]{$k$} (.5,-.5);
		\draw[very thick] (.75,-1) to [out=150,in=270] (.5,-.5);
		\draw[very thick] (.75,-1.25) node[right,xshift=-2pt]{$b$} to (.75,-1);
		\draw[very thick] (-.25,-1.25) node[right,xshift=-2pt]{$a$} to [out=90,in=210] (0,-.25);
		\draw[very thick] (.75,-1) to [out=30,in=330] node[above,yshift=-2pt]{$s$} (.75,1);
	\end{tikzpicture}
	\ar[r,"I^{(s)}(\chi^+_m)"] \ar[d,"\mu_k"]
	& 
	\begin{tikzpicture}[smallnodes, scale=.75,rotate=90,anchorbase]
		\draw[very thick,] (0,.25) to [out=150,in=270] (-.25,1.25) node[left,xshift=2pt]{$c$};
		\draw[very thick,] (.5,.5) to [out=90,in=210] (.75,1);
		\draw[very thick] (.75,1) to (.75,1.25) node[left,xshift=2pt]{$d$};
		\draw[very thick] (0,.25) to  (.5,.5);
		\draw[very thick] (0,-.25) to (0,.25);
		\draw[very thick] (.5,-.5) to [out=30,in=330] (.5,.5);
		\draw[very thick] (0,-.25) to node[right=-3pt,yshift=-1pt]{$k{-}1$} (.5,-.5);
		\draw[very thick] (.75,-1) to [out=150,in=270] (.5,-.5);
		\draw[very thick] (.75,-1.25) node[right,xshift=-2pt]{$b$} to (.75,-1);
		\draw[very thick] (-.25,-1.25) node[right,xshift=-2pt]{$a$} to [out=90,in=210] (0,-.25);
		\draw[very thick] (.75,-1) to [out=30,in=330] node[above,yshift=-2pt]{$s$} (.75,1);
	\end{tikzpicture}
	\ar[d,"\mu_{k-1}"] \\
	W_{k} \otimes \largewedge^s[\zeta^{(k)}_{k+1},\ldots,\zeta^{(k)}_{b}]
	\ar[r,"f"]
	&
	W_{k-1} \otimes \largewedge^s[\zeta^{(k-1)}_{k},\ldots,\zeta^{(k-1)}_{b}]
	\end{tikzcd}
	\]
	where, for $k+1 \leq i_1 < \cdots < i_s \leq b$ and 
	$k+1 \leq j_1 < \cdots < j_s \leq b$, the component
	\[
	W_{k} \otimes \zeta^{(k)}_{i_1}\cdots \zeta^{(k)}_{i_s}
	\xrightarrow{f} W_{k-1} \otimes  \zeta^{(k-1)}_{j_1}\cdots \zeta^{(k-1)}_{j_s}
	\]
	is zero unless $i_p-j_p\in\{0,1\}$ for all $k+1\leq p\leq b$. 
	In this case, it equals $\chi^+_{m+n}$ from \cite[Equation (14)]{HRW1} where $n=\sum_p(i_p-j_p)$.
\end{lemma}

	\begin{proof}[Proof of Proposition~\ref{prop:IMCS}]
	Since $\I^{(s)}$ is a dg functor, the homotopy equivalence from \eqref{eqn:MCSred} induces 
		\[ {}_{cd}\IMCS^s_{ab} = \I^{(s)}
		\left({}_{c(d-s)}\MCS^s_{a(b-s)}\right)\simeq \I^{(s)}
		\left({}_{c(d-s)}\MCSmin^s_{a(b-s)}\right) = \I^{(s)}
		\left({}_{c(d-s)}\sRick_{a(b-s)}\right).\] Now Lemma~\ref{lemma:I	of Wk}
		expresses the chain groups and differentials of the latter as identical
		to those of	${}_{cd}\IMCSmin^s_{ab}$ as presented via
		Proposition~\ref{prop:KCdiffs}.
	\end{proof} 

\begin{proof}[Proof of Theorem~\ref{thm:main}]
	Based on the homotopy equivalences from Proposition~\ref{prop:IMCS},
	homological perturbation gives us a twist $D$ with 
	\[\tw_{D}\left(\bigoplus_{s=0}^b 
	\qdeg^{-s(s+a-d)}
	 {}_{cd}\IMCS^s_{ab} \right) 
	\simeq 
	\tw_{\delta^v}\left(\bigoplus_{s=0}^b 
	\qdeg^{-s(s+a-d)}
	 {}_{cd}\IMCSmin^s_{ab} \right) 
	\overset{\text{Def.}~\ref{def:IMCSmin}}{=} {}_{cd}\KCmin^{l=0}_{ab}
	\overset{\text{Cor.}~\ref{cor:filt}}{\simeq} W_{b}
	\]
	with twist $D$ strictly decreasing in $s$.
\end{proof}

\bibliographystyle{alpha}
\bibliography{pw}

\end{document}

%% file: preamble.tex
\usepackage{latexsym,exscale}
\usepackage{mathtools,amsmath,mathscinet,amsfonts,amssymb,amsthm,bbm,euscript}
\usepackage{amscd, stmaryrd}
\usepackage[normalem]{ulem}
\usepackage[all]{xy}
\SelectTips{cm}{}
\usepackage{graphicx}
\usepackage{wrapfig}
\usepackage{xcolor}
\usepackage{bm}

\usepackage{todonotes}
\usepackage{enumitem}

\setlist[enumerate]{align=left, topsep=1ex,itemsep=0ex,partopsep=1ex,parsep=1ex}
\numberwithin{equation}{subsection}

\newcommand{\Catinfty}{{\EuScript Cat}_{\infty}}
\newcommand{\Stab}{{\EuScript St}}

\def\A{\mathbb{A}}
\def\B{\mathbb{B}}
\def\C{\mathbb{C}}
\def\D{\mathbb{D}}

\def\N{\mathbb{N}}
\def\Q{\mathbb{Q}}

\def\X{\mathbb{X}}
\def\Y{\mathbb{Y}}
\def\Z{\mathbb{Z}}

\def\on{\operatorname}
\def\Fun{\on{Fun}}
\def\Pr{\on{Pr}}
\def\AA{\mathbb{A}}
\def\CC{\mathbb{C}}

\def\NN{\mathbb{N}}
\def\EE{\mathbb{E}}

\renewcommand{\A}{\EuScript A}
\renewcommand{\B}{\EuScript B}
\renewcommand{\C}{\EuScript C}
\renewcommand{\D}{\EuScript D}
\newcommand{\I}{\EuScript I}

\newcommand{\wS}{\on{S}}

\renewcommand{\Y}{\EuScript Y}
\newcommand{\E}{\EuScript E}
\renewcommand{\X}{\EuScript X}

\newcommand{\Comp}{\on{Comp}}
\newcommand{\Expl}{\on{Expl}}
\newcommand{\Grp}{\on{Grp}}
\newcommand{\Soe}{\on{Soe}}

\newcommand{\Sp}{\on{Sp}}  
\newcommand{\FrobExt}{\on{FrobExt}}
\newcommand{\Frob}{\on{Frob}}
\newcommand{\Rick}{\on{C}}
\newcommand{\sRick}{\on{Rick}}
\newcommand{\Par}{\on{Par}}
\newcommand{\fib}{\on{fib}}
\newcommand{\cofib}{\on{cofib}}
\newcommand{\Mor}{\on{Mor}}
\newcommand{\Modk}{\on{Mod}_k}
\newcommand{\deffun}{\on{def}}
\newcommand{\inffun}{\on{inf}}

\RequirePackage{color}
\definecolor{references}{rgb}{.7,.1,.6}

\RequirePackage[pdftex,
colorlinks = true,
urlcolor = references, 
citecolor = references, 
linkcolor = references, 
]
{hyperref}

\usepackage{tikz}
\usepackage{tikz-cd}
\usetikzlibrary{math}
\usetikzlibrary{decorations.markings}
\usetikzlibrary{decorations.pathreplacing}
\usetikzlibrary{arrows,shapes,positioning}
\tikzstyle directed=[postaction={decorate,decoration={markings,
    mark=at position #1 with {\arrow{>}}}}]
\tikzstyle rdirected=[postaction={decorate,decoration={markings,
    mark=at position #1 with {\arrow{<}}}}]

\tikzset{anchorbase/.style={baseline={([yshift=-0.5ex]current bounding box.center)}},
    tinynodes/.style={font=\tiny,text height=0.75ex,text depth=0.15ex},
    smallnodes/.style={font=\scriptsize,text height=0.75ex,text depth=0.15ex},
    >={Latex[length=1mm, width=1.5mm]}
  }
  \tikzset{
    partial ellipse/.style args={#1:#2:#3}{
        insert path={+ (#1:#3) arc (#1:#2:#3)}
    }
}
\tikzcdset{arrow style=tikz, diagrams={>=stealth}}

  \newcommand{\pu}{to [out=90,in=270]}

\allowdisplaybreaks

\colorlet{green}{black!30!green}
\newcommand{\BLUE}[1]{\textcolor[rgb]{0.00,0.00,0.80}{#1}}
\newcommand{\GREEN}[1]{\textcolor[rgb]{0.00,0.70,0.00}{#1}}

\newcommand{\GRAY}[1]{\textcolor[rgb]{0.50,0.50,0.50}{#1}}
\definecolor{CQG}{RGB}{0,153,76}
\definecolor{FS}{RGB}{0,76,153} 

\newcommand{\googlebooks}[1]{(preview at \href{https://books.google.com/books?id=#1}{google books})}

\newcommand{\numdam}[1]{}

\def\emph#1{{\sl #1\/}}

\def\mf#1{\mathfrak{#1}}

\def\lra{{\longrightarrow}}

\renewcommand{\to}{\rightarrow}

\newcommand{\AS}{\EuScript A}
\newcommand{\BS}{\EuScript B}
\newcommand{\CS}{\EuScript C}

\newcommand{\VS}{\EuScript V}

\renewcommand{\d}{\delta}

\renewcommand{\phi}{\varphi}
\renewcommand{\theta}{\vartheta}

\newcommand{\slnn}[1]{\mf{sl}_{#1}}

\newcommand{\glN}{\mf{gl}_N}

\newcommand{\SBim}{\mathrm{SBim}}
\newcommand{\sSBim}{\mathrm{sSBim}}

\newcommand{\sBSBim}{\mathrm{sBSBim}}

\newcommand{\Mod}{\text{-}\mathrm{Mod}}

\newcommand{\del}{\partial}

\newcommand{\id}{\mathrm{id}}

\newcommand{\Br}{\operatorname{Br}}
\newcommand{\Ch}{\operatorname{Ch}}

\newcommand{\Hom}{\operatorname{Hom}}

\newcommand{\largewedge}{\mbox{\Large $\wedge$}}

\newcommand{\rk}{\operatorname{\rk}}

\newcommand{\Sym}{\operatorname{Sym}}

\newcommand{\tw}{\operatorname{tw}}

\newcommand{\wt}{\operatorname{wt}}

\newcommand{\op}{\mathrm{op}}

\newcommand{\inv}{^{-1}}

\newcommand{\qdeg}{\mathbf{q}}
\newcommand{\tdeg}{\mathbf{t}}

\newcommand{\M}{\mathbb{M}}

\newcommand{\leftM}{\mathbb{M}}
\newcommand{\rightM}{\mathbb{M}'}

\newcommand{\MCS}{\mathrm{MCS}}
\newcommand{\MCSmin}{\mathsf{MCS}}

\newcommand{\KMCSmin}{\mathsf{KMCS}}

\newcommand{\IMCS}{\mathrm{IMCS}}
\newcommand{\IMCSmin}{\mathsf{IMCS}}

\newcommand{\KCmin}{\mathsf{KC}}

\newcommand{\St}{\mathrm{Staircase}}

\definecolor{kirb}{RGB}{254,112,198}
\tikzset{kirbystyle/.style={ultra thick,kirb},
  }

\newcommand{\diffslices}{
\tikzmath{
\ylift = 0.2;
\yplus = 0.6;
\xshift = -1;
\fx = -.25;
\fxx = -.2;
\fy = -.5;
\fyy = -.25;
	 } 
\begin{tikzpicture} [scale=.5,fill opacity=0.1,anchorbase]
\path[fill=green, opacity=0.3] (.5+\xshift,1.25+\ylift+\yplus) to [out=135,in=0] (\xshift,1.5+\ylift+\yplus) to [out=180,in=45] 
(-.5+\xshift,1.25+\ylift+\yplus) to [out=225,in=270]  (-2,2.75+\ylift) to [out=90,in=180] (0,4.25+\ylift) to [out=0,in=90] (2,2.75+\ylift) to [out=270,in=315] (.5+\xshift,1.25+\ylift+\yplus) ;
\path[fill=green] (-.5+\xshift,0) to (-.5+\xshift,5) to (-2,6) to (-2,1) to (-.5+\xshift,0);
\path[fill=green] (.5+\xshift,0) to (.5+\xshift,5) to (2,6) to (2,1) to (.5+\xshift,0);
\path[fill=blue] (-3,1) to (-2,1) to [out=20,in=180] (0,1.4) to [out=0,in=160] (2,1) to (3,1) to (3,6) to (2,6) to [out=160,in=0] (0,6.4) to [out=180,in=20] (-2,6) to  (-3,6);
\path[fill=blue]  (3+\xshift+\fx,5+\fy)  to (3+\xshift+\fx,\fy)  to [out=180,in=340] (.5+\xshift,0) to (-.5+\xshift,0) to [out=200,in=0] (-3+\xshift+\fx,\fy) to (-3+\xshift+\fx,5+\fy)to [out=0,in=200] (-.5+\xshift,5) to (.5+\xshift,5) to [out=340,in=180] (3+\xshift+\fx,5+\fy);
	\draw[very thick, red] (.5+\xshift,1.25+\ylift+\yplus) to [out=135,in=0] (0+\xshift,1.5+\ylift+\yplus) to [out=180,in=45] (-.5+\xshift,1.25+\ylift+\yplus);
	\draw[very thick, red, rdirected=.55] (-2,2.75+\ylift) to [out=90,in=180] (0,4.25+\ylift) to [out=0,in=90] (2,2.75+\ylift);
\draw[very thick, red, directed=.80] (-2,6) to (-2,1);
\draw[very thick, red, directed=.15] (2,1) to (2,6);
	\draw[very thick] (3,1) to (2,1) to [out=160,in=0] (0,1.4) to [out=180,in=20] (-2,1) to (-3,1);
\draw[very thick, red, directed=.22] (-.5+\xshift,0) to (-.5+\xshift,5);
\draw[very thick, red, directed=.83] (.5+\xshift,5) to (.5+\xshift,0);
\draw[very thick, red, rdirected=.5] (2,2.75+\ylift) to [out=270,in=315] (.5+\xshift,1.25+\ylift+\yplus);
\draw[very thick, red] (-.5+\xshift,1.25+\ylift+\yplus) to [out=225,in=270] (-2,2.75+\ylift);
	\draw[very thick] (3+\xshift+\fx,\fy) to [out=180,in=340] (.5+\xshift,0) to (-.5+\xshift,0) to [out=200,in=0] (-3+\xshift+\fx,\fy);
	\draw[very thick] (2,1) to  (.5+\xshift,0);
	\draw[very thick] (-.5+\xshift,0) to  (-2,1);
	\draw[very thick] (3,6) to (2,6) to [out=160,in=0] (0,6.4) to [out=180,in=20] (-2,6) to (-3,6);
	\draw[very thick] (3+\xshift+\fx,5+\fy) to [out=180,in=340] (.5+\xshift,5) to (-.5+\xshift,5) to [out=200,in=0] (-3+\xshift+\fx,5+\fy);
	\draw[very thick] (2,6) to  (.5+\xshift,5);
	\draw[very thick] (-.5+\xshift,5) to  (-2,6);
\draw[very thick] (-3,6.04) to (-3,1-.04);
\draw[very thick] (3,6.04) to (3,1-.04);
\draw[very thick] (-3+\xshift+\fx,5+\fy+.04) to (-3+\xshift+\fx,\fy-.04);
\draw[very thick] (3+\xshift+\fx,5+\fy+.04)  to (3+\xshift+\fx,\fy-.04) ;
	\end{tikzpicture}
}

\newcommand{\Resa}[1]{
    \draw[very thick] (1+#1,0) to [out=180,in=315] (0.5+#1,.33) to (0+#1,.33);
	\draw[very thick] (1+#1,.5) to [out=180,in=45] (0.5+#1,0.33);
	\draw[very thick] (1+#1,1) to [out=180,in=0] (0+#1,0.66);
}
\newcommand{\Resb}[1]{
    \draw[very thick] (1+#1,0) to [out=180,in=0] (0+#1,.33);
	\draw[very thick] (1+#1,.5) to [out=180,in=315] (0.5+#1,0.63);
	\draw[very thick] (1+#1,1) to [out=180,in=45] (0.5+#1,.66) to (0+#1,0.66);
}
\newcommand{\Resc}[1]{
	\draw[very thick] (1+#1,.33) to [out=180,in=315] (0.5+#1,.5) to (0+#1,.5);
	\draw[very thick] (1+#1,.66) to [out=180,in=45] (0.5+#1,0.5);
}
\newcommand{\Inda}[1]{
    \draw[very thick] (0+#1,0) to [out=0,in=225] (0.5+#1,.33) to (1+#1,.33);
	\draw[very thick] (0+#1,.5) to [out=0,in=135] (0.5+#1,0.33);
	\draw[very thick] (0+#1,1) to [out=0,in=180] (1+#1,0.66);
}
\newcommand{\Indb}[1]{
    \draw[very thick] (0+#1,0) to [out=0,in=180] (1+#1,.33);
	\draw[very thick] (0+#1,.5) to [out=0,in=225] (0.5+#1,0.66);
	\draw[very thick] (0+#1,1) to [out=0,in=135] (0.5+#1,.66) to (1+#1,0.66);
}
\newcommand{\Indc}[1]{
	\draw[very thick] (0+#1,.33) to [out=0,in=225] (0.5+#1,.5) to (1+#1,.5);
	\draw[very thick] (0+#1,.66) to [out=0,in=135] (0.5+#1,0.5);
}

\theoremstyle{plain}
\newtheorem{thm}{Theorem}[section]
\newtheorem{prop}[thm]{Proposition}
\newtheorem{proposition}[thm]{Proposition}
\newtheorem{cor}[thm]{Corollary}

\newtheorem{lemma}[thm]{Lemma}

\newtheorem{conj}[thm]{Conjecture}

\theoremstyle{definition}
\newtheorem{defi}[thm]{Definition}
\newtheorem{definition}[thm]{Definition}
\newtheorem{construction}[thm]{Construction}
\newtheorem{rem}[thm]{Remark}
\newtheorem{remark}[thm]{Remark}

\newtheorem{exa}[thm]{Example}
\newtheorem{example}[thm]{Example}
\newtheorem{conv}[thm]{Convention}
\newtheorem*{exa-nono}{Example}

\newtheorem*{theoremintro}{Theorem}
\newtheorem*{remarkintro}{Remark}
\newtheorem*{definitionsintro}{Definitions}

\newcommand{\commented}[1]{}

%% file: aschobers.bbl
\newcommand{\etalchar}[1]{$^{#1}$}
\begin{thebibliography}{LMGR{\etalchar{+}}24}

\bibitem[AL17]{AL17}
Rina Anno and Timothy Logvinenko.
\newblock Spherical {DG}-functors.
\newblock {\em J. Eur. Math. Soc. (JEMS)}, 19(9):2577--2656, 2017.

\bibitem[AL21]{AL:skein}
Rina Anno and Timothy Logvinenko.
\newblock Skein-triangulated representations of generalised braids.
\newblock Online Talk: \url{https://www.youtube.com/watch?v=nTXwQ0JLb0o}, 2021.

\bibitem[B\"17]{Bae}
M.~B\"{a}renz.
\newblock {\em Topological state sum models in four dimensions, half-twists and
  their applications}.
\newblock Ph.{D}.~thesis, University of Nottingham, April 2017.

\bibitem[Bec18]{Bec18}
F.~Beckert.
\newblock {\em The bivariant parasimplicial ${S}_\bullet$-construction}.
\newblock Ph.{D}.~thesis, Bergische Universit{\"a}t Wuppertal, July 2018.

\bibitem[Bru24]{brundan2024qschurcategorypolynomialtilting}
Jonathan Brundan.
\newblock The $q$-schur category and polynomial tilting modules for quantum
  $gl_n$, 2024.
\newblock \arxiv{2407.07228}.

\bibitem[CDW23]{cdw:complexes}
Merlin Christ, Tobias Dyckerhoff, and Tashi Walde.
\newblock Complexes of stable $\infty$-categories, 2023.
\newblock \arxiv{2301.02606}.

\bibitem[CDW24]{cdw:lax}
Merlin Christ, Tobias Dyckerhoff, and Tashi Walde.
\newblock Lax additivity, 2024.
\newblock \arxiv{2402.12251}.

\bibitem[CKM14]{CKM}
S.~Cautis, J.~Kamnitzer, and S.~Morrison.
\newblock Webs and quantum skew {H}owe duality.
\newblock {\em Math. Ann.}, 360(1-2):351--390, 2014.

\bibitem[CR08]{MR2373155}
Joseph Chuang and Rapha{\"e}l Rouquier.
\newblock Derived equivalences for symmetric groups and {$\mathfrak
  {sl}_2$}-categorification.
\newblock {\em Ann. of Math. (2)}, 167(1):245--298, 2008.

\bibitem[Dem73]{MR342522}
Michel Demazure.
\newblock Invariants sym\'{e}triques entiers des groupes de {W}eyl et torsion.
\newblock {\em Invent. Math.}, 21:287--301, 1973.

\bibitem[DJW19]{djw:auslander}
Tobias Dyckerhoff, Gustavo Jasso, and Tashi Walde.
\newblock Simplicial structures in higher {A}uslander-{R}eiten theory.
\newblock {\em Adv. Math.}, 355:106762, 73, 2019.

\bibitem[DKSS24]{DKS:spherical}
Tobias Dyckerhoff, Mikhail Kapranov, Vadim Schechtman, and Yan Soibelman.
\newblock Spherical adjunctions of stable {$\infty $}-categories and the
  relative {S}-construction.
\newblock {\em Math. Z.}, 307(4):Paper No. 73, 59, 2024.

\bibitem[EKLP24]{elias2024demazureoperatorsdoublecosets}
Ben Elias, Hankyung Ko, Nicolas Libedinsky, and Leonardo Patimo.
\newblock Demazure operators for double cosets, 2024.
\newblock \arxiv{2307.15021}.

\bibitem[Enr10]{enriquez2010halfbalancedbraidedmonoidalcategories}
Benjamin Enriquez.
\newblock Half-balanced braided monoidal categories and teichmueller groupoids
  in genus zero, 2010.
\newblock \arxiv{1009.2652}.

\bibitem[ESW17]{ESW}
Ben Elias, Noah Snyder, and Geordie Williamson.
\newblock On cubes of {F}robenius extensions.
\newblock In {\em Representation theory---current trends and perspectives}, EMS
  Ser. Congr. Rep. Eur. Math. Soc., Z\"{u}rich, 2017.

\bibitem[GGM85]{GMP85}
A.~Galligo, M.~Granger, and P.~Maisonobe.
\newblock {$\mathcal{D}$}-modules et faisceaux pervers dont le support
  singulier est un croisement normal.
\newblock {\em Ann. Inst. Fourier (Grenoble)}, 35(1):1--48, 1985.

\bibitem[GM84]{gm:rebroussement}
M.~Granger and P.~Maisonobe.
\newblock Faisceaux pervers relativement \`a{} un point de rebroussement.
\newblock {\em C. R. Acad. Sci. Paris S\'er. I Math.}, 299(12):567--570, 1984.

\bibitem[HRW21]{HRW1}
Matthew {Hogancamp}, David~E.~V. {Rose}, and Paul {Wedrich}.
\newblock A skein relation for singular {S}oergel bimodules, 2021.
\newblock \arxiv{2107.08117}.

\bibitem[HRW24]{HRW2}
Matthew Hogancamp, David E.~V. Rose, and Paul Wedrich.
\newblock Link splitting deformation of colored {K}hovanov-{R}ozansky homology.
\newblock {\em Proc. Lond. Math. Soc. (3)}, 129(3):Paper No. e12620, 142, 2024.

\bibitem[Kho00]{Kho}
M.~Khovanov.
\newblock A categorification of the {J}ones polynomial.
\newblock {\em Duke Math. J.}, 101(3):359--426, 2000.

\bibitem[Kho04]{Kho3}
M.~Khovanov$\phantom{a}\!\!\!$.
\newblock $\mathfrak{sl}(3)$ link homology.
\newblock {\em Algebr. Geom. Topol.}, 4:1045--1081, 2004.

\bibitem[Kho07]{MR2339573}
Mikhail Khovanov.
\newblock Triply-graded link homology and {H}ochschild homology of {S}oergel
  bimodules.
\newblock {\em Internat. J. Math.}, 18(8):869--885, 2007.

\bibitem[KL10]{KL3}
M.~Khovanov and A.~Lauda.
\newblock A diagrammatic approach to categorification of quantum groups {III}.
\newblock {\em Quantum Topology}, 1:1--92, 2010.

\bibitem[KLMS12]{KLMS}
Mikhail Khovanov, Aaron Lauda, Marco Mackaay, and Marko {S}to{\v{s}i\'c}.
\newblock Extended graphical calculus for categorified quantum sl(2).
\newblock {\em Mem. Am. Math. Soc.}, 219, 2012.

\bibitem[KR08a]{KR}
Mikhail Khovanov and Lev Rozansky.
\newblock Matrix factorizations and link homology.
\newblock {\em Fund. Math.}, 199(1):1--91, 2008.

\bibitem[KR08b]{MR2421131}
Mikhail Khovanov and Lev Rozansky.
\newblock Matrix factorizations and link homology. {II}.
\newblock {\em Geom. Topol.}, 12(3):1387--1425, 2008.

\bibitem[KS14]{1411.2772}
Mikhail Kapranov and Vadim Schechtman.
\newblock Perverse schobers, 2014.
\newblock \arxiv{1411.2772}.

\bibitem[KS16]{KS:hyperplane}
M.~Kapranov and V.~Schechtman.
\newblock Perverse sheaves over real hyperplane arrangements.
\newblock {\em Ann. of Math. (2)}, 183(2):619--679, 2016.

\bibitem[KS22]{MR4367797}
Mikhail Kapranov and Vadim Schechtman.
\newblock Parabolic induction and perverse sheaves on {$W\backslash \germ h$}.
\newblock {\em Adv. Math.}, 398:Paper No. 108200, 52, 2022.

\bibitem[KS25]{2102.13321}
Mikhail Kapranov and Vadim Schechtman.
\newblock P{ROB}s and perverse sheaves {I}: symmetric products.
\newblock {\em Selecta Math. (N.S.)}, 31(2):Paper No. 23, 32, 2025.

\bibitem[LMGR{\etalchar{+}}24]{LMGRSW}
Yu~Leon Liu, Aaron Mazel-Gee, David Reutter, Catharina Stroppel, and Paul
  Wedrich.
\newblock A braided monoidal $(\infty,2)$-category of {S}oergel bimodules,
  2024.
\newblock \arxiv{2401.02956}.

\bibitem[LT21]{latifi2021minimalpresentationsglnwebcategories}
Genta Latifi and Daniel Tubbenhauer.
\newblock Minimal presentations of gln-web categories, 2021.
\newblock \arxiv{2112.12688}.

\bibitem[Lur09]{lurie:htt}
J.~Lurie.
\newblock {\em Higher topos theory}, volume 170 of {\em Annals of Mathematics
  Studies}.
\newblock Princeton University Press, Princeton, NJ, 2009.

\bibitem[Lur17]{lurie:ha}
Jacob Lurie.
\newblock Higher algebra, 2017.
\newblock Available at the author's website (version dated 9/18/2017).

\bibitem[Mor65]{MR190183}
Kiiti Morita.
\newblock Adjoint pairs of functors and {F}robenius extensions.
\newblock {\em Sci. Rep. Tokyo Kyoiku Daigaku Sect. A}, 9:40--71 (1965), 1965.

\bibitem[MSV11]{MR2746676}
Marco Mackaay, Marko Sto\v{s}i\'{c}, and Pedro Vaz.
\newblock The {$1,2$}-coloured {HOMFLY}-{PT} link homology.
\newblock {\em Trans. Amer. Math. Soc.}, 363(4):2091--2124, 2011.

\bibitem[MWW22]{2019arXiv190712194M}
Scott Morrison, Kevin Walker, and Paul Wedrich.
\newblock Invariants of 4-manifolds from {K}hovanov-{R}ozansky link homology.
\newblock {\em Geom. Topol.}, 26(8):3367--3420, 2022.

\bibitem[QR16]{QR}
Hoel Queffelec and David~E.V. Rose.
\newblock The $\mathfrak{sl}_n$ foam $2$-category: a combinatorial formulation
  of {K}hovanov--{R}ozansky homology via categorical skew {H}owe duality.
\newblock {\em Adv. Math.}, 302:1251--1339, 2016.

\bibitem[{Rou}04]{0409593}
Raphael {Rouquier}.
\newblock {Categorification of the braid groups}, 2004.
\newblock \arxiv{math/0409593}.

\bibitem[Rus25]{rush:thesis}
A.~Rush.
\newblock A homotopy-coherent calculus of lax matrices, 2025.
\newblock Dissertation, Hamburg,
  \url{https://www.math.uni-hamburg.de/en/forschung/bereiche/az/hoehere-strukturen/personen/rush-angus.html}.

\bibitem[RW24]{ren2024khovanov}
Qiuyu Ren and Michael Willis.
\newblock Khovanov homology and exotic $4$-manifolds, 2024.
\newblock \arxiv{2402.10452}.

\bibitem[Soe92]{Soergel}
Wolfgang Soergel.
\newblock The combinatorics of {H}arish-{C}handra bimodules.
\newblock {\em J. Reine Angew. Math.}, 429, 1992.

\bibitem[ST01]{ST01}
Paul Seidel and Richard Thomas.
\newblock Braid group actions on derived categories of coherent sheaves.
\newblock {\em Duke Math. J.}, 108(1):37--108, 2001.

\bibitem[ST09]{MR2579396}
Noah Snyder and Peter Tingley.
\newblock The half-twist for {$U_q({\germ g})$} representations.
\newblock {\em Algebra Number Theory}, 3(7):809--834, 2009.

\bibitem[Str23]{arxiv.2207.05139}
Catharina Stroppel.
\newblock Categorification: tangle invariants and {TQFT}s.
\newblock In {\em I{CM}---{I}nternational {C}ongress of {M}athematicians.
  {V}ol. {II}. {P}lenary lectures}, pages 1312--1353. EMS Press, Berlin, [2023]
  \copyright 2023.

\bibitem[SW24]{stroppel2024braidingtypesoergelbimodules}
Catharina Stroppel and Paul Wedrich.
\newblock Braiding on type {A} {S}oergel bimodules: semistrictness and
  naturality, 2024.
\newblock \arxiv{2412.20587}.

\bibitem[TVW17]{TVW}
Daniel Tubbenhauer, Pedro Vaz, and Paul Wedrich.
\newblock Super {$q$}-{H}owe duality and web categories.
\newblock {\em Algebr. Geom. Topol.}, 17(6):3703--3749, 2017.

\bibitem[Wed19]{Wed3}
Paul Wedrich.
\newblock Exponential growth of colored {HOMFLY}-{PT} homology.
\newblock {\em Adv. Math.}, 353:471--525, 2019.

\bibitem[{Wil}08]{Williamson-thesis}
Geordie {Williamson}.
\newblock {Singular {S}oergel bimodules}, 2008.
\newblock Dissertation, Freiburg,
  \url{http://people.mpim-bonn.mpg.de/geordie/GW-thesis.pdf}.

\bibitem[WW17]{MR3687104}
Ben Webster and Geordie Williamson.
\newblock A geometric construction of colored {HOMFLYPT} homology.
\newblock {\em Geom. Topol.}, 21(5):2557--2600, 2017.

\end{thebibliography}
